\documentclass[11pt]{amsart}

\usepackage{amsmath,amssymb,amsthm}
\usepackage{thmtools,enumerate}
\usepackage{mathtools}
\usepackage{array}
\usepackage{mathrsfs}

\usepackage[german,english]{babel}
\usepackage[autostyle]{csquotes}

\usepackage[all]{xy}
\usepackage{pstricks}

\usepackage{tikz-cd}
\usetikzlibrary{babel}

\usepackage{hyperref}
\hypersetup{
	colorlinks=true,
	linkcolor=red,
	citecolor=blue}

\theoremstyle{plain}

\newtheorem{thm}{Theorem}[section]
\newtheorem{prop}[thm]{Proposition}
\newtheorem{lem}[thm]{Lemma}
\newtheorem{cor}[thm]{Corollary}

\newtheorem{thmA}{Theorem}

\newtheorem*{thm*}{Theorem}

\theoremstyle{definition}

\newtheorem{dfn}[thm]{Definition}
\newtheorem{dfn-prop}[thm]{Definition-Proposition}

\newtheorem{notation}[thm]{Notation}
\newtheorem{setup}[thm]{Setup}

\newtheorem{setupI}{Setup}

\newcommand{\Z}{\mathbb{Z}}
\newcommand{\N}{\mathbb{N}}
\newcommand{\Q}{\mathbb{Q}}
\newcommand{\R}{\mathbb{R}}
\newcommand{\C}{\mathbb{C}}
\newcommand{\pr}{\mathbb{P}}

\newcommand{\cal}[1]{\mathcal{#1}}
\newcommand{\smvee}{{\scriptscriptstyle{\vee}}}
\newcommand{\smperp}{{\scriptscriptstyle{\perp}}}
\newcommand{\smone}{{\scriptscriptstyle{(1)}}}

\newcommand{\smonestar}{{\scriptscriptstyle{(1)*}}}
\newcommand{\smonestarstar}{{\scriptscriptstyle{(1)**}}}

\newcommand{\re}{\mathrm{re}}
\newcommand{\im}{\mathrm{im}}

\DeclareMathOperator{\Eff}{Eff}

\DeclareMathOperator{\Nef}{Nef}
\DeclareMathOperator{\Amp}{Amp}

\DeclareMathOperator{\Comp}{Comp}
\DeclareMathOperator{\Cr}{Cr}
\DeclareMathOperator{\GL}{GL}
\DeclareMathOperator{\Aut}{Aut}
\DeclareMathOperator{\Int}{Int}
\DeclareMathOperator{\Supp}{Supp}
\DeclareMathOperator{\rk}{rk}

\DeclareMathOperator{\id}{Id}
\DeclareMathOperator{\nod}{nod}

\DeclareMathOperator{\irr}{irr}
\DeclareMathOperator{\diag}{diag}

\begin{document}
	\title[Generalised Cone Conjecture, II: Generic Sakai surfaces]{Generalised Cone Conjecture, II:\\ Generic Sakai surfaces}
 
	\author[V.\ Lazi\'c]{Vladimir Lazi\'c}
	\address{Fachrichtung Mathematik, Campus, Geb\"aude E2.4, Universit\"at des Saarlandes, 66123 Saarbr\"ucken, Germany}
	\email{lazic@math.uni-sb.de}

        \author[I.\ Stenger]{Isabel Stenger}
        \address{Institute of Algebraic Geometry, Leibniz University Hannover, Welfengarten 1, 30167 Hannover, Germany}
        \email{stenger@math.uni-hannover.de}
		
	\author[Z.\ Xie]{Zhixin Xie}
	\address{Institut \'Elie Cartan de Lorraine, Universit\'e de Lorraine, 54506 Nancy, France}
	\email{zhixin.xie@univ-lorraine.fr}
	
	
	\begin{abstract}
		In this paper we prove that the Generalised Cone Conjecture holds for generic Sakai surfaces, which completes the proof of the Generalised Cone Conjecture in dimension two. Furthermore, we show that there exists a non-generic Sakai surface for which no meaningful variant of the conjecture holds. 
	\end{abstract}

	\maketitle
	
	%Table of Contents
	\begingroup
		\hypersetup{linkcolor=black}
		\setcounter{tocdepth}{1}
		\tableofcontents
	\endgroup
	
\section{Introduction}

This paper completes the proof of the following theorem from \cite{LSX26}.

\begin{thmA}\label{thm:main0}
Let $(X,B+M)$ be a very general klt Calabi--Yau generalised pair of dimension $2$. Then there exists a rational polyhedral fundamental domain for the action of the group of Cremona isometries $\Cr(X)$ on $\Nef^e(X)$.
\end{thmA}

This proves the Generalised Cone Conjecture from \cite{LSX26} in dimension two.

Recall from \cite{LSX26} that \emph{Sakai surfaces} are particular blowups of $\mathbb{P}^2$ at $9$ points (possibly infinitely near), see Definition \ref{dfn:goodHalphen}. Every Sakai surface has a unique divisor $D\in |{-}K_X|$ and the intersection matrix of $D$ determines a root system of affine type. \emph{Generic Sakai surfaces} are those which do not contain $(-2)$-curves, apart from the components of $D$. By \cite[Proposition B.4]{LSX26}, generic Sakai surfaces are very general Sakai surfaces in the sense of moduli theory. Note that for most types of Sakai surfaces, the nef cone is not rational polyhedral.

We briefly recall the strategy of the proof of Theorem \ref{thm:main0} from \cite{LSX26}. Let $(X,B+M)$ be a very general klt Calabi--Yau generalised pair of dimension two. It was proved in \cite[Theorem B]{LSX26} that either the Cone Conjecture holds on $X$, or $\kappa(X,{-}K_X)=0$ and there exists a birational map $X\dashrightarrow X_0$ to a generic Sakai surface $X_0$. Let $D_0\in|{-}K_{X_0}|$. Then, to complete the proof of Theorem \ref{thm:main0}, by \cite[Theorem B(c)]{LSX26} it suffices to show that there exists a rational polyhedral fundamental domain for the action of the group $\Cr(X_0,D_0)$ on $\Nef^e(X_0)$.
 
Exploiting the rich theory of Weyl groups and root systems, we prove the first main result of this paper. As explained above, Theorem \ref{thm:main2} completes the proof of Theorem \ref{thm:main0}.\footnote{For the sake of consistency with our paper \cite{LSX26}, we use the same enumeration of our main results as in op.\ cit.}

\setcounter{thmA}{2}

\begin{thmA}\label{thm:main2}
Let $X$ be a  generic Sakai surface, and let $D$ be the unique effective divisor in $|{-}K_X|$. Then there exists a rational polyhedral fundamental domain for the action of the group $\Cr(X,D)$ on $\Nef^e(X)$.
\end{thmA}

Here, $\Cr(X,D)$ is the subgroup of the group of Cremona isometries on $X$ which preserve the numerical classes of the components of $D$, see Definition \ref{dfn:cremona} for details.

One of the main points in the proof of Theorem \ref{thm:main2} is that the Tits cone of a root system has a rational polyhedral fundamental domain under the action of the Weyl group. The hard part is to show that the intersection of this fundamental domain with the nef cone of $X$ is also rational polyhedral, which we prove in Proposition \ref{prop:ratpoly}. The crucial ingredient in the proof of this proposition is that $\Cr(X,D)$ acts on the set of classes of $(-1)$-curves with finitely many orbits, which we show in Theorems \ref{thm:mainofsection7} and \ref{thm:mainthmsection8}. More details are given at the beginning of Part \ref{part:2}.

We cannot emphasise enough that our whole paper rests crucially on the techniques developed in Sakai's fundamental paper \cite{Sak01}.

In light of Theorem \ref{thm:main2}, it seems natural to expect a similar result for non-generic Sakai surfaces. Very surprisingly, however, all the results break down in the non-generic case: we construct an example of a  non-generic Sakai surface for which $\Cr(X,D)$ is trivial, but $\Nef(X)$ is not rational polyhedral. Thus, the assumption on \emph{generic} Sakai surfaces in Theorem \ref{thm:main2} is indeed necessary. More strongly, we prove that for \textbf{any} group $G$ satisfying necessary geometric assumptions, there cannot exist a rational polyhedral fundamental domain for the action of $G$ on $\Nef^e(X)$. This is the second main result of this paper. 

\begin{thmA}\label{thm:mainB}
There exists a non-generic Sakai surface $X$ such that the following holds. Let $G\subseteq\GL\big(N^1(X)_\R\big)$ be any group which preserves the intersection form on $N^1(X)_\R$ and the cone $\Eff(X)$. Then there does not exist a fundamental domain for the action of $G$ on $\Nef^e(X)$.
\end{thmA}

We will investigate the case of non-generic Sakai surfaces in detail in \cite{LSX26a}.

We refer to \cite{LSX26} for a detailed comparison of our results with previous results in the literature.

\subsection*{Declaration}

All the results in this paper were established by the three authors: mathematics is a human endeavour.

\subsection*{Acknowledgements} 

We are deeply indebted to Auguste H\'ebert for continuous help on the theory of root systems and Weyl groups, and for guiding us through the vast literature. We would like to thank C\'ecile Gachet and Wendelin Lutz for many useful discussions and comments over several years on the content of the present paper. We are grateful to Pierre-Emmanuel Chaput and Julia Schneider for many valuable discussions. This project started in 2023 when all three authors were based at the Universit\"at des Saarlandes.

Lazi\'c gratefully acknowledges support by the Deutsche Forschungsgemeinschaft (DFG, German Research Foundation) – Project-ID 286237555 – TRR 195 and Project-ID 530132094. The research of Stenger was partly conducted in the framework of the DFG-funded research training group RTG 2965: From Geometry to Numbers, Project number 512730679.
Xie was partly supported by PEPS JCJC de l'Insmi 2025.

\newpage

\part{}

\section{Preliminaries}\label{sec:prelim}

Throughout the paper, unless otherwise stated, we work with varieties that are normal, projective and defined over $\C$. Most of the notation we need in this paper was introduced in \cite{LSX26}; here we recall the most important properties that we need in this work.

If $X$ is a normal projective variety, we denote by $N^1(X)_\R$, $\Nef(X)$, $\Amp(X)$, $\Eff(X)$ and $\overline{\Eff}(X)$ the N\'eron-Severi space and the nef, ample, effective and pseudoeffective cone of $X$, respectively. If $D$ is an $\R$-Cartier $\R$-divisor on $X$, then we denote its numerical class in $N^1(X)_\R$ by $[D]$. We call the cone
$$\Nef^e(X)\coloneqq\Nef(X)\cap\Eff(X)$$
the \emph{effective nef cone}. The \emph{rational nef cone} $\Nef^+(X)$ is the cone in $N^1(X)_\R$ spanned by the classes of nef line bundles. The cones $\Nef^e(X)$ and $\Nef^+(X)$ are the most important cones related to the (Generalised) Cone Conjecture.

\subsection{Curves of negative self-intersection on rational surfaces}

\begin{dfn}
	We call an irreducible curve $C$ on a smooth surface $X$ a \emph{$({-}1)$-curve} if $C^2 = C\cdot K_X = {-}1$, and we call it a \emph{$({-2})$-curve} if $C^2 = {-}2$ and $C\cdot K_X= 0$.
\end{dfn}

Note that, in the context of the previous definition, when ${-}K_X$ is additionally nef, then an irreducible curve $C$ on $X$ is a $({-}2)$-curve if and only if $C^2={-}2$ by \cite[Lemma 2.5(b)]{LSX26}.

We fix some notation that we need throughout the paper.

\begin{notation}\label{not:MXNX}
	Let $X$ be a smooth rational surface with $K_X^2\geq0$, and let $D\in|{-}K_X|$. We define:
	\begin{align*}
		\Comp(D) &\coloneqq \big\{ [C]\in N^1(X) \mid C \text{ is a component of } D\big\},\\
		\cal{M}_X &\coloneqq  \big\{ [E]\in N^1(X) \mid E^2 = K_X\cdot E = -1 \text{ and }h^0(X,E)\geq0 \big\},\\
		\cal{M}_X^{\irr} &\coloneqq \big\{ [E]\in N^1(X) \mid E \text{ is a }(-1)\text{-curve} \big\}, \\
		\mathcal{N}_X &\coloneqq \big\{ E\in \mathcal M_X \mid E\cdot C \geq 0 \text{ for all } C \in \Comp(D) \big\},\\
		\Delta_{X,D}^{\nod}&\coloneqq\big\{[C] \in N^1(X)\mid C\text{ is a $(-2)$-curve with}\ C\cap\Supp D=\emptyset\big\}.
	\end{align*}
The set $\Delta_{X,D}^{\nod}$ is well defined by \cite[Lemma 2.4]{LSX26}.
\end{notation}

We need the following characterisation of the sets $\mathcal M_X$ and $\mathcal N_X$.

\begin{lem}\label{lem:-1classesHalphen}
	Let $X$ be a smooth rational surface such that $K_X^2\geq0$, and let $D\in|{-}K_X|$. Then
	\begin{enumerate}[\normalfont (a)]
		\item $ \mathcal{M}_X = \big\{ E\in N^1(X) \mid E^2 = K_X\cdot E = -1 \big\} $,
		\item $ \mathcal{N}_X = \big\{ E\in N^1(X) \mid E^2 = K_X\cdot E = -1 \text{ and }E\cdot C \geq 0 \text{ for all } C \in \Comp(D)\big\} $.
	\end{enumerate}
\end{lem}

\begin{proof}
	Part (a) is \cite[Lemma 3.2]{LH05}, and part (b) follows from the definition of $\mathcal N_X$ and from (a).
\end{proof}

\subsection{Sakai surfaces}
In this paper, we are mainly interested in a special class of smooth rational surfaces that we introduced as \emph{Sakai surfaces} in \cite[Definition 2.8]{LSX26}. We recall their definition in this subsection.

Let $D = \sum n_iD_i$ be an effective integral divisor on a smooth projective surface $X$, where $D_i$ are prime components. Then $D$ is \emph{of canonical type} if $K_X\cdot D_i = D\cdot D_i = 0$ for all $i$. Recall the following definition from \cite{Sak01}.

\begin{dfn}\label{def:gen_halphen_surface}
		A smooth rational surface $X$ is a \emph{generalised Halphen surface} if the linear system $|{-}K_X|$ contains a divisor of canonical type. If, additionally, we have $h^0(X,-K_X)=1$, then $X$ is called an \emph{$R$-surface}, where $R$ is the type of the unique divisor $D \in |{-}K_X|$.
\end{dfn} 

If $X$ is an $R$-surface, then the intersection matrix of $D \in |{-}K_X|$ determines a root system $R$ of affine type: such surfaces are then classified according to the \emph{type $R$} of $D$, and we also say that $X$ is \emph{of type $R$}.
We refer to Appendix \ref{sec:Sakai_appendix_A}, where we recall the possible types of $R$-surfaces from \cite[Appendix A]{Sak01}.

\begin{dfn}\label{dfn:goodHalphen}
		A smooth projective surface $X$ is a \emph{Sakai surface} if it is an $R$-surface with $\kappa(X,-K_X) = 0$. We say that $X$ is \emph{generic} if $\Delta^{\nod}_{X,D} = \emptyset$, where $D\in|{-}K_X|$. 
	\end{dfn}

A fundamental fact is that a Sakai surface $X$ is always the blowup of $\mathbb P^2$ at nine (possibly infinitely near) points, that $-K_X$ is nef, and that $K_X^2 = 0$ \cite[Lemma 2.9(a)(b)]{LSX26}.

The following lemma gives an easy criterion for determining whether an element of $ \mathcal{M}_X$ is the class of a $(-1)$-curve.

\begin{lem}\label{lem:irredclasses}
	Let $X$ be a Sakai surface and let $D\in|{-}K_X|$. Then 
	$$ \mathcal{M}_X^{\irr} = \big\{ E \in \mathcal M_X \mid E\cdot C \geq 0 \text{ for all } C \in \Comp(D) \cup \Delta^{\nod}_{X,D} \big\}. $$
	In particular, $\mathcal{M}_X^{\irr}\subseteq\mathcal{N}_X$.
\end{lem}

\begin{proof}
	By assumption, $D$ is the only element of the linear system $|{-}K_X|$. Hence, by \cite[Proposition 3.3]{LH05} we have that
	$$ \big\{ E \in \mathcal M_X \mid E\cdot C \geq 0 \text{ for all } C \in \Comp(D) \cup \Delta^{\nod}_{X,D} \big\}\subseteq \mathcal{M}_X^{\irr} . $$
	For the opposite inclusion, let $E\in\mathcal{M}_X^{\irr}$. Then $E\cdot D={-}E\cdot [K_X]=1$, and therefore $E\notin\Comp(D)$ since $D$ is a divisor of canonical type by the definition of Sakai surfaces. Since clearly $E\notin \Delta^{\nod}_{X,D}$, we conclude that $E\cdot C \geq 0$ for all $C \in \Comp(D) \cup \Delta^{\nod}_{X,D}$, as desired.
\end{proof}

The following result gives a precise bound on the number of $(-2)$-curves on a Sakai surface.

\begin{lem}\label{lem:lin_indep_-2curves}
	Let $X$ be a Sakai surface. Then the classes of $(-2)$-curves in $N^1(X)_{\R}$ are linearly independent. In particular, $X$ has at most nine $(-2)$-curves. 
\end{lem}

\begin{proof}
	Suppose that the classes of $(-2)$-curves are linearly dependent. Then by \cite[Proposition 2.4]{Lah04}, there exists an elliptic fibration $f\colon X\to \pr^1$. By \cite[Corollary 5.6.1 on p.\ 347]{CD89} there exists a positive integer $m$ such that $f$ is defined by the linear system $|{-}mK_X|$. Hence, $\kappa(X,-K_X) =1$, which is a contradiction since $X$ is a Sakai surface. 
    
	For the second statement, note that every $(-2)$-curve $C$ satisfies $K_X\cdot C = 0$. Since the Picard rank of $X$ is $10$, and since the classes of $(-2)$-curves are linearly independent by the first paragraph, we conclude that there exist at most nine distinct classes of $(-2)$-curves. Finally, by \cite[Lemma 2.4]{LSX26}, any two $(-2)$-curves with the same numerical class coincide. This proves the last statement of the lemma.
\end{proof}

\subsection{Cremona isometries}
In this subsection, we recall the definition of Cremona isometries from \cite[Definition 2.16]{LSX26}. Since we use this notion only for Sakai surfaces in this paper, we restrict to the case of smooth surfaces, see also \cite{Loo81,Dol83}.

\begin{dfn}\label{dfn:cremona}
	Let $X$ be a smooth projective surface. An automorphism $\sigma$ of the vector space $N^1(X)_\R$ is a \emph{Cremona isometry} of $X$ if it satisfies the following properties:
	\begin{enumerate}[\normalfont (i)]
		\item $\sigma$ maps the lattice $N^1(X)$ onto itself,
		\item $\sigma$ preserves the intersection form on $N^1(X)_\R$,
		\item $\sigma$ fixes the class $[K_X]$,
		\item $\sigma$ preserves the effective cone $\Eff(X)$.
	\end{enumerate}
	We denote the group of Cremona isometries of $X$ by $\Cr(X)$. Additionally, if $D$ is an effective divisor on $X$ with irreducible components $D_1,\dots,D_m$, we define $\Cr(X,D)$ as
	\[\Cr(X,D)\coloneqq\big\{\sigma\in\Cr(X)\mid\sigma\big([D_i]\big)=[D_i]\text{ for }i=1,\dots,m\big\}.\]
\end{dfn}

\subsection{Group actions and fundamental domains}

Let $\Lambda$ be a free $\Z$-module of finite rank and let $V\coloneqq \Lambda\otimes_\Z \R$. A \emph{convex cone} in $V$ is a convex subset $\mathcal C\subseteq V$ such that $\lambda v\in\mathcal C$ for every $\lambda\geq0$ and $v\in\mathcal C$. 
It is \emph{strictly convex} if its closure $\overline{\mathcal C}$ contains no lines. 
It is \emph{polyhedral}, respectively \emph{rational polyhedral}, if it is spanned by finitely many elements of $V$, respectively of $\Lambda$.
For a convex cone $\mathcal C\subseteq V$ we denote by $\mathcal \Int(C)$ the interior of $\mathcal C$. We say that a subgroup of $G\subseteq\GL(V)$ \emph{preserves $\mathcal C$} if $g(\mathcal C)=\mathcal C$ for each $g\in G$; then we also say that \emph{$G$ acts on $\mathcal C$}, and that we have an \emph{action of $G$ on $\mathcal C$}. Moreover, if $\Pi\subseteq\mathcal C$, then we denote $G\cdot\Pi\coloneqq\bigcup_{g\in G}g(\Pi)$.

\begin{dfn}
	Let $\mathcal C\subseteq V$ be a strictly convex cone such that $\Int(\mathcal C)\neq\emptyset$, and let $G$ be a subgroup of $\GL(V)$ which preserves $\Lambda$ and $\mathcal C$. 
	We say that the action of $G$ on $\mathcal C$ admits a \emph{fundamental domain} if there exists a convex cone $\Pi \subseteq \mathcal C$ such that $G \cdot \Pi = \mathcal C$, and if $g\big( \Int(\Pi) \big) \cap \Int(\Pi) \neq\emptyset$ for some $g \in G$, then $g$ is the identity in $G$.
	We say that the action of $G$ on $\mathcal C$ admits a \emph{rational polyhedral fundamental domain} if $\Pi$ can be chosen to be rational polyhedral.
\end{dfn}

\section{Root data}\label{sec:rootsystems}

In this paper we use crucially the extended theory of root systems and Weyl groups. The classical theory of root systems as in \cite{Kac90,Bou08} is, however, not sufficient for the purposes of this work. We need the concept of \emph{root data} from \cite{MP89}, which allows simple roots to be linearly dependent.

We start by recalling definitions and results on root systems and Weyl groups from \cite{Kac90}. In our setup, we need many classical results on finite and affine root systems. Afterwards, we present the notion of sets of root data from \cite{MP89}, which generalises the concept of root systems. This gives considerable flexibility when dealing with \emph{subroot systems} on Sakai surfaces.

Throughout the section, if $V$ is a finite-dimensional real vector space, then we denote by $V^*$ its dual, and we consider the pairing
$$\langle\cdot\,,\cdot\rangle\colon V\times V^*\to\R,\quad \langle v,w\rangle\coloneqq w(v).$$

\subsection{Root systems and Weyl groups}\label{subsec:roots}

    A real $(\ell+1)\times (\ell+1)$-matrix 
    $$A=(a_{i,j})_{0\leq i,j\leq \ell}$$
    is a \emph{generalised Cartan matrix} if
    \begin{enumerate}[(i)]
        \item $a_{i,i}=2$ for any $i$;
        \item $a_{i,j}\in \Z_{\leq 0}$ for distinct $i$ and $j$;
        \item $a_{i,j}=0$ if and only if $a_{j,i}=0$ for distinct $i$ and $j$.
    \end{enumerate}
    The matrix is \emph{indecomposable} if it is not a direct sum of matrices. A \emph{realisation of $A$} is a triple $(V,\Pi,\Pi^\smvee)$, where $V$ is a real vector space\footnote{We use different notation from that in \cite{Kac90}: the roles of vector spaces $V$ and $V^*$ are reversed in op.\ cit. We use this notation to make it compatible with other literature, and in particular with the paper \cite{MP89}, which is crucial for this paper.} and:
    \begin{enumerate}[(i)]
        \item the sets $\Pi=\{\alpha_0,\dots,\alpha_\ell\}\subseteq V$ and $\Pi^\smvee=\{\alpha_0^\smvee, \dots,\alpha_\ell^\smvee\}\subseteq V^*$ are linearly independent,
        \item $\langle \alpha_i,\alpha_j^\smvee\rangle=a_{ij}$ for all $i$ and $j$,
        \item $\dim V=2\ell+2-\rk (A)$.
    \end{enumerate}
    We call $\Pi$ a \emph{root basis} and the elements of $\Pi$ \emph{simple roots}. The lattice 
    $$Q\coloneqq \sum_{i=0}^\ell\Z\alpha_i \subseteq V$$
    is the \emph{root lattice}, and we denote
    $$Q_+\coloneqq \sum_{i=0}^\ell\Z_+\alpha_i \subseteq V.$$
    
    We can associate to $A$ a Kac--Moody algebra $\mathfrak{g}(A)$ with a Lie bracket $[\cdot,\cdot]$ as in \cite[\S1.3]{Kac90}, and for each $\alpha\in Q$ we consider the vector space
    $$\mathfrak{g}_\alpha\coloneqq \big\{x\in\mathfrak{g}(A)\mid [h,x]=\alpha(h)x\text{ for all }h\in V^*\big\};$$
    here we identify $V$ with $V^{**}$. We say that $\alpha\in Q \backslash \{ 0 \} $ is a \emph{root} if $\dim\mathfrak{g}_\alpha>0$, and we denote the set of all roots by $\Delta$: we call it the \emph{root system of $\mathfrak{g}(A)$}. It follows from \cite[Proposition 1.5]{Kac90} that for each root $\alpha$ we have $\alpha\in Q_+$, in which case we call it a \emph{positive root}, or $-\alpha\in Q_+$, in which case we call it a \emph{negative root}. We denote the set of all positive roots by $\Delta_+$, and the set of all negative roots by $\Delta_-$. It follows that we have a disjoint union 
    $$\Delta=\Delta_+\sqcup\Delta_-\subseteq Q.$$
    It follows by construction that 
    $$\Pi\subseteq\Delta_+\subseteq Q_+,$$
    and we have $\Delta_-={-}\Delta_+$ by \cite[\S1.3]{Kac90}.
    
	For any $\alpha_i\in \Pi$ the \emph{fundamental reflection} associated to $\alpha_i$ is the linear map
    $$ r_i\colon V\to V,\quad x\mapsto x-\langle x,\alpha_i^\smvee\rangle \alpha_i. $$
    The map $r_i$ leaves the hyperplane $\ker\alpha_i^\smvee\subseteq V$ pointwise fixed and satisfies $r_i(\alpha_i)={-}\alpha_i$; in particular, it is a linear automorphism of $V$. The subgroup 
    $$W\coloneqq \big\langle r_i\mid\alpha_i\in \Pi\big\rangle\subseteq\GL(V)$$
    is the \emph{Weyl group of $\mathfrak{g}(A)$}. The set $\Delta$ is $W$-invariant by \cite[Proposition 3.7(b)]{Kac90}. We define similarly the \emph{dual fundamental reflection} $r_i^{\smvee}$ and the \emph{dual Weyl group $W^\smvee$}.
    
    The set
    $$\Delta^\re\coloneqq \big\{w(\alpha)\in V\mid \alpha\in \Pi,\ w\in W\big\}\subseteq\Delta$$
    is the \emph{set of real roots}, and the set
    $$\Delta^\re_+\coloneqq \Delta^\re\cap Q_+$$
    is the \emph{set of positive real roots}. The set
    $$\Delta^\im\coloneqq \Delta\setminus\Delta^\re$$
    is the \emph{set of imaginary roots}. The set of \emph{real coroots $(\Delta^\smvee)^\re$} is defined similarly.

Note that there is a unique $W$-equivariant bijection $(\cdot)^\smvee\colon \Delta^\re\to (\Delta^\smvee)^\re$ such that
$$ (\alpha_i)^\smvee=\alpha_i^\smvee\quad\text{for all } i; $$
we will denote simply $\alpha^\smvee\coloneqq (\alpha)^\smvee$. Moreover, we have a bijection $(\cdot)^\smvee\colon \Delta^\re_+\to (\Delta^\smvee)^\re_+$ and
$$ \langle\alpha,\alpha^\smvee\rangle = 2\quad\text{for all } \alpha\in\Delta^\re. $$

   We note that for any $\alpha\in \Delta^\re$, the \emph{reflection}
$$ r_\alpha\colon V \to V, \quad x-\langle x,\alpha^\smvee\rangle \alpha $$
belongs to $W$ by \cite[\S5.1]{Kac90}, and then clearly 
$$ W=\big\langle r_\alpha\mid\alpha\in \Delta^\re\big\rangle. $$
Thus, the Weyl group $W$ depends only on the set $\Delta^\re$ and not on the root basis $\Pi$.

    The triple $(V^*,\Pi^\smvee,\Pi)$ is a realisation of the generalised Cartan matrix $A^T$, and it is called \emph{dual} to $(V,\Pi,\Pi^\smvee)$. We define the action of the group $W$ on $V$ by
$$ \big\langle w(x), y\big\rangle \coloneqq  \big\langle x, w^{-1}(y)\big\rangle \quad \text{for any }w\in W,\ x\in V,\ y\in V^*. $$

\subsection{Standard bilinear forms}\label{subsec:standardbilinear}

We recall the definitions of standard bilinear forms from \cite[Chapter 2]{Kac90}. Let $A=(a_{i,j})_{0\leq i,j\leq \ell}$ be a symmetric generalised Cartan matrix and let $(V,\Pi,\Pi^\smvee)$ be a realisation of $A$. Fix positive rational numbers $\varepsilon_0,\dots,\varepsilon_\ell$ such that there exists a symmetric matrix $B$ with
$$A=\diag(\varepsilon_0,\dots,\varepsilon_\ell)\cdot B.$$
Note that if $A=A_1\oplus A_2$, where $A_1$ and $A_2$ are indecomposable matrices of sizes $\ell_1\times\ell_1$ and $\ell_2\times\ell_2$, then we must have $\varepsilon_0=\ldots=\varepsilon_{\ell_1-1}$ and $\varepsilon_{\ell_1}=\ldots=\varepsilon_\ell$.

Let $V_1$ be the subspace of $V^*$ spanned by $\Pi^\smvee$, and let $V_2$ be any complementary subspace to $V_1$ in $V^*$. We define a symmetric bilinear $\R$-valued form $(\cdot\, ,\cdot)$ on $V^*$ by setting
$$ (\alpha_i^\smvee,v) \coloneqq  \varepsilon_i\langle \alpha_i,v\rangle\quad\text{for }v\in V^*\text{ and all } i,$$
and by $(v_1,v_2)=0$ for $v_1,v_2\in V_2$. This form is non-degenerate, hence we have an induced bilinear form on $V$, which we also denote by $(\cdot\, ,\cdot)$. This bilinear form is $W$-invariant by \cite[Proposition 3.9]{Kac90}, and we call it a \emph{standard invariant bilinear form} associated to the numbers $\varepsilon_1,\dots,\varepsilon_\ell$. For every $\alpha\in V^*$, we denote $|\alpha|^2\coloneqq (\alpha,\alpha)$.

With these definitions, we have
$$(\alpha_i,\alpha_i)>0\text{ for all }i,\quad (\alpha_i,\alpha_j)\leq0\text{ for all }i\neq j,$$
and
$$a_{ij}=\frac{2(\alpha_i,\alpha_j)}{(\alpha_i,\alpha_i)}\quad\text{for all }i\text{ and }j.$$
For future reference we note that we have an isomorphism $\nu\colon V^*\to V$ defined by
$$\langle\nu(v),v'\rangle = (v,v')\quad \text{for }v,v'\in V^*.$$
Then
$$\nu(\alpha_i^\smvee)=\varepsilon_i\alpha_i\quad\text{for all } i.$$

\subsection{Finite and affine root systems}

We recall some special types of generalised Cartan matrix defined in \cite[Chapter 4]{Kac90}. Since we always work with symmetric pairing $\langle\cdot\,,\cdot\rangle\colon V\times V^*\to\R$, we only consider symmetric generalised Cartan matrices and use the following as definition. 

\begin{dfn}
    Let $A$ be an indecomposable symmetric generalised Cartan matrix, and let $(V,\Pi,\Pi^\smvee)$ be its realisation with the corresponding root system $\Delta$ and Weyl group $W$. Then $A$ (or $\Delta$ or $W$) is of \emph{finite type} if it is positive definite, and it is of \emph{affine type} is it is positive semidefinite of corank $1$.
\end{dfn}

Generalised Cartan matrices of finite and affine type are determined by their Dynkin diagram; a complete classification can be found in the tables in \cite[pp.\ 43-45]{Kac90}. We have the following properties by \cite[Propositions 4.7(e) and Theorems 4.8(c) and 5.6]{Kac90}.

\begin{thm}\label{thm:imaginary_roots}
    Let $A$ be an indecomposable symmetric generalised Cartan matrix with a realisation $(V,\Pi,\Pi^\smvee)$, where $\Pi = \{\alpha_0,\dots,\alpha_{\ell} \}$.
    \begin{enumerate}[\normalfont (a)]
    \item If $A$ is of finite type, then $\Delta^\re$ is finite and $\Delta^\im=\emptyset$.
    \item Assume that $A$ is of affine type, and set
    $$ \delta = \sum_{i=0}^{\ell} a_i \alpha_i \in V,$$
    where $a_0,\dots,a_{\ell}$ are the labels in \cite[Tables Aff 1-3]{Kac90}. Then $A\delta =0 $ and
    $$ \Delta^\im_+ = \{n\delta\mid n\in \Z_{>0}\}.$$
    \end{enumerate}
\end{thm}

\subsection{Squared lengths of roots}

    Let $A$ be an indecomposable symmetric generalised Cartan matrix, and let $(V,\Pi,\Pi^\smvee)$ be its realisation with the corresponding root system $\Delta$ and Weyl group $W$. Let $|\cdot|$ be the norm induced by a standard bilinear form on $V$, and set $a_{\min}\coloneqq \min\limits_{\alpha\in\Pi}|\alpha|^2$. Then for any $\alpha\in\Delta^\re$ we have $|\alpha|^2=a_{\min}$; in other words, \emph{short} and \emph{long} roots in \cite[p.\ 51]{Kac90} coincide in this situation. The following result follows from \cite[Proposition 5.10]{Kac90}, and gives a characterisation of real and imaginary roots for root systems of finite or affine type.

\begin{lem}\label{lem:lenthsofroots}
    With notation as above, assume additionally that $A$ is of finite or affine type. Then
    $$ \Delta^\re=\big\{\alpha\in Q\mid |\alpha|^2=a_{\min}\big\}\quad\text{and}\quad\Delta^\im=\big\{\alpha\in Q\setminus\{0\}\mid |\alpha|^2\leq0\big\}. $$
\end{lem}

\subsection{Removing a root}\label{subsec:removingroot}

We briefly recall the following construction from \cite[Chapter 6]{Kac90}. We only consider matrices of type Aff 1 from \cite[Table Aff 1 on p.\ 44]{Kac90}.

Let $A$ be a symmetric generalised Cartan matrix of affine type with realisation $(V,\Pi,\Pi^\smvee)$, where $\Pi = \{\alpha_0,\dots, \alpha_\ell\}$, with the corresponding root system $\Delta$ and the Weyl group $W$. We label the elements of $\Pi$ as in \cite[\S6.1]{Kac90}, so that $\alpha_0$ corresponds the leftmost vertex in the Dynkin diagram of $A$.

Let $A^\circ$ be the symmetric generalised Cartan matrix obtained by deleting the first column and the first row of $A$. Set $\Pi^\circ \coloneqq \{ \alpha_1,\dots, \alpha_\ell \}$, and let $V^\circ$ be the subspace of $V$ generated by $\alpha_1,\dots,\alpha_\ell$. Then $A^\circ$ is a generalised Cartan matrix of finite type whose realisation is $\big(V^\circ,\Pi^\circ, (\Pi^\circ)^\smvee\big)$. Denote by $\Delta^\circ$ the corresponding root system, by $Q^\circ$ the corresponding root lattice and by $W^\circ$ the corresponding Weyl group. Let $\delta \in V$ be as in Theorem \ref{thm:imaginary_roots}(b). 

The following result is contained in \cite[Proposition 6.3(a)(e) and Proposition 4.9]{Kac90}.

\begin{prop}\label{prop:rootsafterremoving}
With notation as above, we have
$$\Delta^{\re} = \{ \alpha + n\delta \mid \alpha \in \Delta^\circ, n\in \Z\}$$
and
$$\Delta^{\re}_+ = \{ \alpha + n\delta \mid \alpha \in \Delta^\circ, n\in \Z_{>0}\} \cup \Delta^\circ_+.$$
Moreover, the set $\Delta^\circ$ is finite.
\end{prop}

\subsection{Normalised standard bilinear form}\label{subsec:normalised}

Let $A=(a_{ij})_{0\leq i,j\leq\ell}$ be a symmetric generalised Cartan matrix of affine type with realisation $(V,\Pi,\Pi^\smvee)$. We only consider matrices of type Aff 1 from \cite[Table Aff 1 on p.\ 44]{Kac90}.

Then, in the notation from \cite[\S6.1]{Kac90}, we have $a_i^\smvee=a_i$ for $i=0,\dots,\ell$, as well as $a_0=1$. Following \cite[\S6.2]{Kac90}, we call the element
$$c\coloneqq \sum_{i=0}^\ell a_i\alpha_i^\smvee\in V^*$$
the \emph{canonical central element}. We fix $d\in V^*$ satisfying $\langle \alpha_0,d\rangle = 1$ and $\langle \alpha_i,d\rangle = 0$ for $i>0$; this element, called the \emph{scaling element}, is defined up to a summand proportional to $c$, which will be useful for us later. The elements $\alpha_0^\smvee,\dots,\alpha_\ell^\smvee, d$ form a basis of $V^*$. Define $\Lambda_0\in V$ by the conditions $\langle\Lambda_0,\alpha_0^\smvee\rangle=1$, $\langle\Lambda_0,\alpha_i^\smvee\rangle=0$ for $i > 0$ and $\langle\Lambda_0,d\rangle=0$. The elements $\alpha_0,\dots,\alpha_\ell, \Lambda_0$ form a basis of $V$. Then as in \cite[\S6.2]{Kac90} we define the \emph{normalised standard bilinear form} $(\cdot\, ,\cdot)$ on $V$ by
\begin{align*}
(\alpha_i,\alpha_j)&=a_{ij}\quad\text{for all }i\text{ and }j,\\
(\alpha_i,\Lambda_0)&=0\quad\text{for all }i>0,\\
(\alpha_0,\Lambda_0)&=1\quad\text{and}\quad (\Lambda_0,\Lambda_0)=0.
\end{align*}
With the definition of the isomorphism $\nu\colon V^*\to V$ from \S\ref{subsec:standardbilinear}, we then have
\begin{equation}\label{eq:cdelta1}
\nu(c)=\delta\quad\text{and}\quad \nu(\alpha_i^\smvee)=\alpha_i\text{ for all }i.
\end{equation}

Now, following \cite[\S6.5]{Kac90} and the notation from \S\ref{subsec:removingroot}, for each $\alpha \in Q^\circ$ we consider its associated \emph{translation} $t_\alpha\colon V\to V$ given by the formula
	    \begin{equation}\label{eq:117}
	    \textstyle t_\alpha(x) \coloneqq x + \langle x, c\rangle \alpha - \big((x,\alpha)+ \frac{1}{2} (\alpha,\alpha)\langle x,c\rangle\big)\delta,
	    \end{equation}
and the \emph{group of translations}
	    $$ T \coloneqq \{ t_\alpha \mid \alpha \in Q^\circ\}. $$
Then we have the following proposition \cite[Proposition 6.5]{Kac90}.

\begin{prop}\label{prop:translations}
With notation as above, we have
$$ W\simeq W^\circ\ltimes T. $$
\end{prop}

\subsection{Root data}

In this paper we need the concept of \emph{root data} from \cite{MP89}. We adapt the definition from \cite[Section 2]{MP89} slightly, since in our setting all index sets are finite, we only consider real vector spaces, and the pairing is always the natural pairing between a vector space and its dual.

\begin{dfn}\label{dfn:root_data}
Let $J$ be a finite index set. A set of \emph{root data} over $\R$ is a $6$-tuple
$$ \big(A,\Pi,\Pi^\smvee,V,V^*,\langle\cdot\, ,\cdot\rangle\big), $$
where $A=(a_{ij})_{i,j\in J}$ is a generalised Cartan matrix, $V$ is an $\R$-vector space together with the natural pairing $\langle\cdot\, ,\cdot\rangle\colon V \times V^* \to\R$, and:
\begin{enumerate}[\normalfont (i)]
\item the subsets $\Pi=\{\alpha_j\mid j\in J\}\subseteq V$ and $\Pi^\smvee=\{\alpha_j^\smvee\mid j\in J\}\subseteq V^*$ satisfy $\langle\alpha_i,\alpha_j^\smvee\rangle = a_{ij}$ for all $i,j\in J$, and
\item the sets $Q \coloneqq  \sum_{j\in J}\Z\alpha_j\subseteq V$ and $Q^\smvee \coloneqq  \sum_{j\in J}\Z\alpha_j^\smvee\subseteq V^*$ are free abelian groups,
\item there exist a finite index set $I$ and $\R$-linearly independent subsets $\Gamma\coloneqq \{\gamma_i\mid i\in I\}\subseteq V$ and $\Gamma^\smvee\coloneqq \{\gamma_i^\smvee\mid i\in I\}\subseteq V^*$, such that $\Gamma$ is a basis of $Q$, $\Gamma^\smvee$ is a basis of $Q^\smvee$, and
$$ \Pi\subseteq \bigoplus_{i\in I}\N\gamma_i\quad\text{and}\quad \Pi^\smvee\subseteq \bigoplus_{i\in I}\N\gamma_i^\smvee.$$
\end{enumerate}
\end{dfn}

Note that if a triple $(V,\Pi,\Pi^\smvee)$ is a realisation of a generalised Cartan matrix $A$ as in \S\ref{subsec:roots}, then the $6$-tuple $\big(A,\Pi,\Pi^\smvee,V,V^*,\langle\cdot\, ,\cdot\rangle\big)$ is a set of root data over $\R$. This is our basic example of root data.

We will need the following lemma several times in Appendix \ref{sec:appendix_cyclic_subgroup}.

\begin{lem}\label{lem:simplelemma}
Let $\mathscr D\coloneqq \big(A,\Pi,\Pi^\smvee,V,V^*,\langle\cdot\, ,\cdot\rangle\big)$ be a set of root data. Let $\alpha,\beta,\gamma\in V^*$. Consider the sets
$$ S\coloneqq \big\{x\in V\mid \langle x,\alpha\rangle\leq0,\ \langle x,\beta \rangle\leq0\big\} \cup \big\{x\in V\mid \langle x,\alpha\rangle\geq0,\ \langle x,\gamma \rangle\leq0\big\} $$
and
$$ P\coloneqq \big\{x\in V\mid \langle x,\beta\rangle\leq0,\ \langle x,\gamma \rangle\leq0\big\}. $$
Then:
\begin{enumerate}[\normalfont (a)]
\item $P\subseteq S$,
\item if $\gamma=p\alpha+q\beta$ for $p,q\in\R_{>0}$, then $P=S$.
\end{enumerate}
\end{lem}

\begin{proof}
Let $x\in P$. Then either $\langle x, \alpha \rangle \leq 0$, in which case $x\in \big\{x\in V\mid \langle x,\alpha\rangle\leq0,\ \langle x,\beta \rangle\leq0\big\}$, or $\langle x , \alpha\rangle > 0$, in which case $x\in \big\{x\in V\mid \langle x,\alpha\rangle\geq0,\ \langle x,\gamma \rangle\leq0\big\}$. Thus, $x\in S$, which shows (a).

Now we show (b). By (a) it suffices to show that $S\subseteq P$. If $x\in S$, then there are two cases. Assume first that $\langle x,\alpha\rangle\leq0$ and $\langle x,\beta \rangle\leq0$. These two inequalities yield $\langle x,p\alpha+q\beta\rangle\leq0$, hence $\langle x,\gamma\rangle\leq0$. Since $\langle x,\beta \rangle\leq0$ by assumption, we obtain $x\in P$. Alternatively, we have $\langle x,\alpha\rangle\geq0$ and $\langle x,p\alpha+q\beta\rangle=\langle x,\gamma\rangle\leq0$. These two inequalities yield $\langle x,\beta\rangle\leq0$. Since $\langle x,\gamma \rangle\leq0$ by assumption, we again obtain $x\in P$. This shows (b).
\end{proof}

\subsection{Root systems of root data}

Let $\mathscr D\coloneqq \big(A,\Pi,\Pi^\smvee,V,V^*,\langle\cdot\, ,\cdot\rangle\big)$ be a set of root data. One can define reflections $r_\alpha$, dual reflections $r_{\alpha^{\smvee}}$, the Weyl group $W_{\mathscr D}$ of $\mathscr D$ and the dual Weyl group $W^{\smvee}_{\mathscr{D}}$ similarly as in \S\ref{subsec:roots}. The \emph{root system of $\mathscr{D}$} is the set
$$\Delta^{\re}_\mathscr D\coloneqq W_{\mathscr D}\Pi\subseteq V.$$
Note that this terminology unfortunately differs from the one in \S\ref{subsec:roots}; we hope this will not cause confusion and it will be clear from the context what the phrase \emph{root system} means.

One can define positive roots $(\Delta^\re_{\mathscr D})_+$ and negative roots $(\Delta^\re_{\mathscr D})_-$, as well as coroots, similarly as in \S\ref{subsec:roots}, see \cite[p.\ 666]{MP89}. We will need the following important property \cite[Corollary on p.\ 668]{MP89}.

\begin{prop}\label{pro:positivenegative}
Let $\mathscr D\coloneqq\big(A,\Pi,\Pi^\smvee,V,V^*,\langle\cdot\, ,\cdot\rangle\big)$ be a set of root data. Then we have the disjoint union
$$\Delta^\re_{\mathscr D}=(\Delta^\re_{\mathscr D})_+\cup(\Delta^\re_{\mathscr D})_-.$$
\end{prop}

A subset $\Delta$ of an $\R$-space $V$ is a \emph{root system} if $\Delta$ is the set of real roots of a set of root data of the form $\big(A,\Pi,\Pi^\smvee,V,V^*,\langle\cdot\, ,\cdot\rangle\big)$.

We will frequently use the following property from \cite[Section 2, Proposition 2]{MP89}.

\begin{prop}\label{prop:W_invariant_pairing}
    Let $\mathscr D \coloneqq \big(A,\Pi,\Pi^\smvee,V,V^*,\langle\cdot\, ,\cdot\rangle\big)$ be a set of root data. There exists a unique group isomorphism between $W_{\mathscr{D}}$ and $W_{\mathscr{D}}^\smvee$. Moreover, if we identify $W_{\mathscr{D}}$ and $W_{\mathscr{D}}^\smvee$ via this isomorphism, then
    $$ \langle x, y \rangle = \big\langle w(x), w(y) \big\rangle\quad\text{for any $x\in V$, $y\in V^*$ and $w\in W_{\mathscr{D}}$.} $$
\end{prop}

\subsection{Root data of affine type}

\begin{dfn}
    Let $\mathscr D \coloneqq \big(A,\Pi,\Pi^\smvee,V,V^*,\langle\cdot\, ,\cdot\rangle\big)$ be a set of root data. We say that $\mathscr{D}$ is of \emph{affine type} if $A$ is an indecomposable symmetric generalised Cartan matrix of  affine type, and there exist vector subspaces $U\subseteq V$ and $U'\subseteq V^*$ such that $\Pi\subseteq U$, $\Pi^\smvee\subseteq U'$, $U'$ naturally has the structure of the dual of $U$, and $(U,\Pi, \Pi^\smvee)$ is a realisation of the generalised Cartan matrix $A$ in the sense of classical root systems.
\end{dfn}

The following lemma says that the root system of a set of root data of affine type is actually the set of real roots of an affine root system in the sense of \cite{Kac90}.

\begin{lem}\label{lem:affine_root_data}
    Let $\mathscr D \coloneqq \big(A,\Pi,\Pi^\smvee,V,V^*,\langle\cdot\, ,\cdot\rangle\big)$ be a set of root data. Assume that $A$ is an indecomposable symmetric generalised Cartan matrix of affine type, and that $\Pi$ and $\Pi^\smvee$ are linearly independent with $|\Pi| = |\Pi^\smvee| < \dim V$. Then $\mathscr D $ is a set of root data of affine type.
\end{lem}

\begin{proof}
   Set $n \coloneqq \dim V$, $\Pi \eqqcolon \{\alpha_0,\alpha_1,\dots,\alpha_\ell\}$ and $\Pi^\smvee \eqqcolon \{\alpha_0^\smvee,\alpha_1^\smvee,\dots,\alpha_\ell^\smvee\}$. Denote
   $$\textstyle V_\Pi\coloneqq\sum\limits_{i=0}^\ell\R\alpha_i\subseteq V\quad\text{and} \quad V_{\Pi^\smvee}\coloneqq\sum\limits_{i=0}^\ell\R\alpha_i^\smvee\subseteq V^*.$$
   Recall that $A=\big(\langle\alpha_i,\alpha_j^\smvee\rangle\big)_{i,j=0}^\ell$, and set $A^\circ\coloneqq \big(\langle\alpha_i,\alpha_j^\smvee\rangle\big)_{i,j=1}^\ell$. Since $A$ is of affine type, both $A$ and $A^\circ$ have rank $\ell$.
   
   \medskip

   \emph{Step 1.}
   In this step we show that:
	\begin{enumerate}[\normalfont (i)]
	\item there exists $\alpha_{\ell+1}\in V\backslash V_\Pi$ such that $\langle \alpha_{\ell+1}, \alpha_0^\smvee \rangle =1$ and $\langle \alpha_{\ell+1}, \alpha_i^\smvee \rangle =0$ for any $i\in \{1,\ldots,\ell\}$,
	\item there exists $\alpha_{\ell+1}^\smvee\in V^*\backslash V_{\Pi^\smvee}$ such that $\langle \alpha_0,\alpha_{\ell+1}^\smvee \rangle =1$ and $\langle \alpha_i,\alpha_{\ell+1}^\smvee \rangle =0$ for any $i\in \{1,\ldots,\ell\}$.
	\end{enumerate}	   
   We only show (ii), as the proof of (i) is analogous.
   
	Denote $V_\Pi^\circ \coloneqq \sum\limits_{i=1}^\ell\R\alpha_i\subseteq V$, and let $ (V_\Pi)^\smperp $  and $ (V_\Pi^\circ)^\smperp $ be the vector subspaces of $V^*$ consisting of all vectors which are orthogonal to $V_\Pi$ and $V_\Pi^\circ$ with respect to the pairing $\langle\cdot\, ,\cdot\rangle$. Since the pairing $\langle\cdot\, ,\cdot\rangle$ is non-degenerate, we have $\dim (V_\Pi)^\smperp=n-\ell-1$ and $\dim(V_\Pi^\circ)^\smperp=n-\ell$. Therefore, there exists $\alpha_{\ell+1}^\smvee\in (V_\Pi^\circ)^\smperp\setminus (V_\Pi)^\smperp$, and by rescaling, we may assume that $\langle \alpha_0,\alpha_{\ell+1}^\smvee \rangle =1$.
   
   Assume that $\alpha_{\ell+1}^\smvee\in V_{\Pi^\smvee}$. Then there exists a vector $\mathbf{x}\coloneqq (x_0,\dots,x_\ell)\in\R^{\ell+1}$ such that $\alpha_{\ell+1}^\smvee=x_0\alpha_0^\smvee+\dots+x_\ell\alpha_\ell^\smvee$. Since $\langle \alpha_0,\alpha_{\ell+1}^\smvee \rangle =1$ and $\langle \alpha_i,\alpha_{\ell+1}^\smvee \rangle =0$ for any $i\in \{1,\ldots,\ell\}$, we conclude that $\mathbf{x}A=(1,0,\dots,0)$. Now, since $A$ is of affine type, by \cite[Theorem 4.8 and Tables Aff 1--3]{Kac90} there exists a vector $\delta=(a_0,\dots,a_\ell)^T$ such that $A\delta=0$ and $a_0\neq0$. Thus,
   $$0\neq a_0=(1,0,\dots,0)\cdot\delta=(\mathbf{x}A)\delta=\mathbf{x}(A\delta)=0,$$
   a contradiction.

   \medskip
   
   \emph{Step 2.}
   Set
   $$U_\Pi\coloneqq V_\Pi\oplus\R\alpha_{\ell+1}\quad\text{and}\quad U_{\Pi^\smvee}\coloneqq V_{\Pi^\smvee}\oplus\R\alpha_{\ell+1}^\smvee.$$
   In order to finish the proof it suffices to show that $\dim U_\Pi=\ell+2$ and $U_{\Pi^\smvee}$ naturally has the structure of the dual of $U_\Pi$. To that end, since $U_{\Pi^\smvee}\subseteq V^*$, it suffices to show that $\dim U_\Pi=\dim U_{\Pi^\smvee}=\ell+2$. We only show that $\dim U_\Pi=\ell+2$, as the proof of the second equality is analogous.

   We need to show that the vectors $\alpha_0,\dots,\alpha_{\ell+1}$ are linearly independent. To that end, let $\mathbf{y}\coloneqq(y_0,\dots,y_{\ell+1})\in\R^{\ell+2}$ such that $y_0\alpha_0+\dots+y_{\ell+1}\alpha_{\ell+1}=0$. By applying the linear forms $\alpha_0^\smvee,\dots,\alpha^\smvee_{\ell+1}$ to this equation, we conclude that $\mathbf{y}B=0$, where $B\coloneqq\big(\langle\alpha_i,\alpha_j^\smvee\rangle\big)_{i,j=0}^{\ell+1}$. By applying the Laplace expansion twice and since $\det A=0$, it is easy to calculate that $\det B={-}\det A^\circ\neq0$. Thus, $\mathbf{y}=0$, as desired.
\end{proof}

\subsection{The negative fundamental chamber and the negative Tits cone}

    Consider a system of root data $\mathscr D\coloneqq \big(A,\Pi,\Pi^\smvee,V,V^*,\langle\cdot\, ,\cdot\rangle\big)$. The cone
        $$ \mathcal F_{\mathscr D}\coloneqq  \big\{x\in V\mid \langle x,\beta\rangle < 0 \text{ for all }\beta\in \Pi^\smvee\big\} \subseteq V $$
        is the \emph{negative fundamental chamber of $\mathscr D$}. If $W_{\mathscr D}$ is the Weyl group of $\mathscr D$ and if $\overline{\mathcal F_{\mathscr D}}$ denotes the closure of $\mathcal F_{\mathscr D}$ in $V$, then
        $$ \mathcal T_{\mathscr D} \coloneqq  \bigcup_{w\in W_{\mathscr D}} w\big(\overline{\mathcal F_{\mathscr D}}\big)\subseteq V $$
        is the \emph{negative Tits cone of $\mathscr D$}. It is $W_{\mathscr D}$-invariant. Note that in the literature one usually considers the fundamental chamber and the Tits cone which are the negatives of these cones. In this paper we opt for this change of sign, since it simplifies notation considerably.

The following result is \cite[Proposition 8(FD1)]{MP89}, and it is of fundamental importance for our paper. 

\begin{prop}\label{prop:Titsfundchamber}
Let $\mathscr D\coloneqq \big(A,\Pi,\Pi^\smvee,V,V^*,\langle\cdot\, ,\cdot\rangle\big)$ be a system of root data. Then $\overline{\mathcal F_{\mathscr D}}$ is a fundamental domain for the action of $W_{\mathscr D}$ on $\mathcal T_{\mathscr D}$.
\end{prop}        
        
        We also have the following alternative description of the negative Tits cone, see \cite[Proposition 4]{MP89} and \cite[Proposition 3.12(c)]{Kac90}.

\begin{prop}\label{prop:titscone}
Let $\mathscr D \coloneqq \big(A,\Pi,\Pi^\smvee,V,V^*,\langle\cdot\, ,\cdot\rangle\big)$ be a set of root data. Then $\mathcal T_{\mathscr D} $ is a convex cone and
$$ \mathcal T_{\mathscr D} = \{ x \in V \mid \langle x,\alpha^\smvee \rangle \leq 0 \text{ for all but finitely many }\alpha\in (\Delta^\re_\mathscr D)_+\}. $$
\end{prop}

\subsection{Subroot systems}

We need the following definition of subroot systems from \cite[Section 6]{MP89}.

\begin{dfn}\label{dfn:subsystem}
   Let $\mathscr D\coloneqq\big(A,\Pi,\Pi^\smvee,V,V^*,\langle\cdot\, ,\cdot\rangle\big)$ be a set of root data. A non-empty subset $\Omega\subseteq \Delta^\re_\mathscr D$ is a \emph{subroot system} if there exists a generalised Cartan matrix $B$ and a subset $\Upsilon\subseteq \Delta^\re_+$ such that if we define $\Upsilon^\smvee \coloneqq \{\alpha^\smvee\mid\alpha\in\Upsilon\}$, then $\big(B,\Upsilon,\Upsilon^\smvee,V,V^*,\langle\cdot\, ,\cdot\rangle\big)$ is a set of root data with the root system $\Omega$.
\end{dfn}

As an immediate corollary to Proposition \ref{prop:titscone}, we have:

\begin{cor}\label{cor:Titssubcone}
Let $\mathscr D\coloneqq\big(A,\Pi,\Pi^\smvee,V,V^*,\langle\cdot\, ,\cdot\rangle\big)$ be a set of root data. Let $\Omega\subseteq \Delta^\re_\mathscr D$ be a subroot system and let $\mathscr D'\coloneqq\big(B,\Upsilon,\Upsilon^\smvee,V,V^*,\langle\cdot\, ,\cdot\rangle\big)$ be a set of root data with the root system $\Omega$ as in Definition \ref{dfn:subsystem}. Then
$$\mathcal T_{\mathscr D}\subseteq\mathcal{T}_{\mathscr D'}.$$
\end{cor}

\begin{proof}
Since $\Upsilon \subseteq (\Delta^\re_\mathscr D)_+$, we have $\Omega_+\subseteq (\Delta^\re_\mathscr D)_+$. Hence, the statement follows from Proposition \ref{prop:titscone}.
\end{proof}

The following lemma is used in the proof of Proposition \ref{prop:WDkl}.

\begin{lem}\label{lem:equivalence_chamber_inclusion}
    Let $\mathscr D\coloneqq \big(A,\Pi,\Pi^\smvee,V,V^*,\langle\cdot\, ,\cdot\rangle\big)$ be a set of root data. Let $\Sigma$ be a subset of $\Delta^\re_\mathscr D$ and define
    $$\mathcal{F}_\Sigma \coloneqq \{ x\in V \mid \langle x, \alpha^\smvee\rangle < 0 \text{ for all }\alpha\in \Sigma \}.$$
    Assume that $\mathcal F_\mathscr D\neq\emptyset$. Then $\mathcal{F}_{\mathscr D}\subseteq \mathcal{F}_\Sigma$ if and only if $\Sigma\subseteq (\Delta^\re_\mathscr D)_+$.
\end{lem}

\begin{proof}
	Let $\Sigma=\{\alpha_i\mid i\in I\}$ and $\Pi=\{\beta_j\mid j\in J\}$. Since every real root is either positive or negative, there exist integers $n_{ij}$ such that
    $$ \alpha_i = \sum_{j\in J} n_{ij} \beta_j\quad\text{for each }i\in I, $$
    where for each fixed $i$ the numbers $n_{ij}$ are either all non-positive or non-negative. Moreover, for each $i$ we have $\alpha_i\neq0$, hence there exists $j_i\in J$ such that $n_{ij_i}\neq0$.
    
    Assume first that $\mathcal{F}_{\mathscr D}\subseteq \mathcal{F}_\Sigma$, and assume that there exists $i\in I$ such that $\alpha_i\notin(\Delta^\re_\mathscr D)_+$. Then $n_{ij}\leq0$ for all $j\in J$. Fix $x\in\mathcal F_\mathscr D$; thus $x\in\mathcal{F}_\Sigma$ by assumption. Then $\langle\alpha_i,x\rangle<0$ and $\langle\beta_j,x\rangle<0$ for all $j\in J$, hence
    $$ 0>\langle\alpha_i,x\rangle = \sum_{j\in J} n_{ij} \langle\beta_j,x\rangle\geq n_{ij_i} \langle\beta_{j_i},x\rangle>0, $$
    a contradiction which shows that $\alpha_i\in (\Delta^\re_\mathscr D)_+$ for all $i\in I$.

    Conversely, assume that $\Sigma\subseteq (\Delta^\re_\mathscr D)_+$. Then $n_{ij}\geq0$ for all $i$ and $j$. Let $x\in\mathcal F_\mathscr D$. Then $\langle\beta_j,x\rangle<0$ for all $j\in J$, hence for each $i\in I$ we have
    $$ \langle \alpha_i, x\rangle = \sum_{j\in J} n_{ij} \langle\beta_j,x\rangle\leq n_{ij_i} \langle\beta_{j_i},x\rangle<0. $$
    This shows that $x\in \mathcal{F}_{\Sigma}$, and therefore $\mathcal{F}_\mathscr D\subseteq \mathcal{F}_\Sigma$, as desired.
\end{proof}

\subsection{Decomposable sets of root data}

We recall the definition of decomposable sets from \cite{MP89}.

\begin{dfn}
Let $\mathscr D\coloneqq \big(A,\Pi,\Pi^\smvee,V,V^*,\langle\cdot\, ,\cdot\rangle\big)$ be a set of root data. A subset $S\subseteq\Delta^\re_\mathscr D$ is \emph{decomposable} if we can write $S = S_1\cup S_2$ with $S_1\neq\emptyset$ and $S_2\neq\emptyset$, and for all $\gamma_1\in S_1$ and $\gamma_2\in S_2$ we have $\langle\gamma_1,\gamma_2^\smvee\rangle = \langle\gamma_2,\gamma_1^\smvee\rangle = 0$. Otherwise we say that $S$ is \emph{indecomposable}. If $\Delta^\re_\mathscr D$ is indecomposable, then we say that $\mathscr D$ is indecomposable. By \cite[Proposition 1 on p.\ 689]{MP89}, the set of root data $\mathscr D$ is indecomposable if and only if $\Pi$ is indecomposable.
\end{dfn}

Note that in the notation as in the previous definition, the union $S = S_1\cup S_2$ is automatically disjoint, since $\langle\alpha,\alpha^\smvee\rangle=2$ for all $\alpha\in\Delta^\re_\mathscr D$.

Moreover, the set $\Pi$ is decomposable if and only if the generalised Cartan matrix $A$ is decomposable. More precisely, let $\Pi=\Pi_1\cup\Pi_2$ be the corresponding decomposition. If $A_1$ and $A_2$ are the corresponding generalised Cartan matrices, then $A=A_1\oplus A_2$. Furthermore, if we set $\mathscr D_i\coloneqq\big(A_i,\Pi_i,\Pi_i^\smvee,V,V^*,\langle\cdot\, ,\cdot\rangle\big)$ for $i\in\{1,2\}$, then $\mathscr D_1$ and $\mathscr D_2$ are sets of root data by \cite[p.\ 669]{MP89}, and we have the following result.

\begin{prop}\label{pro:decomposable}
With notation as above, then:
\begin{enumerate}[\normalfont (a)]
\item we have $W_{\mathscr D}=W_{\mathscr D_1}W_{\mathscr D_2}$, and for all $w_1\in W_{\mathscr D_1}$ and $w_2\in W_{\mathscr D_2}$ we have $w_1w_2=w_2w_1$,
\item the set $\Delta^\re_\mathscr D$ is decomposable, with the corresponding decomposition $\Delta^\re_\mathscr D=\Delta^\re_{\mathscr D_1}\cup\Delta^\re_{\mathscr D_2}$,
\item the set $(\Delta^\re_\mathscr D)_+$ is decomposable, with the corresponding decomposition $(\Delta^\re_\mathscr D)_+=(\Delta^\re_{\mathscr D_1})_+\cup(\Delta^\re_{\mathscr D_2})_+$,
\item $\cal{F}_{\mathscr D}=\cal{F}_{\mathscr D_1}\cap \cal{F}_{\mathscr D_2}$,
\item for distinct $i,j\in\{1,2\}$ and for any $w_i\in W_{\mathscr D_i}$ we have $w_i(\cal{F}_{\mathscr{D}_j}) = \cal{F}_{\mathscr{D}_j}$ and $w_i(\cal{T}_{\mathscr{D}_j}) = \cal{T}_{\mathscr{D}_j}$,
\item for distinct $i,j\in\{1,2\}$ and for any $w_i\in W_{\mathscr D_i}$ we have $w_i(\cal{F}_{\mathscr{D}}) = w_i(\cal{F}_{\mathscr{D}_i})\cap \cal{F}_{\mathscr{D}_j}$.
\end{enumerate}
\end{prop}

\begin{proof}
Parts (a) and (b) follow from the proof of \cite[Proposition 1 on p.\ 689]{MP89}.

To show (c), note that $(\Delta^\re_{\mathscr D_1})_+\cup(\Delta^\re_{\mathscr D_2})_+\subseteq(\Delta^\re_{\mathscr D})_+$ by (b) and by the definition of positive roots. For the opposite inclusion, let $\gamma\in(\Delta^\re_{\mathscr D})_+$. Then by (b) we have $\gamma\in\Delta^\re_{\mathscr D_1}\cup\Delta^\re_{\mathscr D_2}$, and without loss of generality we may assume that $\gamma\in\Delta^\re_{\mathscr D_1}$. By Proposition \ref{pro:positivenegative} we conclude that either $\gamma\in(\Delta^\re_{\mathscr D_1})_+$ or $\gamma\in(\Delta^\re_{\mathscr D_1})_-$. Since $(\Delta^\re_{\mathscr D_1})_-\subseteq(\Delta^\re_{\mathscr D})_-$, the second alternative would imply that $\gamma\in(\Delta^\re_{\mathscr D})_-$. But this is impossible since $(\Delta^\re_{\mathscr D})_+\cap(\Delta^\re_{\mathscr D})_-=\emptyset$ by Proposition \ref{pro:positivenegative}.

Part (d) follows immediately from the definition of negative fundamental chambers.

Now we prove (e) and (f). Fix $w_1 \in W_{\mathscr{D}_1}$. Note that we have $\alpha_i\cdot \alpha'_j = 0$ for any $\alpha_i \in \Pi_1$ and for any $\alpha'_j\in \Pi_2$, hence $r_{\alpha_i}(\alpha'_j) = \alpha'_j$. Therefore, by Proposition \ref{prop:W_invariant_pairing} we have
\begin{align*}
    r_{\alpha_i} (\cal{F}_{\mathscr{D}_2}) &= 
    \{ r_{\alpha_i}(y) \mid y \in V, \langle y,\alpha'_j \rangle < 0 \text{ for any }\alpha'_j\in \Pi_2 \} \\
    &= \{  x \in V \mid x \in V, \langle x, r_{\alpha_i}(\alpha'_j) \rangle <0 \text{ for any } \alpha'_j\in \Pi_2 \} \\
    &= \{ x\in V \mid \langle x , \alpha'_j \rangle < 0 \text{ for any } \alpha'_j\in \Pi_2 \} = \cal{F}_{\mathscr{D}_2}.
\end{align*}
As $w_1$ is a word in the reflections $r_{\alpha_i}$, we conclude that
\begin{equation}\label{eq:decompsable_fund_chambers}
w_1(\cal{F}_{\mathscr{D}_2}) = \cal{F}_{\mathscr{D}_2}.
\end{equation}
Therefore, by (d) we have
$$
w_1 (\cal{F}_{\mathscr{D}}) = w_1 (\cal{F}_{\mathscr{D}_1}) \cap w_1 ( \cal{F}_{\mathscr{D}_2}) = w_1 (\cal{F}_{\mathscr{D}_1}) \cap \cal{F}_{\mathscr{D}_2}.
$$
Furthermore, by \eqref{eq:decompsable_fund_chambers} and by (a) we have
\begin{align*}
w_1(\cal{T}_{\mathscr{D}_2}) &= \textstyle w_1 \big(\bigcup_{w_2\in W_{\mathscr {D}_2}}w_2(\overline{\cal{F}_{\mathscr{D}_2}}) \big)= \bigcup_{w_2\in W_{\mathscr{D}_2}}w_1 w_2 (\overline{\cal{F}_{\mathscr{D}_2}}) \\
&= \textstyle \bigcup_{w_2 \in W_{\mathscr{D}_2}}w_2 w_1 (\overline{\cal{F}_{\mathscr{D}_2}}) = \bigcup_{w_2\in W_{\mathscr{D}_2}}w_2(\overline{\cal{F}_{\mathscr{D}_2}}) =\cal{T}_{\mathscr{D}_2}.
\end{align*}
The rest of (e) and (f) is proved analogously.
\end{proof}

\subsection{The length function}

Let $\mathscr D\coloneqq\big(A,\Pi,\Pi^\smvee,V,V^*,\langle \cdot\, ,\cdot\rangle\big)$ be a set of root data, let $\Pi = \{ \alpha_i \mid i\in I\}$, and let $W_{\mathscr D}$ be its associated Weyl group with generators $r_i$. Then $W_{\mathscr D}$ together with this set of generators is a Coxeter group by \cite[p.\ 668, property (a)]{MP89}, thus we can associate to it the length function $\ell\colon W_{\mathscr D} \to \N$. The following results are analogues of \cite[Lemma 3.11(a) and Proposition 3.12(d)]{Kac90}, where they were proved in the context of the usual root systems. We give the proofs here for the benefit of the reader.

\begin{lem}\label{lem:length}
     Let $\mathscr D\coloneqq \big(A,\Pi,\Pi^\smvee,V,V^*,\langle \cdot\, ,\cdot\rangle\big)$ be a set of root data. Let $w\in W_{\mathscr D}$ and let $\alpha\in \Pi$. If $w(\alpha)\in\Delta^\re_\mathscr D$ is a negative real root, then $\ell(w \cdot r_{\alpha}) < \ell(w)$.
\end{lem}

\begin{proof}
    With notation as above, let $n\coloneqq \ell(w)$ and let $w = r_{i_1} \dots r_{i_n}$ be a reduced expression of $w\in W_\mathscr D$, where $i_s\in I$ for any $s\in\{1,\dots,n\}$.

    We set $\beta_k \coloneqq (r_{i_{k+1}}\dots r_{i_n})(\alpha)$ for any $k<n$, and $\beta_n \coloneqq \alpha$. Then $\beta_0 = w(\alpha)$ is a negative real root by assumption, and $\beta_n = \alpha$ is a positive real root. Thus, there exists $s\in \{ 1,\dots, n\}$ such that
    $$ \beta_{s-1} = r_{i_s}(\beta_s) \in (\Delta^\re_\mathscr D)_- \quad\text{and}\quad \beta_s \in (\Delta^\re_\mathscr D)_+. $$
	In particular, there exist $a_i\in \Z_{\geq 0}$ such that $\beta_s = \sum_{i\in I} a_i \alpha_i$. Then
     \begin{align*}
     (\Delta^\re_\mathscr D)_-\ni\beta_{s-1}&=r_{i_s}(\beta_s) = \sum_{i\in I\setminus\{i_s\}} a_i r_{i_s} (\alpha_i) + a_{i_s} r_{i_s}(\alpha_{i_s}) \\
     &= \sum_{i\in I\setminus\{i_s\}} a_i \alpha_i - \Big( \sum_{i\in I\setminus\{i_s\}} a_i \langle \alpha_i, \alpha_{i_s}^\smvee \rangle + a_{i_s}\Big) \alpha_{i_s},
     \end{align*}
     hence $a_i=0$ for any $i\in I\setminus\{i_s\}$. Therefore, 
     $$ \beta_s=a_{i_s}\alpha_{i_s}, \quad\text{where }a_{i_s}>0.$$
Note that $\langle \beta_s, \beta_s^\smvee \rangle = \langle \alpha,\alpha^\smvee\rangle = 2$ by construction and by Proposition \ref{prop:W_invariant_pairing}. As $\langle \alpha_{i_s},\alpha_{i_s}^\smvee \rangle =2$, we conclude that $\beta_s = \alpha_{i_s}$, hence
     $$ w'(\alpha) = \alpha_{i_s}, \quad\text{where }w'\coloneqq r_{i_{s+1}}\cdot\ldots\cdot r_{i_n}. $$
     It follows from \cite[p.\ 669, property (c)]{MP89} that
     $$ r_{i_s} = w'\cdot r_\alpha \cdot w'^{-1}. $$
     Multiplying both sides by $r_{i_1}\cdot\ldots\cdot r_{i_{s-1}}$ on the left and by $r_{i_{s+1}}\cdot\ldots\cdot r_n \cdot r_\alpha$ on the right, we obtain
     $$ w\cdot r_\alpha = r_{i_1}\dots r_{i_{s-1}} \cdot r_{i_{s+1}}\dots r_{i_n}, $$
     where we used that $w'^{-1}=r_{i_n}\cdot\ldots\cdot r_{i_{s+1}}$. Therefore, $\ell (w\cdot r_\alpha)\leq n-1 = \ell(w)-1$, which completes the proof.
\end{proof}

The following lemma will be essential for the construction of a rational polyhedral fundamental domain in the proof of Theorem \ref{thm:main2}.

\begin{lem}\label{lem:actionW}
    Let $\mathscr D\coloneqq \big(A,\Pi,\Pi^\smvee,V,V^*,\langle \cdot\, ,\cdot\rangle\big)$ be a set of root data, and let $x\in \overline{\mathcal{F}_{\mathscr D}}$. Then for any $w\in W_\mathscr D$ we have that $ w(x) - x $ is a nonnegative linear combination of the elements of $\Pi$.
\end{lem}

\begin{proof}
We will use a few times in the proof that for every $\alpha\in \Pi$, since $x\in \overline{\mathcal{F}_{\mathscr D}}$, we have
    \begin{equation}\label{eq:induction1}
    r_\alpha(x) - x = - \langle x , \alpha^\smvee \rangle \alpha \in \R_{\geq 0}\cdot \alpha.
    \end{equation}
    
    We use the notation from the paragraph preceding Lemma \ref{lem:length}. Fix $w\in W_\mathscr D$. Let $n\coloneqq \ell(w)$ and let $w = r_{i_1} \dots r_{i_n}$ be a reduced expression of $w\in W_\mathscr D$, where $i_s\in I$ for any $s\in\{1,\dots,n\}$. We will prove the statement by induction on $n$.
    
    If $n=1$, then we conclude by \eqref{eq:induction1}. Now assume that the statement holds for any element in $W_\mathscr D$ of length $n-1$. Set $w' \coloneqq  r_{i_1} \cdot\ldots\cdot r_{i_{n-1}}$. Then
    $$ w(x) - x = \big( w'(x) - x \big) + w'\big( r_{i_n}(x)-x\big). $$
    Note that the first summand is a nonnegative linear combination of the elements in $\Pi$ by induction, while the second summand belongs to the set $\R_{\geq 0} \cdot w'(\alpha_{i_n}) $ by \eqref{eq:induction1}. Thus, it suffices to show that $w'(\alpha_{i_n})$ is a nonnegative linear combination of the elements in $\Pi$.
    
    To that end, since $w'(\alpha_{i_n})$ is a real root, it is either positive or negative. If $w'(\alpha_{i_n})\in(\Delta^\re_\mathscr D)_-$, then by Lemma \ref{lem:length} we have
    $$n=\ell(w)=\ell(w'\cdot r_{i_n}) < \ell(w')=n-1,$$
    a contradiction. Consequently, $w'(\alpha_{i_n})\in(\Delta^\re_\mathscr D)_+$, and in particular, $w'(\alpha_{i_n})$ is a nonnegative linear combination of the elements in $\Pi$, as desired.
\end{proof}

\section{Root systems on Sakai surfaces}\label{sec:rootsystemsHalphen}

Throughout the section, $X$ is a Sakai surface with the unique divisor $D\in|{-}K_X|$. We denote by $\cal{D}_i$ the classes of the components of $D$.

Recall from \cite[Lemma 2.9(b)]{LSX26} that $X$ is the blowup of $\pr^2$ at $9$ (possibly infinitely near) points. Then there exist birational morphisms $\pi_i\colon X_i\to X_{i-1}$ for $i=1,\dots,9$, such that $X_9:=X$ and $X_0:=\pr^2$, and each $\pi_i$ is a contraction of a $(-1)$-curve $E_i\subseteq X_i$. We call the induced morphism $\pi_1\circ\ldots\circ\pi_9\colon X\to\mathbb{P}^2$ a \emph{blowdown structure} on $X$.

Let $\mathcal E_0\in N^1(X)_\R$ be class of the pullback of a line on $\pr^2$, and let $\mathcal E_i\in N^1(X)_\R$ be the class of the pullback of $E_i$ for any $i\in\{1,\dots,9\}$. Then $\{\mathcal{E}_0,\ldots,\mathcal{E}_9\}$ is a basis of the vector space $N_1(X)_\R$, see for instance \cite[Lemma 2.1(c)]{LSX26}; we call this basis the \textit{tautological basis} of $N_1(X)_\R$, and we say that $(\mathcal E_0,\dots,\mathcal E_9)$ \emph{defines the blowdown structure $\pi_1\circ\ldots\circ\pi_9\colon X\to\mathbb{P}^2$}.

It is easy to check that
\begin{equation}\label{eq:intersectnumbers}
    \mathcal E_0^2=1, \quad \mathcal E_i^2={-}1 \text{ for }i>0,\quad \mathcal E_i\cdot\mathcal E_j= 0 \text{ for }i\neq j,
\end{equation}
and
\begin{equation}\label{eq:anticandiv}
    -K_X \sim 3\mathcal E_0 - \mathcal E_1-\dots-\mathcal E_9.
\end{equation}
Moreover,
\begin{equation}\label{eq:picgroup}
N^1(X) = \Z \mathcal E_0\oplus\dots\oplus \Z \mathcal E_9,
\end{equation}
as can easily be seen by the induction on the number of blowups of $\mathbb{P}^2$.

\subsection{Root system $E_8^\smone$}\label{subsec:tautological}

With notation as above, consider the following classes in $N^1(X)_\R$:
$$ \rho_8 \coloneqq \mathcal E_0 - \mathcal E_1 - \mathcal E_2 - \mathcal E_3, \quad \rho_i \coloneqq \mathcal E_i-\mathcal E_{i+1}\text{ for }i=1,\dots, 7,\quad \rho_0 \coloneqq \mathcal E_8-\mathcal E_9;$$
note here that the labelling of these elements is chosen to be compatible with the discussion of affine root systems in Section \ref{sec:rootsystems}. Then the set
$$ \Pi(E_8^\smone)\coloneqq \{\rho_0,\rho_1,\dots,\rho_8\}\subseteq N^1(X)_\R $$
is linearly independent in $N^1(X)_\R$. We consider the pairing
\begin{equation}\label{eq:pairing}
\langle\cdot\, ,\cdot\rangle\colon N^1(X)_\R\times N_1(X)_\R\to\R,\quad \langle\alpha,\beta\rangle\mapsto{-}\alpha\cdot \beta,
\end{equation}
and we set 
$$\rho_i^\smvee\coloneqq \rho_i\in N_1(X)_\R\quad\text{for all }i.$$
Set
$$A\coloneqq \big(\langle \rho_i,\rho_j^\smvee\rangle\big)_{0\leq i,j\leq8}.$$
Then it is easy to check, using \eqref{eq:intersectnumbers}, that $A$ is a generalised Cartan matrix and that the triple $\big(N^1(X)_\R,\Pi(E_8^\smone),\Pi(E_8^\smone)^\smvee\big)$ is its realisation. The Dynkin diagram associated to $A$ is of type $E_8^\smone$, see \cite[p.\ 44]{Kac90}. In particular, $A$ is indecomposable and of affine type. Additionally, if we denote
$$\delta\coloneqq [{-}K_X],$$
then it is easy to check by using \eqref{eq:anticandiv} that $\delta$ is the imaginary root of this root system from Theorem \ref{thm:imaginary_roots}(b). Note that
$$ \delta^2=0\quad\text{and}\quad\delta\cdot\rho_i=0\text{ for all }i. $$
We denote by $Q(E_8^\smone)$ the lattice spanned by $\Pi(E_8^\smone)$. By \cite[Propositions 7 and 8]{Sak01} we have $Q(E_8^\smone)=\{x\in N^1(X)\mid x\cdot\delta=0\}$. We denote by $W(E_8^\smone)$ the associated Weyl group and by $\Delta(E_8^\smone)^\re$ the associated set of real roots. Since every element of $Q(E_8^\smone)$ is orthogonal to $\delta$, we have
\begin{equation}\label{eq:orthogonaltodelta}
w(\delta)=\delta\quad\text{for any }w\in W(E_8^\smone).
\end{equation}

We will need the following simple lemma in the proof of Corollary \ref{cor:finiteness}.

\begin{lem}\label{lem:integral_elements}
With notation as above, we have
$$ \sum_{i=1}^{8} \Z \rho_i = \sum_{i=1}^{8} \R \rho_i \cap \sum_{i=0}^8 \Z \mathcal E_i . $$
\end{lem}

\begin{proof}
One inclusion is clear. Conversely, let $\alpha \in \sum_{i=1}^{8} \R \rho_i \cap \sum_{i=0}^8 \Z \mathcal E_i $. Then there exist $a_i\in \R$ such that $\alpha = \sum_{i=1}^{8} a_i \rho_i $. This implies that
\begin{align*}
\alpha &= a_8\mathcal E_0+(a_1-a_8)\mathcal E_1+(a_2-a_1-a_8)\mathcal E_2+(a_3-a_2-a_8)\mathcal E_3\\
&+(a_4-a_3)\mathcal E_4+(a_5-a_4)\mathcal E_5+(a_6-a_5)\mathcal E_6+(a_7-a_6)\mathcal E_7-a_7\mathcal E_8.
\end{align*}
Since $\alpha \in \sum_{i=0}^8 \Z \mathcal E_i $ and since $\mathcal E_i$ are linearly independent in $N^1(X)_\R$, we conclude that the coefficients with $\mathcal E_i$ on the right hand side of the equation above are integral, which easily implies that $a_i\in\Z$ for all $i$. This proves the lemma.
\end{proof}
	  	
\subsection{Normalised standard bilinear form on $E_8^\smone$}\label{subsec:normalisedE_8}

We now specialise the setup from \S\ref{subsec:normalised} to our geometric situation. Then, if $c\in N_1(X)_\R$ is the canonical central element of the root system $E_8^\smone$, we have 
\begin{equation}\label{eq:cdelta}
c=\delta,
\end{equation}
where we view $\delta$ as an element of $N_1(X)_\R$. We set $d\coloneqq {-}\mathcal E_9-\frac12\delta\in N_1(X)_\R$. Then it is easy to check that $d$ is a scaling element of the root system $E_8^\smone$, and set $\Lambda_0\coloneqq d$. Then it follows easily that the bilinear form
$$(\cdot\, ,\cdot)\colon N^1(X)_\R\times N^1(X)_\R\to\R,\quad (F,G)\mapsto{-}F\cdot G$$
is the associated normalised standard bilinear form for $E_8^\smone$ as in \S\ref{subsec:normalised}; here we use that $\mathcal{E}_9\cdot\delta=1$.

\subsection{Translations}\label{subsec:translations1}

With notation from \S\ref{subsec:tautological}, if we remove the vertex corresponding to $\rho_0$ from the Dynkin diagram of $E_8^\smone$ as in \cite[Table on p.\ 44]{Kac90}, we obtain the finite Dynkin diagram $E_8$. Note that
	\begin{equation}\label{eq:6677}
	\mathcal E_9\cdot\alpha=\delta\cdot\alpha=0\quad\text{for each }\alpha\in Q(E_8).
	\end{equation}
 For each $\alpha \in Q(E_8)$, we obtain the translation $t_\alpha\colon N^1(X)_\R\to N^1(X)_\R$ given by the formula \eqref{eq:117}; in our concrete geometric situation and taking \S\ref{subsec:normalisedE_8} into account, this formula takes the following form:
	    $$\textstyle t_\alpha(x) \coloneqq x - (x\cdot\delta)\alpha + (x\cdot\alpha)\delta - \frac{1}{2} \alpha^2(x\cdot\delta)\delta.$$
		It is easy to see that $\alpha^2$ is an even integer. In particular, as $\mathcal{E}_9\cdot \delta=1$ and by \eqref{eq:6677}, we have
	    \begin{equation}\label{eq:117aa}
	    \textstyle t_\alpha(\mathcal{E}_9) = \mathcal{E}_9 - \alpha - \big(\frac{1}{2}\alpha^2\big)\delta\in N^1(X).
	    \end{equation}
	    We note for future reference the following simple result.
	    
	  	\begin{lem}\label{lem:translationsatE9a}
	  	Assume notation as above. Let $\alpha$ and $\beta$ be elements in $Q(E_8)$ such that $t_\alpha(\mathcal{E}_9)=t_\beta(\mathcal{E}_9)$. Then $\alpha=\beta$.
	  	\end{lem}
	  	
	  	\begin{proof}
	  	The assumption implies that
	  	\begin{equation}\label{eq:alphabeta}
	  	\textstyle \alpha-\beta=\frac{1}{2}(\beta^2-\alpha^2)\delta.
		\end{equation}
		Multiplying both sides by $\alpha+\beta$ and using \eqref{eq:6677} gives $\alpha^2-\beta^2=0$. This, together with \eqref{eq:alphabeta}, gives $\alpha=\beta$, as desired.
	  	\end{proof}

\subsection{The negative fundamental chamber and the negative Tits cone}

With our notation as above, we have the negative fundamental chamber
        $$ \mathcal F(E_8^\smone) = \big\{x\in V\mid x\cdot\beta > 0 \text{ for all }\beta\in \Pi(E_8^\smone)^\smvee\big\} \subseteq N^1(X)_\R, $$
and the negative Tits cone
        $$ \mathcal T(E_8^\smone) = \bigcup_{w\in W(E_8^\smone)} w\big(\overline{\mathcal{F}(E_8^\smone)}\big)\subseteq N^1(X)_\R. $$
By \cite[Proposition 5.8(b)]{Kac90}, by \eqref{eq:cdelta1} and \eqref{eq:cdelta} we have 
\begin{equation}\label{eq:tautological_tits_cone}
\mathcal T(E_8^\smone) = \{x\in N^1(X)_{\R}\mid \delta\cdot x>0 \}\cup\R[\delta].
\end{equation}

\subsection{The subroot system $R$}\label{subsec:root_system_R}

We consider special subroot systems of the root system $E_8^\smone$ above, following \cite{Sak01}.

The \emph{root system $R$} is a root system with the basis
$$ \Pi(R)\coloneqq \{\cal{D}_i\mid \cal{D}_i \text{ is the class of a component of } D\}, $$
and we denote by $Q(R)$ the associated lattice, by $W(R)$ the associated Weyl group, and by $\Delta(R)^\re$ the associated set of real roots. We denote by $\mathcal F(R)$ its negative fundamental chamber and by $\mathcal T(R)$ its negative Tits cone. This is indeed a root system by \cite[p.\ 181]{Sak01}, and it depends on the blowing-down structure of $X$. In Appendix \ref{sec:Sakai_appendix_A}, for each possible blowdown structure in \cite[Appendix B]{Sak01}, a relationship between the classes of a component $\cal{D}_i$ and the tautological basis $\{\mathcal{E}_0,\dots,\mathcal{E}_9\}$ is written down. 

The surfaces $X$ are classified in according to the \emph{type $R$} of $D$. All possible subroot systems $R$ are given in \cite[pp.\ 181--184]{Sak01} and in Appendix \ref{sec:Sakai_appendix_A}:
\begin{align*}
&A_0^\smone,\ A_1^\smone,\ A_2^\smone,\ A_3^\smone,\ A_4^\smone,\ A_5^\smone,\ A_6^\smone,\ A_7^\smone,\ {A_7^\smone}',\\
&A_8^\smone,\ D_4^\smone,\ D_5^\smone,\ D_6^\smone,\ D_7^\smone,\ D_8^\smone,\ E_6^\smone,\ E_7^\smone,\ E_8^\smone.
\end{align*}
There are four additional items in the list in Appendix \ref{sec:Sakai_appendix_A}: $A_0^\smonestar$, $A_0^\smonestarstar$, $A_1^\smonestar$ and $A_2^\smonestar$; however, these are the same as the root systems $A_0^\smone$, $A_1^\smone$ and $A_2^\smone$, respectively, hence they carry no new numerical information; the $*$ refers to different geometric configurations of the components of $D$, and this will be relevant only when we use information from \cite[Appendix B]{Sak01}.

We set $\cal{D}_i^\smvee \coloneqq \cal{D}_i\in N_1(X)_{\R}$ for each class of a component of $D$, and set
$$ A\coloneqq \big( \langle \cal{D}_i,\cal{D}_j^\smvee \rangle \big)_{i,j}. $$
By Lemma \ref{lem:affine_root_data} and Appendix \ref{sec:Sakai_appendix_A}, if $R\neq A_0^\smone$, then
$$
\big(A,\Pi(R),\Pi(R)^\smvee,N^1(X)_\R,N_1(X)_\R,\langle\cdot\, ,\cdot\rangle\big)
$$
is a set of root data of affine type.

\subsection{The set of root data $R^\smperp$}\label{subsec:Rperp0}

Consider the lattice
$$Q(R^\smperp) \coloneqq \{v \in N^1(X) \mid v \cdot  \cal{D}_i =0 \text{ for all }i \}, $$
and observe that $Q(R)\cap Q(R^\smperp) = \Z \delta$. We denote by $\Pi(R^\smperp)$ the corresponding set as in Appendix \ref{sec:Sakai_appendix_A}. We consider the pairing $\langle\cdot\, ,\cdot\rangle$ as in \eqref{eq:pairing}.

As we discuss in \S\ref{subsec:Rperp1}, \S\ref{subsec:Rperp2} and \S\ref{subsec:Rperp3} below, for each $ R\notin\{A_8^\smone, D_8^\smone, E_8^\smone \}$ we obtain a set of root data which we denote by $R^\smperp$. We denote by
$$W(R^\smperp)\quad\text{and}\quad \Delta(R^\smperp)^\re$$
the associated Weyl group and the associated set of real roots; we denote by 
$$\mathcal F(R^\smperp)\quad\text{and}\quad \mathcal T(R^\smperp)$$
its negative fundamental chamber and its negative Tits cone. 

\subsection{The set of root data $R^\smperp$: indecomposable cases}\label{subsec:Rperp1}

Continuing the discussion from \S\ref{subsec:Rperp0}, assume now that
$$ R\notin\{A_5^\smone,A_6^\smone, A_8^\smone, D_6^\smone, D_8^\smone, E_8^\smone \}.$$
Then by going through the list in Appendix \ref{sec:Sakai_appendix_A}, we observe that we can write
$$\Pi(R^\smperp) = \{ \alpha_i \mid 0\leq i \leq k \} \quad\text{for some }k,$$
and the set $\Pi(R^\smperp)$ is linearly independent in $N^1(X)_\R$. Then for each $\alpha_i\in\Pi(R^\smperp)$ we define
$$\alpha_i^\smvee\coloneqq \begin{cases}
\alpha_i\in N_1(X)_\R, & \text{if additionally }R\notin \{A_7^\smone,D_7^\smone\},\\
\frac{1}{4}\alpha_i\in N_1(X)_\R, & \text{if }R=A_7^\smone,\\
\frac{1}{2}\alpha_i\in N_1(X)_\R, & \text{if }R=D_7^\smone.
\end{cases}$$
With these choices, it is easy to check that the associated matrix
$$A\coloneqq \big(\langle \alpha_i,\alpha_j^\smvee\rangle\big)_{0\leq i,j\leq k} $$
is an indecomposable symmetric generalised Cartan matrix of affine type.

We will see below in Proposition \ref{prop:subrootdata000} that 
$$\big(A,\Pi(R^\smperp),\Pi(R^\smperp)^\smvee,N^1(X)_\R,N_1(X)_\R,\langle\cdot\, ,\cdot\rangle\big)$$
is a set of root data. Then by Lemma \ref{lem:affine_root_data} and Appendix \ref{sec:Sakai_appendix_A}, this set of root data is of affine type.

\subsection{The set of root data $R^\smperp$: decomposable cases}\label{subsec:Rperp2}

Continuing the discussion from \S\ref{subsec:Rperp0}, assume now that
$$ R\in\{A_5^\smone,A_6^\smone, D_6^\smone \}.$$
Then by going through the list in Appendix \ref{sec:Sakai_appendix_A}, we observe that we can write
$$\Pi(R^\smperp) = \{ \alpha_i \mid 0\leq i \leq k \} \cup \{\alpha'_j \mid 0\leq j\leq \ell \}\quad\text{for some }k\text{ and }\ell.$$
With this notation, if we denote 
$$\Pi\coloneqq \{ \alpha_i \mid 0\leq i \leq k \} \quad\text{and}\quad\Pi'\coloneqq \{\alpha'_j \mid 0\leq j\leq \ell \},$$
then the sets $\Pi$ and $\Pi'$ are linearly independent in $N^1(X)_\R$.

We will need in the proof of Proposition \ref{prop:subrootdata000} the following additional information. Denote $\widetilde\Pi\coloneqq \Pi(R^\smperp)\setminus\{\alpha'_0\}$. Then it is easy to check that the set $\widetilde\Pi$ is linearly independent in $N^1(X)_\R$ and that
\begin{equation}\label{eq:113a}
\alpha'_0 = \sum_{i=0}^k\alpha_i - \sum_{j=1}^\ell \alpha'_j.
\end{equation}

Now, for each $\alpha_i\in\Pi$ we define
$$(\alpha_j)^\smvee\coloneqq \begin{cases}
\alpha_j\in N_1(X)_\R, & \text{if }R\in \{A_5^\smone,D_6^\smone\},\\
\frac{1}{7}\alpha_j\in N_1(X)_\R, & \text{if }R=A_6^\smone.
\end{cases}$$
For each $\alpha_j'\in\Pi'$ we define 
$(\alpha_j')^\smvee\coloneqq \alpha'_j$.
With these choices, it is easy to check that the associated matrices
$$B\coloneqq \big(\langle \alpha_i,\alpha_j^\smvee\rangle\big)_{0\leq i,j\leq k} \quad\text{and}\quad B'\coloneqq\big( \big\langle \alpha_i',(\alpha'_j)^\smvee\big\rangle \big)_{0\leq i,j\leq \ell} $$
are indecomposable symmetric generalised Cartan matrices of affine type.

Furthermore, it is easy to check that for all $\gamma\in \Pi$ and $\gamma'\in \Pi'$ we have $\big\langle\gamma,(\gamma')^\smvee\big\rangle = \langle\gamma',\gamma^\smvee\big\rangle = 0$. Therefore, the direct sum matrix
$$A\coloneqq B \oplus B' $$
is a symmetric generalised Cartan matrix associated to $\Pi(R^\smperp)$. We will see below in Proposition \ref{prop:subrootdata000} that 
$$\big(A,\Pi(R^\smperp),\Pi(R^\smperp)^\smvee,N^1(X)_\R,N_1(X)_\R,\langle\cdot\, ,\cdot\rangle\big)$$
is a decomposable system of root data. We denote by
$$\mathscr D\coloneqq \big(B,\Pi,\Pi^\smvee,N^1(X)_\R,N_1(X)_\R,\langle\cdot\, ,\cdot\rangle\big)$$
and
$$\mathscr D'\coloneqq \big(B',\Pi',(\Pi')^\smvee,N^1(X)_\R,N_1(X)_\R,\langle\cdot\, ,\cdot\rangle\big)$$
the corresponding root systems. Then by Lemma \ref{lem:affine_root_data} and Appendix \ref{sec:Sakai_appendix_A}, the sets of root data $\mathscr{D}$ and $\mathscr{D}'$ are of affine type.

\subsection{The set of root data $R^\smperp$: properties}\label{subsec:Rperp3}

Now we can prove the properties of $R^\smperp$ that we announced above.
 
\begin{prop}\label{prop:subrootdata000}
    Let $X$ be a Sakai surface of type $R$, and assume that $ R\notin\{A_8^\smone, D_8^\smone, E_8^\smone \} $. Denote $V(R)\coloneqq Q(R)\otimes_\Z\R\subseteq N^1(X)_\R$. Then, with notation from \S\ref{subsec:Rperp0}, \S\ref{subsec:Rperp1} and \S\ref{subsec:Rperp2}, we have:
\begin{enumerate}[\normalfont (a)]
\item $\big(A,\Pi(R^\smperp),\Pi(R^\smperp)^\smvee,N^1(X)_\R,N_1(X)_\R,\langle\cdot\, ,\cdot\rangle\big)$ is a set of root data,
\item for each $\alpha\in \Delta(R^\smperp)^\re$ there exists a positive rational number $\varepsilon_\alpha$ such that $\alpha^\smvee=\varepsilon_\alpha \alpha$,
\item we have $\{x\in N^1(X)_\R \mid \delta\cdot x>0\} \cup V(R) \subseteq \mathcal{T}(R^\smperp)$, and in particular,
    $$\mathcal{T}(E_8^\smone)\subseteq \mathcal{T}(R^\smperp).$$
\end{enumerate}
\end{prop}

\begin{proof}
First note that the second statement of (c) follows from the first statement of (c) and from \eqref{eq:tautological_tits_cone}. In the remainder of the proof we show (a), (b) and the first statement of (c).

\medskip

\emph{Step 1.}
In this step we show (a). The condition (i) in Definition \ref{dfn:root_data} follows from the construction. Further, denote 
$$\mathcal Q\coloneqq \sum_{\alpha\in \Pi(R^\smperp)}\Z \alpha.$$
It is immediate that $\mathcal Q$ is a lattice in $\mathcal Q\otimes_\Z \R$ when $A$ is indecomposable by the discussion in \S\ref{subsec:Rperp1}. If $A$ is decomposable, then, with notation from \S\ref{subsec:Rperp2}, the equation \eqref{eq:113a} implies that $\mathcal Q=\sum\limits_{\alpha\in \widetilde\Pi}\Z \alpha$, hence $\mathcal Q$ is a lattice in $\mathcal Q\otimes_\Z \R$ by the discussion in \S\ref{subsec:Rperp2}, and the condition (ii) in Definition \ref{dfn:root_data} holds.

In order to show that the condition (iii) in Definition \ref{dfn:root_data} holds, we follow verbatim the proof of \cite[Theorem 6 on p.\ 688]{MP89}. Set 
$$\mathcal C \coloneqq \sum_{\alpha\in \Pi(R^\smperp)}\R_+ \alpha\subseteq \mathcal Q\otimes_\Z \R.$$
Since $\mathcal Q$ is a lattice in $\mathcal Q\otimes_\Z \R$ by the previous paragraph, and since $\mathcal C\subseteq\mathcal Q\otimes_\Z \R$ is a strictly convex polyhedral cone, \cite[Lemma 5 on p.\ 688]{MP89} yields that there exists a basis $\{v_1,\dots,v_m\}$ of $\mathcal Q$ such that $\mathcal C \subseteq \sum_{i=1}^m \R_+ v_i$. Therefore,
$$ \Pi(R^\smperp)\subseteq \mathcal C\cap\mathcal Q \subseteq \Big(\sum_{i=1}^m \R_+ v_i\Big)\cap \Big( \bigoplus_{i=1}^m \Z v_i \Big)= \bigoplus_{i=1}^m \N v_i, $$
which verifies the condition (iii) in Definition \ref{dfn:root_data}. This finishes the proof of (a).

\medskip

\emph{Step 2.}
In this step we show (b).

With notation from \S\ref{subsec:Rperp1} and \S\ref{subsec:Rperp2}, it is immediate that
$$\mathscr D\coloneqq \big(B,\Pi,\Pi^\smvee,N^1(X)_\R,N_1(X)_\R,\langle\cdot\, ,\cdot\rangle\big)$$
and
$$\mathscr D'\coloneqq \big(B',\Pi',(\Pi')^\smvee,N^1(X)_\R,N_1(X)_\R,\langle\cdot\, ,\cdot\rangle\big)$$
are sets of root data, since $\Pi$ and $\Pi'$ are linearly independent. By the discussion in \S\ref{subsec:Rperp2}, if $\Pi'\neq\emptyset$, then the set $\Pi(R^\smperp)$ is decomposable with the decomposition $\Pi(R^\smperp)=\Pi\cup\Pi'$; otherwise, $\Pi(R^\smperp)=\Pi$. Therefore, by Proposition \ref{pro:decomposable}(b) we have the disjoint union
$$ \Delta(R^\smperp)^\re = \Delta^\re_{\mathscr D}\cup\Delta^\re_{\mathscr D'}. $$
Thus, if $\alpha\in\Delta(R^\smperp)^\re$, then $\alpha\in\Delta^\re_{\mathscr D}$ or $\alpha\in\Delta^\re_{\mathscr D'}$. 

Assume that $\alpha\in\Delta^\re_{\mathscr D}$; the other case is analogous. Then there exist $w\in W_{\mathscr D}$ and $\alpha_i\in\Pi$ such that $\alpha=w(\alpha_i)$. By \S\ref{subsec:Rperp1} and \S\ref{subsec:Rperp2}, there exists a positive rational number $\varepsilon_i$ such that $\alpha_i^\smvee=\varepsilon_i\alpha_i$, hence
$$\alpha^\smvee=w(\alpha_i^\smvee)=w(\varepsilon_i\alpha_i)=\varepsilon_i w(\alpha_i)=\varepsilon_i\alpha,$$
as desired.

\medskip

\emph{Step 3.}
From now on we show (c). In this step we give a description of the set $\Delta(R^\smperp)^\re_+$.

By Proposition \ref{pro:decomposable}(c) we have the disjoint union
\begin{equation}\label{eq:disjointunion}
\Delta(R^\smperp)^\re_+ = (\Delta^\re_{\mathscr D})_+\cup(\Delta^\re_{\mathscr D'})_+.
\end{equation}
With notation from \S\ref{subsec:removingroot}, denote $\Delta^\circ=(\Delta^\re_{\mathscr D})^\circ$, $(\Delta')^\circ=(\Delta^\re_{\mathscr D'})^\circ$, and
$$S\coloneqq \{ \beta + n\delta \mid \beta\in \Delta^\circ, n\in \Z_{>0} \} \cup \{ \beta' + n\delta \mid \beta'\in (\Delta')^\circ, n\in \Z_{>0}\}.$$
Then \eqref{eq:disjointunion} and Proposition \ref{prop:rootsafterremoving} give
    \begin{equation}\label{eq:union_real_roots}
     \Delta(R^\smperp)^\re_+ = S \cup \Delta^\circ_+ \cup (\Delta')^\circ_+.
    \end{equation}
    
    \medskip
    
    \emph{Step 4.}
In this step we show the first statement of (c).
    
    Fix $y\in\{x\in N^1(X)_\R \mid \delta\cdot x>0\} \cup V(R)$ and assume, for contradiction, that $y\notin \mathcal T(R^\smperp)$. Then by Proposition \ref{prop:titscone} there are infinitely many $\alpha\in \Delta(R^\smperp)^\re_+$ such that $\langle y,\alpha^\smvee \rangle > 0$. Since $\langle y,\alpha^\smvee \rangle={-}y\cdot\alpha^\smvee $, we conclude by (b) that that there are infinitely many $\alpha\in \Delta(R^\smperp)^\re_+$ such that $y\cdot\alpha < 0$. Since $\Delta^\circ_+$ and $(\Delta')^\circ_+$ are finite by Proposition \ref{prop:rootsafterremoving}, by \eqref{eq:union_real_roots} we conclude that
\begin{equation}\label{eq:inifinitelymany}
\text{there are infinitely many }\alpha\in S\text{ such that }y\cdot\alpha < 0.
\end{equation}    

    Now there are two cases. If $y\in\{x\in N^1(X)_\R \mid \delta\cdot x>0\}$, then for each $\beta\in\Delta^\circ\cup(\Delta')^\circ$ and for each $n\gg0$ we have $y\cdot(\beta + n \delta) >0$. Since $\Delta^\circ$ and $(\Delta')^\circ$ are finite by Proposition \ref{prop:rootsafterremoving}, we infer that there exists a positive integer $N$ such that $y\cdot(\beta + n \delta) >0$ for each $\beta\in\Delta^\circ\cup(\Delta')^\circ$ and for each $n\geq N$. But this contradicts \eqref{eq:inifinitelymany}.

Finally, if $y\in V(R)$, then $y\cdot \alpha = 0$ for any $\alpha\in Q(R^\smperp)$. In particular, $y\cdot \alpha =0 $ for any $\alpha\in S$, which again contradicts \eqref{eq:inifinitelymany}. This concludes the proof.
\end{proof}

\subsection{An auxiliary lemma}

We need the following lemma in Section \ref{sec:finiteness}.

\begin{lem}\label{lem:somestuff}
Let $X$ be a Sakai surface of type $R$, and assume that $ R\notin\{A_8^\smone, D_8^\smone, E_8^\smone \} $. Let $\cal{D}_0,\ldots, \cal{D}_z$ be the numerical classes of irreducible components of $D$, and let $\alpha_0,\ldots,\alpha_u,\alpha_0',\dots,\alpha_v'$ be the elements of $\Pi(R^\smperp)$ as in Appendix \ref{sec:Sakai_appendix_A}.\footnote{There exists elements $\alpha_i'$ only when $R\in\{A_5^\smone, A_6^\smone, D_6^\smone\}$.} Denote $\mathcal S:=\{\alpha_1,\dots,\alpha_u,\alpha_1',\dots,\alpha_v',\cal{D}_1,\dots,\cal{D}_z\}$. Then:
\begin{enumerate}[\normalfont (a)]
\item $u+v+z=8$,
\item $\alpha_i\cdot \cal{D}_j = \alpha_k'\cdot \cal{D}_j = \alpha_i\cdot\alpha_k' = 0$ for all $i\geq0$, $j\geq0$ and $k\geq0$,
\item $\alpha_0\cdot \mathcal{E}_9 = \alpha_0'\cdot \mathcal{E}_9 =1$ and $\alpha_i \cdot \mathcal{E}_9 = \alpha_i'\cdot \mathcal{E}_9 = 0$ for $i\geq1$,
\item $\cal{D}_0\cdot \mathcal{E}_9 =1$ and $\cal{D}_j \cdot \mathcal{E}_9 = 0$ for $j\geq 1$,
\item $\alpha_i\cdot \delta = \alpha_i'\cdot \delta = \cal{D}_i \cdot \delta = 0$ for all $i$,
\item the set $\mathcal S$ is linearly independent,
\item $\mathcal S\subseteq \sum_{i=1}^{8} \Z \rho_i $.
\end{enumerate}
\end{lem}

\begin{proof}
Parts (a), (b), (c), (d) and (e) are easily verified. Part (f) follows by using (b) and by observing that the sets $\{\alpha_1,\dots,\alpha_u\}$, $\{\alpha_1',\dots,\alpha_v'\}$ and $\{\cal{D}_1,\dots,\cal{D}_z\}$ are linearly independent.

    Now we show (g). We first claim that any element in $\mathcal S$ is an integral linear combination of elements from the set
    $$\mathcal U:=\{\cal{E}_i-\cal{E}_j\mid 1\leq i<j\leq 8 \} \cup \{\cal{E}_0-\cal{E}_i-\cal{E}_j-\cal{E}_k\mid 1\leq i<j<k\leq 8 \}.$$
    Indeed, this is easy to see by analysing most cases in Appendix \ref{sec:Sakai_appendix_A}. The only case where the claim is not obvious is when $R=A_7^\smone$: then $\alpha_1 = (\mathcal{E}_0-\mathcal{E}_2-\mathcal{E}_3-\mathcal{E}_6)+(\mathcal{E}_0-\mathcal{E}_4-\mathcal{E}_5-\mathcal{E}_6)+(\mathcal{E}_1-\mathcal{E}_2)$.
    
    Next, we claim that any element in $\mathcal U$ is an integral linear combination of elements from the set
    $$\mathcal U':=\{\cal{E}_i-\cal{E}_j\mid 1\leq i<j\leq 8 \}\cup\{\rho_8\}.$$
    Indeed, this follows since for any $1\leq i<j<k\leq 8$ we have $\cal{E}_0-\cal{E}_i-\cal{E}_j-\cal{E}_k = \rho_8+(\cal{E}_1-\cal{E}_i)+(\cal{E}_2-\cal{E}_j)+(\cal{E}_3-\cal{E}_k)$.
   
    Therefore, it suffices to show that $\mathcal U' \subseteq \sum_{i=1}^{8} \Z \rho_i $. But this follows since for all $1\leq i<j\leq 8$ we have $ \cal{E}_i-\cal{E}_j = \rho_i+\dots +\rho_{j-1}\in \sum_{i=1}^{8} \Z \rho_i $.
\end{proof}

\section{Cones of divisors and Weyl groups on Sakai surfaces}

In this section we prove a lemma which relates several cones of divisors on Sakai surfaces, and a lemma which explains how two Weyl groups act on these cones. These results are indispensable for the proof of our main results.

We start with the following lemma which relates several cones introduced in Sections \ref{sec:prelim} and \ref{sec:rootsystemsHalphen}.

\begin{lem}\label{lem:effconeHalpensurface1}
Let $X$ be a Sakai surface of type $R$, and let $D\in|{-}K_X|$. Then:
\begin{enumerate}[\normalfont (a)]
\item $\Eff(X)=\overline{\Eff}(X)$,
\item $\Nef(X)=\Nef^e(X)=\Nef^+(X)$,
\item $\Eff(X)$ is the convex hull of ${\mathcal M}^{\irr}_X\cup\Comp(D)\cup\Delta_{X,D}^{\nod}$,
\item if $X$ is generic, then ${\mathcal M}^{\irr}_X=\mathcal N_X$, and $\Eff(X)$ is the convex hull of $\mathcal N_X\cup\Comp(D)$,
\item $\mathcal{N}_X\subseteq \mathcal M_X \subseteq \mathcal T(E_8^\smone)\subseteq \mathcal{T}(R^\smperp)$,
\item $\Nef(X)\subseteq \mathcal T(E_8^\smone)\subseteq \mathcal{T}(R^\smperp)$.
\end{enumerate}
\end{lem}

\begin{proof}
    Parts (a) and (b) follow from \cite[Theorem 3.1(b)]{LSX26}. If $X$ is generic, then ${\mathcal M}^{\irr}_X=\mathcal N_X$ by Lemma \ref{lem:irredclasses}, hence (d) follows from (c).
    
	Now we show (c). By \cite[Lemma 4.1]{LH05}, the cone $\Eff(X)$ is spanned by the classes of prime divisors on $X$ with negative self-intersection and by $[{-}K_X]$. By \cite[Lemma 2.5]{LSX26}, a prime divisor with negative self-intersection, which is not in $\Comp(D)$, is either a $(-1)$-curve or an element of $\Delta_{X,D}^{\nod}$, which was to be proved.
    
    Next we show (e). We have $\mathcal{N}_X\subseteq\mathcal{M}_X$ by definition. For each $x\in\mathcal M_X$ we have Since $-K_X\cdot x=1$, hence $x\in \cal{T}(E^{\smone}_8)$ by \eqref{eq:tautological_tits_cone}. The last inclusion follows from Proposition \ref{prop:subrootdata000}(c).
    
    Finally, we show (f). Let $x \in \Nef(X)$. Since $-K_X$ is nef, we have $-K_X\cdot x \geq 0$. If $-K_X\cdot x > 0$, then $x\in \cal{T}(E^{\smone}_8)$ by \eqref{eq:tautological_tits_cone}. Now assume that $-K_X \cdot x = 0$. Since $x$ is nef, we have $x^2 \geq 0$, hence $x \in\R[K_X]$ by the proof of \cite[Lemma 2.2]{LH05}. Consequently, $x \in \mathcal{T}(E^\smone_8)$ by \eqref{eq:tautological_tits_cone}. This, together with Proposition \ref{prop:subrootdata000}(c), concludes the proof.
\end{proof}
   	
The next lemma shows several properties of the actions of the Weyl groups $W(E_8^\smone)$ and $W(R^\smperp)$ on several cones which are of interest to us.

\begin{lem}\label{lem:weylpreservesnefgen1}
   Let $X$ be a Sakai surface of type $R$, and let $D\in|{-}K_X|$. Then:
    \begin{enumerate}[\normalfont (a)]
	\item the group $W(E_8^\smone)$ acts transitively on the set $\mathcal M_X$,
    \item $\Cr(X,D)$ preserves $\mathcal M_X$ and $\mathcal{N}_X$,
    \item $W(R^\smperp)$ fixes $\Comp(D)$ pointwise,
    \item for every $\gamma \in \cal N _X$ and every $w \in W(R^\smperp)$ we have that $w(\gamma) \in {\cal N_X}$ if and only if $w(\gamma) \in N^1(X)$.
    \end{enumerate}
    Assume, furthermore, that $R\notin \{A_6^\smone, A_7^\smone, A_8^\smone, D_7^\smone, D_8^\smone, E_8^\smone\}$. Then:
    \begin{enumerate}[\normalfont (a)]
    \item[\normalfont{(e)}] $W(R^\smperp)$ preserves $\mathcal M_X$ and $\mathcal{N}_X$,
    \item[\normalfont{(f)}] for each $\gamma \in \mathcal{N}_X$ there exists $w \in W(R^\smperp)$ such that $ w(\gamma) \in \overline{\mathcal F(R^\smperp)}\cap \mathcal{N}_X $. 
    \end{enumerate}
\end{lem}

\begin{proof}
Let $w\in W(E_8^\smone)$ and let $x\in \mathcal M_X$. By Proposition \ref{prop:W_invariant_pairing} and by \eqref{eq:orthogonaltodelta} we have $w(x)^2=w(x)\cdot[K_X]={-}1$. As $w$ preserves the lattice $N^1(X)$, we have $w(x)\in N^1(X)$, hence $w(x)\in\mathcal M_X$ by Lemma \ref{lem:-1classesHalphen}(a). Now (a) follows from \cite[Lemma 14]{Sak01}. 

Next we show (b). The first statement in (b) follows as in the previous paragraph. For the second statement, let $x \in \mathcal{N}_X$. Then $x \in \mathcal M_X$ and $x\cdot\mathcal D \geq 0$ for each $\mathcal D\in\Comp(D)$. We have $w(x) \in \mathcal M_X$ by the first statement in (b), and $w(\mathcal D) = \mathcal D $ as $w\in\Cr(X,D)$. Therefore, since $w$ preserves the intersection product, we have $ w(x) \cdot\mathcal D = w(x) \cdot w(\mathcal D) =  x\cdot\mathcal D \geq 0 $, hence $w(x) \in \mathcal{N}_X $, as desired.

Now we show (c). Let $\alpha\in \Pi(R^\smperp)$ and let $\mathcal D\in\Comp(D)$. Then $\alpha\cdot\mathcal D=0$, hence $ r_\alpha(\mathcal D) = \mathcal D + (\mathcal D\cdot \alpha)\alpha = \mathcal D $. Since every element in $W(R^\smperp)$ is a word in the elements $r_\alpha$ for $\alpha\in \Pi(R^\smperp)$, this proves (c).

Next we prove (d). One direction is clear. For the converse, assume that $w(\gamma) \in N^1(X)$. Since $w\in W(R^\smperp)$, the element $w$ fixes the class $[{-}K_X]$, and $w$ fixes $\Comp(D)$ pointwise by (c). Then as in the proof of (b) above we have $w(\gamma) \in {\cal M_X}$ and $ w(\gamma) \cdot\mathcal D \geq 0 $ for any $\mathcal D \in  \Comp(D)$. This shows that $w(\gamma) \in {\cal N_X}$.

Now we prove (e). Since every element in $W(R^\smperp)$ is a word in the elements $r_\alpha$ for $\alpha\in \Pi(R^\smperp)$, it suffices to show that each $r_\alpha$ preserves $\mathcal M_X$ and $\mathcal N_X$. To that end, fix $\alpha\in \Pi(R^\smperp)$. For any $x\in \mathcal M_X$, by Proposition \ref{prop:W_invariant_pairing} we have $ r_\alpha(x)^2 = x^2 = {-}1 $. Moreover, since $\alpha\in \Pi(R^\smperp)$, we have $\alpha\cdot K_X=0$, and thus $ r_\alpha(x)\cdot K_X = x\cdot K_X + (x\cdot \alpha)(\alpha\cdot K_X) = {-}1$. Since $R\notin \{A_6^\smone, A_7^\smone, A_8^\smone, D_7^\smone, D_8^\smone, E_8^\smone\}$, we can check easily that $r_\alpha$ preserves the lattice $N^1(X)$, hence $r_\alpha(x)\in N^1(X)$. By Lemma \ref{lem:-1classesHalphen} we conclude that $r_\alpha(x)\in \mathcal M_X$, which shows the first statement in (e). To show the second statement, note that as in the proof of (b) above, for any $x\in\mathcal N_X$ and any $\mathcal D \in  \Comp(D)$ we have by (c) that $ r_\alpha(x) \cdot\mathcal D \geq 0 $. Hence, $r_\alpha(x) \in \mathcal{N}_X $.
    
Finally, we show (f). Since $\mathcal{N}_X \subseteq \mathcal T(R^\smperp)$ by Lemma \ref{lem:effconeHalpensurface1}(e), we have by Proposition \ref{prop:Titsfundchamber} that there exists $w \in W(R^\smperp)$ such that $ w(\gamma) \in \overline{\mathcal F(R^\smperp)}$. But $ w(\gamma) \in \mathcal N_X$ by (e), which concludes the proof.
\end{proof}

The following corollary will be crucial in Section \ref{sec:finiteness}.

	  	\begin{lem}\label{lem:translationsatE9}
	  	Let $X$ be a Sakai surface, and assume notation from \S\ref{subsec:translations1}. Then for each $\gamma\in \mathcal{M}_X$ there exists a unique $\alpha \in Q(E_8)$ such that $ \gamma=t_\alpha(\mathcal E_9)$.
	  	\end{lem}
	  	
	  	\begin{proof}
		Since $\mathcal E_9\in\mathcal M_X$, by Lemma \ref{lem:weylpreservesnefgen1}(a) we have that for each $\gamma\in \mathcal{M}_X$ there exists $w\in W(E_8^\smone)$ with $\gamma=w(\mathcal E_9)$. By Proposition \ref{prop:translations} there exist unique $\alpha \in Q(E_8^\smone)^\circ=Q(E_8)$ and $w' \in W(E_8^\smone)^\circ=W(E_8)$ such that $w =  t_\alpha\circ w'$. As each fundamental reflection $r_{\rho_i}$, for $i\in \{1,\dots,8\}$, fixes $\mathcal E_9$, we have that $w'(\mathcal E_9)=\mathcal E_9$, hence
	    $$ \gamma=w(\mathcal{E}_9) = t_\alpha\big(w'(\mathcal{E}_9)\big) = t_\alpha(\mathcal{E}_9). $$
	    By Lemma \ref{lem:translationsatE9a}, the element $\alpha \in Q(E_8)$ is uniquely determined by $t_\alpha(\mathcal{E}_9)$, which finishes the proof.
	  	\end{proof}

\newpage

\part{}\label{part:2}

Let $X$ be a generic Sakai surface and let $D\in|{-}K_X|$. In this Part \ref{part:2} of the paper, we prove our main results, Theorems \ref{thm:main2} and \ref{thm:mainB}. 

The proof of Theorem \ref{thm:main2} is very technical, but the logic of the proof is quite natural. In order to convey the main ideas of the proof, we now present its main steps.

Let $R$ be the type of $X$. We may assume that $R\notin\{A_8^\smone, D_8^\smone, E_8^\smone\}$, since otherwise the proof of Theorem \ref{thm:main2} is easy, see the proof of Theorem \ref{thm:existencefunddom}(b). We recall now that by Lemma \ref{lem:effconeHalpensurface1}(f) we have
$$\Nef(X) \subseteq \mathcal{T}(R^\smperp),$$
and we know by Proposition \ref{prop:Titsfundchamber} that $\overline{\mathcal F(R^\smperp)}$ is a fundamental domain for the action of $ W(R^\smperp) $ on $\mathcal{T}(R^\smperp)$. Therefore, if one could additionally show that $ \Cr(X,D)=W(R^\smperp) $, then it is natural to expect that the cone $\overline{\mathcal F(R^\smperp)}\cap\Nef(X)$ is a fundamental domain for the action of $ \Cr(X,D) $ on $\Nef(X)$.

This is indeed the case when $R \notin \{A_6^\smone, A_7^\smone,D_7^\smone\}$. Then $\Cr(X,D) = W(R^\smperp)$ by \cite[Theorem 26(a)]{Sak01}, and it is easy to show that $\overline{\mathcal F(R^\smperp)}\cap\Nef(X)$ is a fundamental domain for the action of $ \Cr(X,D) $ on $\Nef(X)$. Thus, the main issue is to prove that the cone $\overline{\mathcal F(R^\smperp)}\cap\Nef(X)$ is rational polyhedral: this is done in Proposition \ref{prop:ratpol1}. The crucial ingredient in the proof is Corollary \ref{cor:finite_(-1)_curves_chamber1}(c), where we show that the classes of only finitely many $(-1)$-curves are contained in $\overline{\mathcal F(R^\smperp)}$; in particular, the group $\Cr(X,D)$ acts on the set of $(-1)$-curves with only finitely many orbits. We then show that the faces of the cone $\overline{\mathcal F(R^\smperp)}\cap\Nef(X)$ are determined by these finitely many $(-1)$-curves and by the finitely many $(-2)$-curves on $X$; here the proof relies crucially on Lemma \ref{lem:actionW}, which exploits the properties of the reflection group $W(R^\smperp)$. Therefore, the cone $\overline{\mathcal F(R^\smperp)}\cap\Nef(X)$  has finitely many faces and is thus rational polyhedral.

When $R \in \{A_6^\smone, A_7^\smone,D_7^\smone\}$, then the situation is much more subtle, and a large part of the proof in Part \ref{part:2} deals with these three exceptional cases. As a first step, we determine a fundamental domain for the action of $\Cr(X,D)$ on $\mathcal{T}(R^\smperp)$. The issue is that in those cases, the group $W(R^\smperp)$ preserves neither the lattice $N^1(X)$ nor the cone $\Nef(X)$, and by \cite[Theorem 26(c)(d)(e)]{Sak01}, the group $\Cr(X,D)$ is a subgroup of $W(R^\smperp)$ of index $2m$, where for each such $R$, the positive integer $m$ is introduced in Setup \ref{setup:allcases} below. Therefore, it is natural to expect that a fundamental domain for the action of $\Cr(X,D)$ on $\mathcal{T}(R^\smperp)$ can be obtained as a suitable union of $2m$ translates of the fundamental chamber $\overline{\mathcal F(R^\smperp)}$. Indeed this is the case: in Lemma \ref{lem:OmegaR} and Proposition \ref{prop:fund_domain_Cr(X,D)_Tits_cone} we show that there exist elements $v_1,\ldots,v_{2m}\in W(R^\smperp)$ such that the rational polyhedral cone
\begin{equation}\label{eq:Omega_union}
\textstyle \Omega_R \coloneqq \bigcup_{i=1}^{2m}v_i\big(\overline{\mathcal F(R^\smperp)}\big)
\end{equation}
is a fundamental domain for the action of $\Cr(X,D)$ on $\mathcal{T}(R^\smperp)$. The main issue again is to prove that the cone $\Omega_R\cap\Nef(X)$ is rational polyhedral: this is done in Proposition \ref{prop:ratpoly}. We follow the same strategy as in the previous paragraph: we first show that there are only finitely many $(-1)$-curves contained in $\Omega_R$; this is established in Theorem \ref{thm:mainthmsection8}. Unlike in the previous paragraph, we cannot apply Lemma \ref{lem:actionW} directly to $\Cr(X,D)$, since now $\Cr(X,D)$ is not generated by reflections. To overcome this problem, we introduce auxiliary reflection subgroups $\widehat{W}_{\mathscr{D}_{k,m-k-1}}\subseteq W(R^\smperp)$ for $k=0,\ldots,m-1$; see Setup \ref{setup:allcases}. A crucial fact, which we prove in Theorem \ref{thm:mainthmsection8}, is that the set of classes of $(-1)$-curves on $X$ decomposes as the union of orbits of these reflection subgroups $\widehat{W}_{\mathscr{D}_{k,m-k-1}}$. Using this decomposition and Lemma \ref{lem:actionW}, we show in Proposition \ref{prop:ratpol_ugly_cases} that the cone $w\big(\overline{\mathcal F(R^\smperp)}\big) \cap \Nef(X)$ is rational polyhedral for every $w \in W(R^\smperp)$.  By \eqref{eq:Omega_union}, it finally follows that $\Omega_R \cap \Nef(X)$ is rational polyhedral as well. 

\medskip

In the following Setup \ref{setup:allcases} we set the notation for Part \ref{part:2}.

\begin{setupI}\label{setup:allcases}
Let $X$ be a Sakai surface of type $R$. We use the notation from Section \ref{sec:rootsystemsHalphen}. By Proposition \ref{prop:subrootdata000} we have a set of root data
$$ \big(A,\Pi(R^\smperp),\Pi(R^\smperp)^\smvee,N^1(X)_\R,N_1(X)_\R,\langle\cdot\, ,\cdot\rangle\big). $$
We set
$$ m\coloneqq \begin{cases}
7, & \text{if }R=A_6^\smone,\\
4, & \text{if }R=A_7^\smone,\\
2, & \text{if }R = D_7^\smone.\\
\end{cases} $$
\begin{enumerate}[\normalfont (a)]
\item Assume that $ R\notin\{A_6^\smone, A_7^\smone, A_8^\smone, D_7^\smone, D_8^\smone, E_8^\smone \} $. If $X$ is generic, by \cite[Theorem 26(a)]{Sak01} we have $\Cr(X,D)=W(R^\smperp)$. We set
$$ \Omega_R\coloneqq \overline{\cal F(R^\smperp)}. $$
\item Assume that $R=A_6^\smone$. We use the notation from \S\ref{subsec:Rperp2}. Then $\Pi(R^\smperp) = \{ \alpha_0,\alpha_1\} \cup \{\alpha'_0,\alpha'_1 \}$, and we have root data $\mathscr D$ and $\mathscr D'$ of affine type $A_1^\smone$. Denote $r_0\coloneqq r_{\alpha_0}$, $r_1\coloneqq r_{\alpha_1}$, $r_0'\coloneqq r_{\alpha_0'}$ and $r_1'\coloneqq r_{\alpha_1'}$. Then, if we define $\sigma$ as in \cite[Theorem 26(c)]{Sak01}, a calculation shows that $\sigma^{14} = (r_0r_1)^7 = (r_0r_1)^m$, hence if $X$ is generic, by loc.\ cit.\ we have\footnote{Note that there is a typo in \cite[Theorem 26(c)]{Sak01}, where $\sigma^{7n}$ should be $\sigma^{14n}$; the reason is that $\sigma^{14}$ is the smallest power of $\sigma$ which is in $\Aut\big(N^1(X)\big)$ and not only in $\Aut\big(N^1(X)_\Q\big)$.}
$$\Cr(X,D)=\big\langle (r_0r_1)^m\big\rangle\cdot \langle r_0',r_1'\rangle \subseteq W(R^\smperp).$$
As in Setup \ref{setup2}, consider the subroot system $\mathscr D_{m,m-1}$ of $\mathscr D$, and set
$$ \Omega_R \coloneqq \overline{\cal F_{{\mathscr D}_{m,m-1}}}\cap \overline{\cal F_{\mathscr D'}}. $$
For every $k\in \{ 0,\dots, m-1\}$, we denote $\widehat W_{\mathscr D_{k,m-k-1}}\coloneqq W_{\mathscr D_{k,m-k-1}}\cdot W_{\mathscr{D}'}$ and $\widehat{\mathcal F}_{\mathscr D_{k,m-k-1}}\coloneqq \overline{\mathcal F_{\mathscr D_{k,m-k-1}}}\cap \overline{\mathcal F_{\mathscr D'}}$.

\item Assume that $R\in\{A_7^\smone,D_7^\smone\}$. We use the notation from \S\ref{subsec:Rperp1}. Then $\Pi(R^\smperp) = \{ \alpha_0,\alpha_1\} $, and we have root data $\mathscr D$ of affine type $A_1^\smone$. Denote $r_0\coloneqq r_{\alpha_0}$ and $r_1\coloneqq r_{\alpha_1}$. If $R=A_7^\smone$ and if we define $\sigma$ as in \cite[Theorem 26(d)]{Sak01}, a calculation shows that $\sigma^8 = (r_1r_0)^4 = (r_1r_0)^m$. If $R=D_7^\smone$ and if we define $\sigma$ and $\sigma'$ as in \cite[Theorem 26(e)]{Sak01}, a calculation shows that $(\sigma\circ \sigma')^4 = (r_1r_0)^2 = (r_1r_0)^m$. Hence, if $X$ is generic, by \cite[Theorem 26(d)(e)]{Sak01} we have
$$\Cr(X,D)=\big\langle (r_0r_1)^m\big\rangle\subseteq W(R^\smperp).$$
As in Setup \ref{setup2}, consider the subroot system $\mathscr D_{m,m-1}$ of $\mathscr D$, and set
$$ \Omega_R \coloneqq \overline{\cal F_{{\mathscr D}_{m,m-1}}}. $$
To ease the notation in the proofs of Lemma \ref{lem:alternative_unique_gamma0} and Proposition \ref{prop:ratpol_ugly_cases}, for every $k\in \{ 0,\dots, m-1\}$, we denote $\widehat W_{\mathscr D_{k,m-k-1}}\coloneqq W_{\mathscr D_{k,m-k-1}}$ and $\widehat{\mathcal F}_{\mathscr D_{k,m-k-1}}
\coloneqq \overline{\mathcal F_{\mathscr D_{k,m-k-1}}}$.
\end{enumerate}
\end{setupI}

\section{Finiteness results}\label{sec:finiteness}

The main goal of this section is to prove several finiteness results which will be crucial in the proof of Theorem \ref{thm:main2}.

The section is of very technical nature. As an illustration of the techniques developed here, we obtain the following result which is of independent interest. Recall the set $\mathcal N_X$ defined in Section \ref{sec:prelim}; if $X$ is a generic Sakai surface, then $\cal N_X$ is precisely the set of $({-}1)$-curves on $X$ by Lemma \ref{lem:effconeHalpensurface1}(d).

\begin{thm}\label{thm:mainofsection7}
    Let $X$ be a generic Sakai surface of type $R$, let $D\in|{-}K_X|$, and assume that $ R\notin\{A_8^\smone, D_8^\smone, E_8^\smone \} $. Let $\cal N_X$ be the set of $({-}1)$-curves on $X$. Then:
\begin{enumerate}[\normalfont (a)]
\item the set $\mathcal{N}_X \cap \overline{\mathcal F(R^\smperp)}$ is finite,
\item if, furthermore, $R\notin\{A_6^\smone, A_7^\smone, D_7^\smone \} $, then $\Cr(X,D)$ acts on $\cal N_X$ with finitely many orbits.
\end{enumerate}
\end{thm}

If $X$ is a generic Sakai surface as in Theorem \ref{thm:mainofsection7}(b), then $\Cr(X,D)=W(R^\smperp)$ by \cite[Theorem 26(a)]{Sak01}. Therefore, Theorem \ref{thm:mainofsection7} follows immediately from Corollary \ref{cor:finite_(-1)_curves_chamber1}(c)(d).

In the remainder of this section, we consider Sakai surfaces of type $ R\notin\{A_8^\smone, D_8^\smone, E_8^\smone \} $ and we assume Setup \ref{setup:allcases}. As in \S\ref{subsec:translations1}, if we remove the vertex corresponding to $\rho_0$ from the Dynkin diagram of $E_8^\smone$, we obtain the finite Dynkin diagram $E_8$.

As in Appendix \ref{sec:Sakai_appendix_A}, we denote by $\cal{D}_0,\ldots, \cal{D}_z$ the classes of an irreducible component of $D$, and by $\alpha_0,\ldots,\alpha_u,\alpha_0',\dots,\alpha_v'$ the elements of $\Pi(R^\smperp)$. Then by Lemma \ref{lem:somestuff}(a)(f)(g) we have that $u+v+z=8$ and that
$$(\alpha_1,\ldots,\alpha_u,\alpha_1',\dots,\alpha_v',\cal{D}_1,\dots,\cal{D}_z)$$
is an ordered basis of $Q(E_8)\otimes_\Z\R$.

Since each of the sets $\{\alpha_1,\dots,\alpha_u\}$, $\{\alpha_1',\dots,\alpha_v'\}$ and $\{\cal{D}_1,\dots,\cal{D}_z\}$ is a root basis of a root system of finite type, \cite[Proposition 4.7(b)]{Kac90} implies that the corresponding intersection matrices
    $$ B_R\coloneqq \big(\alpha_i\cdot \alpha_j\big)_{i,j=1}^u, \quad B_R'\coloneqq \big(\alpha'_i\cdot \alpha'_j\big)_{i,j=1}^v \quad\text{and}\quad C_R\coloneqq \big(\cal{D}_i\cdot \cal{D}_j\big)_{i,j=1}^z $$
    are invertible. Set
    \begin{equation}\label{eq:matrix_M_R}
      M_R \coloneqq B_R\oplus B_R'\oplus C_R.
    \end{equation}

\begin{lem}\label{lem:map_v}
    Assume notation as above. We view each $\alpha\in Q(E_8) $ as a row of coordinates in the basis $(\alpha_1,\ldots,\alpha_u,\alpha_1',\dots,\alpha_v',\cal{D}_1,\dots,\cal{D}_z)$ of the real vector space $Q(E_8)\otimes_\Z\R$. Then:
\begin{enumerate}[\normalfont (a)]
\item the map
    $$ \mathbf{v} \colon  Q(E_8) \to \Z^8, \qquad \alpha  \mapsto \alpha \cdot M_R $$
   is well defined and injective,
\item $\mathbf{v}(\alpha) = (\alpha\cdot \alpha_1,\ldots, \alpha\cdot \alpha_u,\alpha\cdot\alpha_1',\ldots,\alpha\cdot \alpha_v',\alpha \cdot \cal{D}_1,\dots,\alpha \cdot \cal{D}_z)$ for each $\alpha\in Q(E_8)$.
\end{enumerate}    
\end{lem}

\begin{proof}
Fix $\alpha \in Q(E_8)$. Then there exist rational numbers $r_{\alpha,1},\dots,r_{\alpha,u}$, $r_{\alpha,1}',\dots,r_{\alpha,v}'$ and $s_{\alpha,1},\dots,s_{\alpha,z}$ such that
$$\alpha = \sum\limits_{i=1}^u r_{\alpha,i} \alpha_i + \sum\limits_{i=1}^v r_{\alpha,i}' \alpha_i' + \sum\limits_{i=1}^z s_{\alpha,i} \cal{D}_i.$$
    Then by Lemma \ref{lem:somestuff}(b) we obtain
    $$ \textstyle \alpha \cdot  \alpha_i=\sum\limits_{j=1}^u r_{\alpha,j} \alpha_i\cdot\alpha_j,\quad \alpha \cdot  \alpha_i'=\sum\limits_{j=1}^v r_{\alpha,j}'\alpha_i'\cdot\alpha_j', \quad \alpha \cdot \cal{D}_i=\sum\limits_{j=1 }^z s_{\alpha,j} \cal{D}_i\cdot \cal{D}_j. $$
     These equalities give (b). In particular, this shows that $\mathbf{v}(\alpha)\in\Z^8$, hence the map $\mathbf{v}$ is well defined. The injectivity of $\mathbf{v}$ follows since $M_R$ is invertible.
\end{proof}

The following proposition is the main technical result of the section.

\begin{prop}\label{prop:finmanytranslat}
   Let $X$ be a Sakai surface of type $R$, and assume that $ R\notin\{A_8^\smone, D_8^\smone, E_8^\smone \} $. Assume Setup \ref{setup:allcases} and Setup \ref{setup2}. For each $\alpha \in Q(E_8)$ denote
   $$ x_\alpha\coloneqq \textstyle \mathcal{E}_9 - \alpha - \big(\frac{1}{2}\alpha^2\big)\delta, $$
   and note that $x_\alpha\in N^1(X)$ since $\alpha^2$ is an even integer. Recall that we have the injective map $\mathbf{v} \colon  Q(E_8) \to \Z^8$ from Lemma \ref{lem:map_v}. Then the following statements hold.
\begin{enumerate}[\normalfont (a)]
\item The set $\big\{\alpha \in Q(E_8)\mid x_\alpha \in \mathcal{N}_X\cap\overline{\mathcal F(R^\smperp)}\big\}$ is finite.
\item If $R \in\{A_6^\smone, A_7^\smone, D_7^\smone\} $, then for any $w\in W_{\mathscr{D}}$ the set
$$\mathcal G_R(w)\coloneqq \big\{\alpha \in Q(E_8)\mid x_\alpha \in \mathcal{N}_X\cap w\big(\overline{\mathcal F(R^\smperp)}\big)\big\}$$
is finite. 
\item If $R \in\{A_7^\smone, D_7^\smone\} $, then for any $w\in W_{\mathscr{D}}$ we have:
\begin{enumerate}[\normalfont (a)]
	\item[$\mathrm{(c)}_1$] $\mathcal G_R(w)=\big\{\alpha \in Q(E_8)\mid x_\alpha \in \mathcal{N}_X\cap \big( w(F_0)\cup w(F_1)\big)\big\}$,\footnote{Recall the definitions of $F_0$ and $F_1$ from Setup \ref{setup2}.}
	\item[$\mathrm{(c)}_2$] there exist distinct integers $m_{w,0}$ and $m_{w,1}$ with the following property: for each $\alpha \in Q(E_8)$ and each $i\in\{0,1\}$ we have
	$$x_\alpha \in \mathcal{N}_X\cap w(F_i)\quad\Longleftrightarrow\quad\mathbf{v}(\alpha)\in\{m_{w,i}\}\times\Z^7.$$
	\end{enumerate}
\item If $R = A_6^\smone $, then for any $w\in W_{\mathscr{D}}$ we have:
\begin{enumerate}[\normalfont (a)]
	\item[$\mathrm{(d)}_1$] $ \mathcal G_R(w)=\big\{\alpha \in Q(E_8)\mid x_\alpha \in \mathcal{N}_X\cap \big( w(F_0)\cup w(F_1)\big)\cap\overline{\cal F_{\mathscr D'}} \big\} $,
		\item[$\mathrm{(d)}_2$] there exist distinct integers $n_{w,0}$ and $n_{w,1}$ with the following property: for each $\alpha \in Q(E_8)$ and each $i\in\{0,1\}$ we have
		$$x_\alpha \in \mathcal{N}_X\cap w(F_i)\cap\overline{\cal F_{\mathscr D'}}\quad\Longleftrightarrow\quad\mathbf{v}(\alpha)\in\{n_{w,i}\}\times\Z^7.$$
		\end{enumerate}
\end{enumerate}   
\end{prop}

\begin{proof}
	  We follow the notation introduced before Lemma \ref{lem:map_v}. Note also that, by Appendix \ref{sec:Sakai_appendix_A}, there exist positive integers $p_1,\dots,p_u,p_1',\dots,p_v',q_1,\dots,q_z$ such that
      \begin{equation}\label{eq:114c}
      \delta = \alpha_0+p_1 \alpha_1+\ldots+p_u\alpha_u = \alpha_0+p_1' \alpha_1'+\ldots+p_v'\alpha_v'
      \end{equation}
      and
      \begin{equation}\label{eq:114d}
      \delta = \cal{D}_0+q_1 \cal{D}_1+\ldots+q_z \cal{D}_z.
      \end{equation}

\medskip

\emph{Step 1.}
	In this step we show part (a).
	
	Set
	\begin{align*}
	\mathcal S_R\coloneqq \Big\{&(a_1,\ldots,a_u,b_1,\ldots,b_v,c_1,\ldots,c_z)\in\Z_{\leq0}^8\ \big| \\
	&\textstyle \sum\limits_{i=1}^u p_ia_i\in\{0,-1\},\ \sum\limits_{i=1}^v p_i'b_i\in\{0,-1\},\ \sum\limits_{i=1}^z q_i c_i\in\{0,-1\}\Big\}.
	\end{align*}
		We claim that for any $\alpha \in Q(E_8)$,
		\begin{equation}\label{eq:xalpha}
		x_\alpha \in \mathcal{N}_X\cap\overline{\mathcal F(R^\smperp)}\quad\Longleftrightarrow\quad \mathbf{v}(\alpha)\in\mathcal S_R.		
		\end{equation}
	    Since the map $\mathbf{v}$ is injective, the claim immediately shows (a).
	    
	    We now prove the claim \eqref{eq:xalpha}. Fix $\alpha \in Q(E_8)$. Recall that $x_{\alpha} \in N^1(X)$, and we have $x_{\alpha}^2=K_X\cdot x_{\alpha} = {-}1 $ by \eqref{eq:6677}. Therefore, by Lemma \ref{lem:-1classesHalphen} we have that $x_{\alpha} \in {\cal N}_X$ if and only if $x_{\alpha} \cdot \cal{D}_j \geq 0$ for all $j$, which, by Lemma \ref{lem:somestuff}(d)(e), is equivalent to
	    \begin{equation}\label{eq:curve1a}
	    \alpha \cdot  \cal{D}_0  \leq 1 \quad\text{and}\quad \alpha \cdot  \cal{D}_j  \leq 0 \text{ for } j \geq 1. 
	    \end{equation}
	   Similarly, by the definition of $\overline{\mathcal F(R^\smperp)}$ and by Lemma \ref{lem:somestuff}(c)(e), we have that $x_{\alpha} \in \overline{\mathcal F(R^\smperp)}$ is equivalent to
	    \begin{equation}\label{eq:curve1b}
	    \alpha \cdot  \alpha_0 \leq 1,\quad \alpha \cdot  \alpha_0' \leq 1, \quad\text{and}\quad \alpha \cdot  \alpha_i \leq 0,\quad \alpha \cdot  \alpha_i' \leq 0 \quad \text{for } i \geq1.
	    \end{equation}
	    Moreover, by \eqref{eq:6677} and \eqref{eq:114c} we have 
	    \begin{equation}\label{eq:curve0a}
	    \textstyle \alpha \cdot  \alpha_0 = {-} \sum\limits_{i=1}^u p_1\alpha \cdot  \alpha_i,\quad \alpha \cdot  \alpha_0' = {-} \sum\limits_{i=1}^v p_1'\alpha \cdot  \alpha_i',
	    \end{equation}
		and by \eqref{eq:6677} and \eqref{eq:114d} we have 
	    \begin{equation}\label{eq:curve1}
	    \textstyle\alpha \cdot  \cal{D}_0 = {-} \sum\limits_{i=1}^z q_1 \alpha \cdot  \cal{D}_i.
	    \end{equation}
	Now it follows from relations \eqref{eq:curve1a}, \eqref{eq:curve1b}, \eqref{eq:curve0a} and \eqref{eq:curve1} that
	$$\alpha \cdot  \alpha_0,\,\alpha \cdot  \alpha_0',\,\alpha \cdot \cal{D}_0\in\{0,1\}.$$
	All this, together with Lemma \ref{lem:map_v}(b), shows \eqref{eq:xalpha} and finishes the proof of (a).

\medskip

\emph{Step 2.}
	In this step we show (b) assuming that $R \in\{A_7^\smone, D_7^\smone\} $.
	
In this case, by Appendix \ref{sec:Sakai_appendix_A} we have $\Pi(R^\smperp) = \{ \alpha_0,\alpha_1\}$ and $\Comp (D) = \{ \cal{D}_0,\ldots,\cal{D}_7\}$. The equation \eqref{eq:114c} becomes
 \begin{equation}\label{eq:AD7-cases}
     \delta = \alpha_0+\alpha_1,
 \end{equation}
 	    and by Setup \ref{setup:allcases} we have
	    \begin{equation}\label{eq:666777}
		\textstyle \alpha_0^\smvee = \frac{1}{m}\alpha_0,\quad \alpha_1^\smvee = \frac{1}{m}\alpha_1,\quad c = \alpha_0^\smvee + \alpha_1^\smvee = \frac{1}{m}(\alpha_0 + \alpha_1) = \frac{1}{m} \delta,   
		\end{equation}
		where $c$ is the canonical central element of the affine root system of type $A_1^\smone$ associated to the roots $\{\alpha_0,\alpha_1\}$.  With notation from \eqref{eq:114d}, by Appendix \ref{sec:Sakai_appendix_A} we have
		$$(q_1,\dots,q_7)\coloneqq \begin{cases}
		(1,1,1,1,1,1,1), & \textrm{if } R=A_7^\smone,\\
		(1,2,2,2,2,1,1), & \textrm{if } R=D_7^\smone.
		\end{cases}$$

 Fix $w \in W(R^\smperp)$ and let $k$ be the length of $w$. Set
	$$ \mathcal S_R(w)\coloneqq \{-k-1,-k\}\times\Big\{(c_1,\ldots,c_7)\in\Z_{\leq0}^7\ \big| \ \textstyle \sum\limits_{i=1}^7 q_i c_i\in\{0,-1\}\Big\}$$ 
	if $w$ starts with $r_0$, and set
	$$ \mathcal S_R(w)\coloneqq \{k-1,k\}\times\Big\{(c_1,\ldots,c_7)\in\Z_{\leq0}^7\ \big| \ \textstyle \sum\limits_{i=1}^7 q_i c_i\in\{0,-1\}\Big\}$$
	if $w$ starts with $r_1$. We claim that for any $\alpha \in Q(E_8)$,
		\begin{equation}\label{eq:xalpha1}
		\alpha \in \mathcal{G}_R(w)\quad\Longleftrightarrow\quad \mathbf{v}(\alpha)\in\mathcal S_R(w).
		\end{equation}
	    Since the map $\mathbf{v}$ is injective, the claim immediately shows (b) when $R \in\{A_7^\smone, D_7^\smone\} $.

	    We now prove the claim \eqref{eq:xalpha1}. Fix $\alpha \in Q(E_8)$. As in Step 1 we have that \eqref{eq:curve1a} and \eqref{eq:curve1} hold. From these two we obtain that
	    \begin{equation}\label{eq:9089}
		\alpha \cdot \cal{D}_0\in\{0,1\}.	    
	    \end{equation}

	    Using the notation from Setup \ref{setup1}, assume first that $w$ starts with $r_0$. Then by Lemma \ref{lem:r1r0_n}(g) we have 
	    $$ w\big(\overline{\cal F(R^\smperp)}\big) = \big\{x\in V\mid \langle x , kc+\alpha_0^\smvee \rangle \leq 0,\ \langle x, (k-1) c+\alpha_0^\smvee \rangle \geq 0 \big\}. $$
		Therefore, by \eqref{eq:666777} we have that $x_{\alpha}\in w\big(\overline{\cal F(R^\smperp)}\big)$ is equivalent to
	    $$ x_{\alpha}\cdot (k\delta+\alpha_0) \geq 0\quad\text{and}\quad x_{\alpha}\cdot\big((k-1) \delta+\alpha_0\big) \leq 0, $$
		which, by \eqref{eq:6677}, \eqref{eq:AD7-cases} and by Lemma \ref{lem:somestuff}(c)(e), is equivalent to
	    $$ \alpha \cdot  \alpha_1 \geq {-}k-1\quad\text{and}\quad \alpha \cdot  \alpha_1 \leq {-}k. $$
	    Therefore, since $\alpha$ and $\alpha_1$ are integral classes, we conclude that this is equivalent to
	    \begin{equation}\label{eq:6677a}
	    \alpha \cdot  \alpha_1 \in\{-k-1,-k\}.
	    \end{equation}

	    If $w$ starts with $r_1$, then similarly as above, by symmetry we have 
	    $$ w\big(\overline{\cal F(R^\smperp)}\big) = \big\{x\in V\mid \langle x , kc+\alpha_1^\smvee \rangle \leq 0,\ \langle x, (k-1) c+\alpha_1^\smvee \rangle \geq 0 \big\}. $$
		Therefore, $x_{\alpha}\in w\big(\overline{\cal F(R^\smperp)}\big)$ is equivalent to
	    $$ x_{\alpha}\cdot (k\delta+\alpha_1) \geq 0\quad\text{and}\quad x_{\alpha}\cdot\big((k-1) \delta+\alpha_1\big) \leq 0, $$
		which, by \eqref{eq:6677}, \eqref{eq:AD7-cases} and by Lemma \ref{lem:somestuff}(c)(e), is equivalent to
	    $$ \alpha \cdot  \alpha_1 \leq k\quad\text{and}\quad \alpha \cdot  \alpha_1 \geq k-1. $$
	    Therefore, since $\alpha$ and $\alpha_1$ are integral classes, we conclude that this is equivalent to
	    \begin{equation}\label{eq:6677aa}
	    \alpha \cdot  \alpha_1 \in\{k-1,k\}.
	    \end{equation}
		From Lemma \ref{lem:map_v}(b), \eqref{eq:curve1}, \eqref{eq:9089}, \eqref{eq:6677a} and \eqref{eq:6677aa} we obtain \eqref{eq:xalpha1}.

\medskip

\emph{Step 3.}
		In this step we show (c). We use notation from Step 2. Recall the definitions of $F_{k,\ell}^0$ and $F_{k,\ell}^1$ from Setup \ref{setup2}.
		
		If $w$ starts with $r_0$, then by Step 2 we have $\alpha \cdot  \alpha_1 \in\{-k-1,-k\}$. If $\alpha \cdot  \alpha_1={-}k-1$, then
	    $$\langle x_{\alpha} , kc+\alpha_0^\smvee \rangle = 0\quad\text{and}\quad \big\langle x_{\alpha}, c \big\rangle < 0, $$
	    hence $x_{\alpha}\in F^0_{k,0}$ by Lemma \ref{lem:equalities_of_faces}(a). Analogously, if $\alpha \cdot  \alpha_1={-}k$, then $x_{\alpha}\in F^0_{k-1,0}$. This shows that $x_{\alpha}\in w(F_0)\cup w(F_1)$ by Lemma \ref{lem:equalities_of_faces}(c), and that the first coordinate of $\mathbf{v}(\alpha)$ decides whether $x_{\alpha}$ belongs to $w(F_0)$ or to $w(F_1)$ by Lemma \ref{lem:map_v}(b).
	    
	    If $w$ starts with $r_1$, then by Step 2 we have $\alpha \cdot  \alpha_1 \in\{k-1,k\}$. If $\alpha \cdot  \alpha_1=k-1$, then
	    $$\big\langle x_{\alpha} , (k-1) c+\alpha_1^\smvee \big\rangle = 0\quad\text{and}\quad \big\langle x_{\alpha}, c \big\rangle < 0, $$
	    hence $x_{\alpha}\in F^1_{0,k-1}$ by Lemma \ref{lem:equalities_of_faces}(b). Analogously, if $\alpha \cdot  \alpha_1=k$, then $x_{\alpha}\in F^1_{0,k}$. This shows that $x_{\alpha}\in w(F_0)\cup w(F_1)$ by Lemma \ref{lem:equalities_of_faces}(d), and that the first coordinate of $\mathbf{v}(\alpha)$ decides whether $x_{\alpha}$ belongs to $w(F_0)$ or to $w(F_1)$ by Lemma \ref{lem:map_v}(b). This finishes the proof of (c).

	\medskip
	
	\emph{Step 4.}
	In this step we show (b) assuming that $R=A_6^\smone$.
	
In this case, by Appendix \ref{sec:Sakai_appendix_A} we have $\Pi(R^\smperp) = \{ \alpha_0,\alpha_1,\alpha_0',\alpha_1'\}$ and $\Comp (D) = \{ \cal{D}_0,\ldots,\cal{D}_6\}$. The equation \eqref{eq:114c} becomes
 \begin{equation}\label{eq:A6-case}
     \delta = \alpha_0+\alpha_1 = \alpha_0'+\alpha_1',
 \end{equation}
  	    and by Setup \ref{setup:allcases} we have
	    \begin{equation}\label{eq:666777aaa}
		\textstyle \alpha_0^\smvee = \frac{1}{m}\alpha_0,\quad \alpha_1^\smvee = \frac{1}{m}\alpha_1,\quad c = \alpha_0^\smvee + \alpha_1^\smvee = \frac{1}{m}(\alpha_0 + \alpha_1) = \frac{1}{m} \delta,    
		\end{equation}
		where $c$ is the canonical central element of the affine root system of type $A_1^\smone$ associated to the roots $\{\alpha_0,\alpha_1\}$. With notation from \eqref{eq:114d}, by Appendix \ref{sec:Sakai_appendix_A} we have
		$$ (q_1,\dots,q_6)\coloneqq (1,1,1,1,1,1). $$

 	Fix $w  \in W_{\mathscr D}$ and let $k$ be the length of $w$. Set
		$$ \mathcal S_R(w)\coloneqq \{-k-1,-k\}\times\{0,-1\}\times\Big\{(c_1,\ldots,c_6)\in\Z_{\leq0}^6\ \big| \ \textstyle \sum\limits_{i=1}^6 c_i\in\{0,-1\}\Big\} $$
		if $w$ starts with $r_0$, and set
		$$ \mathcal S_R(w)\coloneqq \{k-1,k\}\times\{0,-1\}\times \Big\{(c_1,\ldots,c_6)\in\Z_{\leq0}^6\ \big| \ \textstyle \sum\limits_{i=1}^6 c_i\in\{0,-1\}\Big\} $$
		if $w$ starts with $r_1$. We claim that for any $\alpha \in Q(E_8)$,
		\begin{equation}\label{eq:xalpha2}
		\alpha \in \mathcal{G}_R(w)\quad\Longleftrightarrow\quad \mathbf{v}(\alpha)\in\mathcal S_R(w).
		\end{equation}
	    Since the map $\mathbf{v}$ is injective, the claim immediately shows (b) when $R=A_6^\smone$.

	    We now prove the claim \eqref{eq:xalpha2}. Fix $\alpha \in Q(E_8)$. As in Step 2 we have that
	    \begin{equation}\label{eq:9089a}
		\alpha \cdot \cal{D}_0\in\{0,1\},
	    \end{equation}
	    and \eqref{eq:curve1} becomes
	    \begin{equation}\label{eq:curve1aaa}
	    \textstyle\alpha \cdot  \cal{D}_0 = {-} \sum\limits_{i=1}^6 \alpha \cdot  \cal{D}_i.
	    \end{equation}

	    Using the notation from Setup \ref{setup1}, assume first that $w$ starts with $r_0$. Then by Lemma \ref{lem:r1r0_n}(g) and Proposition \ref{pro:decomposable}(f) we have
	    \begin{align*}
	    w\big(\overline{\cal F(R^\smperp)}\big) = \big\{x\in V\mid\ &\langle x , kc+\alpha_0^\smvee \rangle \leq 0,\ \langle x, (k-1) c+\alpha_0^\smvee \rangle \geq 0,\\
	    &\langle x , (\alpha_0')^\smvee \rangle \leq 0,\ \langle x, (\alpha_1')^\smvee \rangle\leq0  \big\}.
	    \end{align*}
		Then as in Step 2, by using \eqref{eq:A6-case} and \eqref{eq:666777aaa} we obtain that $x_{\alpha} \in w\big(\overline{\cal F(R^\smperp)}\big)$ is equivalent to
	    \begin{equation}\label{eq:6677a1}
	    \alpha \cdot  \alpha_1 \in \{-k-1,-k\}\quad\text{and}\quad \alpha \cdot  \alpha_1'\in\{0,-1\}.
	    \end{equation}
	    If $w$ starts with $r_1$, then analogously we obtain that $x_{\alpha} \in w\big(\overline{\cal F(R^\smperp)}\big)$ is equivalent to
	    \begin{equation}\label{eq:6677a2}
	    \alpha \cdot  \alpha_1 \in \{k-1,k\}\quad\text{and}\quad \alpha \cdot  \alpha_1'\in\{0,-1\}.
	    \end{equation}
		From Lemma \ref{lem:map_v}(b), \eqref{eq:9089a}, \eqref{eq:curve1aaa}, \eqref{eq:6677a1} and \eqref{eq:6677a2} we obtain \eqref{eq:xalpha2}.
		
		\medskip

\emph{Step 5.}
In this step we show (d). We use notation from Step 4.

Denote
$$\mathcal G_R'(w)\coloneqq\big\{\alpha \in Q(E_8)\mid x_\alpha \in \mathcal{N}_X\cap \big( w(F_0)\cup w(F_1)\big)\cap\overline{\cal F_{\mathscr D'}} \big\}.$$
Fix $\alpha\in\mathcal G_R(w)$. Similarly as in Step 3, by Step 4 we have that $ x_{\alpha}\in w(F_0)\cup w(F_1) $, and one can read off whether $x_{\alpha}$ belongs to $w(F_0)$ or to $w(F_1)$ by looking at the first coordinate of $\mathbf{v}(\alpha)$. Thus, since $x_{\alpha}\in w\big(\overline{\cal F(R^\smperp)}\big)\subseteq \overline{\cal F_{\mathscr D'}}$ by Proposition \ref{pro:decomposable}(f), we obtain $\mathcal G_R(w)\subseteq\mathcal G_R'(w)$. The converse inclusion follows since $\big( w(F_0)\cup w(F_1)\big)\cap\overline{\cal F_{\mathscr D'}}\subseteq w\big(\overline{\cal F(R^\smperp)}\big)$ by Proposition \ref{pro:decomposable}(f).
	    \end{proof}

The first important corollary of Proposition \ref{prop:finmanytranslat} is the following result.

\begin{cor}\label{cor:finite_(-1)_curves_chamber1}
   Let $X$ be a Sakai surface of type $R$. For each $\alpha\in Q(E_8)$ denote $ x_\alpha\coloneqq \textstyle \mathcal{E}_9 - \alpha - \big(\frac{1}{2}\alpha^2\big)\delta $. Then:
\begin{enumerate}[\normalfont (a)]
\item for each $\gamma\in \mathcal{M}_X$ there exists a unique $\alpha \in Q(E_8)$ such that $ \gamma=x_\alpha$,
\item for each subset $\mathcal Q\subseteq \mathcal M_X$ we have $|\mathcal Q|=\big|\big\{\alpha \in Q(E_8)\mid x_\alpha \in \mathcal{Q}\big\}\big|$,
\item if $R\notin\{A_8^\smone,D_8^\smone,E_8^\smone\} $, then the set $\mathcal{N}_X \cap \overline{\mathcal F(R^\smperp)}$ is finite,
\item if $R\notin \{A_6^\smone, A_7^\smone, A_8^\smone, D_7^\smone, D_8^\smone, E_8^\smone\} $ and if $D\in|{-}K_X|$, then $W(R^\smperp)$ acts on $\cal N_X$ with finitely many orbits.
\end{enumerate}   
\end{cor}

\begin{proof}
Part (a) follows from Lemma \ref{lem:translationsatE9} and from \eqref{eq:117aa}. Part (b) follows immediately from (a). If $R\notin\{A_8^\smone,D_8^\smone,E_8^\smone\} $, by (b) we have
	    $$ \big|\mathcal{N}_X\cap\overline{\mathcal F(R^\smperp)}\big|=\big|\big\{\alpha \in Q(E_8)\mid x_\alpha \in \mathcal{N}_X\cap\overline{\mathcal F(R^\smperp)}\big\}\big|, $$
	    which is finite by Proposition \ref{prop:finmanytranslat}(a). This proves (c).

Now we show (d). We have $\mathcal N_X\subseteq\mathcal T(R^\smperp)$ by Lemma \ref{lem:effconeHalpensurface1}(e), and $W(R^\smperp)$ preserves $\mathcal{N}_X$ by Lemma \ref{lem:weylpreservesnefgen1}(e). Therefore, by Proposition \ref{prop:Titsfundchamber}, for every $x\in\mathcal N_X$ there exists $w\in W(R^\smperp)$ such that $w(x)\in \mathcal{N}_X\cap\overline{\mathcal F(R^\smperp)}$. Then (d) follows immediately from (c).
	   \end{proof}

The second important corollary of Proposition \ref{prop:finmanytranslat} is the following.

\begin{cor}\label{cor:finiteness}
    Let $X$ be a Sakai surface of type $R \in \{A_6^\smone,A_7^\smone,D_7^\smone \}$. Assume Setup \ref{setup:allcases} and Setup \ref{setup2}. Let $w \in W_{\mathscr{D}}$. Then the following statements hold.
\begin{enumerate}[\normalfont (a)]
\item We have $ \big|\cal N_X \cap w \big( \overline{\cal F(R^\smperp)} \big)\big|=2 $.
\item If $R\in \{ A_7^\smone, D_7^\smone \}$, then:
\begin{enumerate}[\normalfont (a)]
\item[$\mathrm{(b)}_1$] $ \cal N_X \cap w \big( \overline{\cal F(R^\smperp)} \big) = \cal N_X \cap \big( w(F_0)\cup w(F_1)\big) $,
\item[$\mathrm{(b)}_2$] $\big|\cal N_X \cap w(F_0)\big|=1$ and $\big|\cal N_X \cap w(F_1)\big|=1$,
\item[$\mathrm{(b)}_3$] if $F$ is a face of $\overline{\mathcal F_{\mathscr D_{k,\ell}}}$ for some non-negative integers $k$ and $\ell$, then $\big|\cal N_X \cap F\big|=1$.
\end{enumerate}
\item If $R=A_6^\smone$, then:
\begin{enumerate}[\normalfont (a)]
\item[$\mathrm{(c)}_1$] $ \cal N_X \cap w \big( \overline{\cal F(R^\smperp)} \big) = \cal N_X \cap \big( w(F_0)\cup w(F_1)\big)\cap \overline{\cal{F}_{\mathscr{D}'}} $,
\item[$\mathrm{(c)}_2$] $\big|\cal N_X \cap w(F_0)\cap \overline{\cal{F}_{\mathscr{D}'}}\big|=1$ and $\big|\cal N_X \cap w(F_1)\cap \overline{\cal{F}_{\mathscr{D}'}}\big|=1$,
\item[$\mathrm{(c)}_3$] if $F$ is a face of $\overline{\mathcal F_{\mathscr D_{k,\ell}}}$ for some non-negative integers $k$ and $\ell$, then $\big|\cal N_X \cap F\cap \overline{\cal{F}_{\mathscr{D}'}}\big|=1$.
\end{enumerate}
\end{enumerate}    
\end{cor}

\begin{proof}
We first note that (a) follows from $\mathrm{(b)}_1$, $\mathrm{(b)}_2$, $\mathrm{(c)}_1$ and $\mathrm{(c)}_2$.

Now we show that $\mathrm{(b)}_3$ follows from $\mathrm{(b)}_2$. To that end, assume first that $F=F_{k,\ell}^0$. If $k=0$, then $F=F_0$ by Lemma \ref{lem:equalities_of_faces}(a), hence $\mathrm{(b)}_3$ is an immediate consequence of $\mathrm{(b)}_2$ when $w$ is the identity. If $k>0$, then by Lemma \ref{lem:equalities_of_faces}(c) there exists $u\in W_{\mathscr D}$ of length $k$ such that $F=u(F_j)$, hence $\mathrm{(b)}_3$ is an immediate consequence of $\mathrm{(b)}_2$ for $w =   u$. If $F=F_{k,\ell}^1$, then the proof is analogous, by replacing Lemma \ref{lem:equalities_of_faces}(a) by Lemma \ref{lem:equalities_of_faces}(b), and by replacing Lemma \ref{lem:equalities_of_faces}(c) by Lemma \ref{lem:equalities_of_faces}(d).

The proof that $\mathrm{(c)}_3$ follows from $\mathrm{(c)}_2$ is analogous to the previous paragraph.

In the remainder of the proof we show $\mathrm{(b)}_1$, $\mathrm{(b)}_2$, $\mathrm{(c)}_1$ and $\mathrm{(c)}_2$. For each $\alpha\in Q(E_8)$ denote $ x_\alpha\coloneqq \textstyle \mathcal{E}_9 - \alpha - \big(\frac{1}{2}\alpha^2\big)\delta $, and set
$$ \mathcal G_R(w)\coloneqq \big\{\alpha \in Q(E_8)\mid x_\alpha \in \mathcal{N}_X\cap w\big(\overline{\mathcal F(R^\smperp)}\big)\big\}. $$

\medskip

\emph{Step 1.}
In this step we show $\mathrm{(b)}_1$ and $\mathrm{(c)}_1$.

Note first that by Corollary \ref{cor:finite_(-1)_curves_chamber1}(b) we have
\begin{equation}\label{eq:45678}
\big|\mathcal{N}_X\cap w \big( \overline{\cal F(R^\smperp)} \big)\big|=|\mathcal G_R(w)|.
\end{equation}

Assume first that $R\in \{ A_7^\smone, D_7^\smone \}$. By Corollary \ref{cor:finite_(-1)_curves_chamber1}(b) we have
\begin{align}\label{eq:firstset}
\big|\cal N_X &\cap \big( w(F_0)\cup w(F_1)\big)\big| \\
&= \big|\big\{\alpha \in Q(E_8)\mid x_\alpha \in \mathcal{N}_X\cap \big( w(F_0)\cup w(F_1)\big)\big\}\big|.\notag
\end{align}
We now deduce from \eqref{eq:45678}, \eqref{eq:firstset} and from Proposition \ref{prop:finmanytranslat}$\mathrm{(c)}_1$ that
$$ \big|\cal N_X \cap w \big( \overline{\cal F(R^\smperp)} \big)\big| = \big|\cal N_X \cap \big( w(F_0)\cup w(F_1)\big)\big| , $$
and this cardinality is finite by Proposition \ref{prop:finmanytranslat}(b). Since $ \cal N_X \cap \big( w(F_0)\cup w(F_1)\big)\subseteq \cal N_X \cap w \big( \overline{\cal F(R^\smperp)} \big) $, part $\mathrm{(b)}_1$ follows.

Now assume that $R=A_6^\smone$. By Corollary \ref{cor:finite_(-1)_curves_chamber1}(b) we have 
\begin{align}\label{eq:secondset}
\big|\cal N_X &\cap \big( w(F_0)\cup w(F_1)\big)\cap \overline{\cal{F}_{\mathscr{D}'}}\big| \\
&= \big|\big\{\alpha \in Q(E_8)\mid x_\alpha \in \mathcal{N}_X\cap \big( w(F_0)\cup w(F_1)\big)\cap \overline{\cal{F}_{\mathscr{D}'}}\big\}\big|.\notag
\end{align}
We now deduce from \eqref{eq:45678}, \eqref{eq:secondset} and from Proposition \ref{prop:finmanytranslat}$\mathrm{(d)}_1$ that
$$ \big|\cal N_X \cap w \big( \overline{\cal F(R^\smperp)} \big)\big| = \big|\cal N_X \cap \big( w(F_0)\cup w(F_1)\big)\cap \overline{\cal{F}_{\mathscr{D}'}}\big|, $$
and this cardinality is finite by Proposition \ref{prop:finmanytranslat}(b). Since 
$$ \cal N_X \cap \big( w(F_0)\cup w(F_1)\big)\cap \overline{\cal{F}_{\mathscr{D}'}}\subseteq \cal N_X \cap w(\mathcal F_{\mathscr D})\cap \overline{\cal{F}_{\mathscr{D}'}} = \cal N_X \cap w \big( \overline{\cal F(R^\smperp)} \big) $$
by Proposition \ref{pro:decomposable}(f), part $\mathrm{(c)}_1$ follows.

\medskip

\emph{Step 2.}
In this step we show $\mathrm{(b)}_2$ and $\mathrm{(c)}_2$. By \eqref{eq:firstset}, \eqref{eq:secondset} and by Proposition \ref{prop:finmanytranslat}(c)(d), it suffices to prove that
$$\mathbf{v}\big(\mathcal G_R(w)\big)=\{\beta_1,\beta_2\},$$
where the first coordinate of $\beta_1$ is different from the first coordinate of $\beta_2$. Here, $ \mathbf{v} $ is the injective map from Lemma \ref{lem:map_v}. 

In this step we view elements of the $\Q$-vector space $Q(E_8)\otimes_\Z\Q$ as rows of coordinates in the basis $(v_1,\dots,v_8)\coloneqq (\alpha_1,\ldots,\alpha_u,\alpha_1',\dots,\alpha_v',\cal{D}_1,\dots,\cal{D}_z)$, and we identify it in this way with $\Q^8$. By Lemma \ref{lem:somestuff}(g), for each $i=1,\dots 8$, there exists $(a_{i,0},\dots,a_{i,8})\in\Z^9$ such that $v_i=a_{i,0}\mathcal E_0+\dots+a_{i,8}\mathcal E_8$. Let $A_R$ be the $8\times 9$ matrix whose $i$-th row is the vector $(a_{i,0},\dots,a_{i,8})$. Recall the set $\mathcal S_R(w)$ defined in Steps 2 and 4 of the proof of Proposition \ref{prop:finmanytranslat}, and let $M_R$ be the invertible $8\times8$ matrix defined in \eqref{eq:matrix_M_R}.

The map $\mathbf{v}$ induces the bijective map
$$ \mathbf{v}_\Q \colon  Q(E_8)\otimes_\Z\Q \to \Q^8, \qquad \alpha  \mapsto \alpha \cdot M_R. $$
For each $\beta\in\mathcal S_R(w)$, define
$$\alpha_{\beta} \coloneqq \mathbf{v}_\Q^{-1}(\beta)=\beta\cdot M_R^{-1}\in Q(E_8)\otimes_\Z\Q\simeq\Q^8$$
and 
	\begin{equation}\label{eq:xbetas}
	(x_{\beta,0},\dots,x_{\beta,8})\coloneqq \alpha_{\beta}\cdot A_R\in\Q^9.
	\end{equation}
	Then by the definition of $A_R$, $(x_{\beta,0},\dots,x_{\beta,8})$ is precisely the vector of coefficients of $\alpha_\beta$, when written as a linear combination of the elements $\mathcal E_0,\dots,\mathcal E_8$. Therefore, Lemma \ref{lem:integral_elements} yields
\begin{equation}\label{eq:equiv}
\alpha_\beta\in Q(E_8)\quad\Longleftrightarrow\quad (x_{\beta,0},\dots,x_{\beta,8})\in\Z^9.
\end{equation}

Now, by the relations \eqref{eq:xalpha1} and \eqref{eq:xalpha2} from the proof of Proposition \ref{prop:finmanytranslat} we have $ \mathcal G_R(w)=\big\{\alpha_\beta\mid \beta\in\mathcal S_R(w)\big\}\cap Q(E_8) $, hence \eqref{eq:xbetas} and \eqref{eq:equiv} give
	$$ \mathcal G_R(w)=\big\{\alpha_\beta\mid \beta\in\mathcal S_R(w)\text{ and }\beta\cdot M_R^{-1}\cdot A_R\in\Z^9\big\}. $$
In other words,
$$ \mathbf{v}\big(\mathcal G_R(w)\big)=\{ \beta\in\mathcal S_R(w) \mid \beta\cdot M_R^{-1}\cdot A_R\in\Z^9\}.$$
Then the result follows from Lemma \ref{lem:aux_finiteness}.
\end{proof}

\section{Orbits in exceptional cases}\label{sec:orbits}

To prove Theorem \ref{thm:main2} when the surface $X$ is of type $A_6^\smone$, $A_7^\smone$ or $D_7^\smone$, the finiteness of the number of orbits of $\Cr(X,D)$ on $\mathcal N_X$ is not sufficient. We need a finer decomposition of the set of $(-1)$-curves with respect to the subgroups $\widehat W_{\mathscr D_{k,m-k-1}}$ from Setup \ref{setup:allcases}, which is the main result of this section. 

\begin{thm}\label{thm:mainthmsection8}
    Let $X$ be a generic Sakai surface of type $A_6^\smone$, $A_7^\smone$ or $D_7^\smone$, and let $D\in|{-}K_X|$. Assume Setup \ref{setup:allcases}, and let $\cal N_X$ be the set of $({-}1)$-curves on $X$. Then:
    \begin{enumerate}[\normalfont (a)]
    	\item $\Omega_R\cap\mathcal N_X$ is finite,
        \item $\Cr(X,D)$ acts on $\cal N_X$ with finitely many orbits,
        \item for each $k\in \{ 0,\dots, m-1\}$ there exists a finite subset $\mathcal R_k\subseteq\mathcal N_X$ such that
$$ \textstyle \cal N_X = \bigcup_{k=0}^{m-1} \widehat W_{\mathscr D_{k,m-k-1}} \cdot \mathcal R_k. $$
    \end{enumerate} 
\end{thm}

We prove Theorem \ref{thm:mainthmsection8} at the end of the section.

We start with the following lemmas, which characterise the set $\Omega_R$ as the finite union of translates of the negative fundamental chamber $\overline{\cal F(R^\smperp)}$, as well as a fundamental domain for the action of $\Cr(X,D)$ on $\cal T(R^\smperp)$.

\begin{lem}\label{lem:OmegaR}
     Let $X$ be a generic Sakai surface of type $R$ which is different from $ A_8^\smone$, $D_8^\smone$ and $E_8^\smone$. Assume Setup \ref{setup:allcases}. Then there exist finitely many elements $u_1,\dots,u_{2m}\in W_{\mathscr D}$ of length at most $m$, such that
$$ \textstyle \Omega_R = \bigcup_{j=1}^{2m} u_j\big(\overline{\cal F(R^\smperp)}\big) . $$
\end{lem}

\begin{proof}
If $ R\notin\{A_6^\smone, A_7^\smone, D_7^\smone \} $, then the lemma follows immediately by the definition of $\Omega_R$.

In the remainder of the proof we assume that $ R\in\{A_6^\smone, A_7^\smone, D_7^\smone \} $. By Proposition \ref{prop:WDkl}(b), there are finitely many elements $u_1,\dots,u_{2m}\in W_{\mathscr D}$ of length at most $m$, such that
$$ \textstyle\overline{\cal F_{\mathscr D_{m,m-1}}} = \bigcup_{j=1}^{2m} u_j(\overline{\cal F_{\mathscr D}}) . $$
This immediately settles the lemma if $R\in \{A_7^\smone, D_7^\smone \}$.

Finally, we assume that $R=A_6^\smone$. Since $u_j \in W_{\mathscr D} \subseteq W(R^\smperp)$ for all $j$, by Proposition \ref{pro:decomposable}(f) we obtain
$$ \Omega_R = \textstyle \overline{\cal F_{{\mathscr D}_{m, m-1}}}\cap \overline{\cal{F}_{\mathscr{D}'}} =  \big( \bigcup_{j=1}^{2m} u_j(\overline{\cal F_{\mathscr D}}) \big) \cap \overline{\cal{F}_{\mathscr{D}'}} = \textstyle\bigcup_{j=1}^{2m} u_j\big(\overline{\cal F(R^\smperp)}\big), $$
as desired.
\end{proof}

\begin{prop}\label{prop:fund_domain_Cr(X,D)_Tits_cone}
     Let $X$ be a generic Sakai surface of type $R$ which is different from $ A_8^\smone$, $D_8^\smone$ and $E_8^\smone$. Assume Setup \ref{setup:allcases}, and let $D\in|{-}K_X|$. Then $\Omega_R$ is a fundamental domain for the action of $\Cr(X,D)$ on $\cal T(R^\smperp)$.
\end{prop}

\begin{proof}
If $ R\notin\{A_6^\smone, A_7^\smone, D_7^\smone \} $, then by \cite[Theorem 26(a)]{Sak01} we have $\Cr(X,D)=W(R^\smperp)$, and $\Omega_R=\overline{\cal F(R^\smperp)}$ by the definition of $\Omega_R$. Then the lemma follows immediately by Proposition \ref{prop:Titsfundchamber}. If $R \in \{A_7^\smone, D_7^\smone \}$, then the result follows immediately from Proposition \ref{prop:fund_domain_G}(c) and by Setup \ref{setup:allcases}.

Thus, in the remainder of the proof we assume that $ R = A_6^\smone$. In particular, we have $m=7$ and
$$\Omega_R = \overline{\cal F_{{\mathscr D}_{m,m-1}}}\cap \overline{\cal F_{\mathscr D'}},$$
and set
$$g\coloneqq (r_0r_1)^m\in W_{\mathscr D}.$$

\medskip

\emph{Step 1.}
By Lemma \ref{lem:OmegaR} there are finitely many elements $u_1,\dots,u_{2m}\in W_{\mathscr D}$ such that
$$ \textstyle \Omega_R = \bigcup_{j=1}^{2m} u_j\big(\overline{\cal F(R^\smperp)}\big)\subseteq \cal T(R^\smperp) . $$
This, together with $\Cr(X,D) \subseteq W(R^\smperp)$, proves that
$$ \Cr(X,D) \cdot \Omega_R \subseteq \cal T(R^\smperp). $$

\medskip

\emph{Step 2.}
In this step, we show that
$$ \cal T (R^\smperp) \subseteq \Cr(X,D) \cdot \Omega_R. $$
Together with Step 1, this will show that $ \Cr(X,D) \cdot \Omega_R = \cal T(R^\smperp) $.

Let $x \in \cal T (R^\smperp)$. By Corollary \ref{cor:Titssubcone} we have $ \cal T (R^\smperp)\subseteq \cal{T}_{\mathscr{D}'}$, hence there exists an element $w' \in W_{\mathscr D'}$ such that $w'(x) \in \overline{\cal F_{\mathscr D'}}$. Clearly we have $w'(x) \in \cal T(R^\smperp)$, since $W_{\mathscr D'} \subseteq W(R^\smperp)$. Since $ \cal T (R^\smperp)\subseteq \cal{T}_{\mathscr{D}}$ again by Corollary \ref{cor:Titssubcone}, we conclude that $w'(x) \in \cal{T}_{\mathscr{D}}\cap\overline{\cal F_{\mathscr D'}}$.

Then by Proposition \ref{prop:fund_domain_G}(c) there exists $k \in \Z$ such that $g^k w'(x) \in \overline{\cal F_{{\mathscr D}_{m,m-1}}}\cap g^k (\overline{\cal F_{\mathscr D'}})$. Since $g^k (\overline{\cal F_{\mathscr D'}})=\overline{\cal F_{\mathscr D'}}$ by Proposition \ref{pro:decomposable}(e), we obtain
$$ g^k\big(w'(x)\big) \in \overline{\cal F_{{\mathscr D}_{m,m-1}}} \cap \overline{\cal F_{\mathscr D'}} = \Omega_R. $$
As $g^k w' \in \langle g\rangle \cdot W_{\mathscr D'} = \Cr(X,D)$ by Setup \ref{setup:allcases}, we conclude.

\medskip

\emph{Step 3.}
	Assume that there exists $w \in \Cr(X,D)$ such that 
	$$ w \big(\operatorname{Int}(\Omega_R)\big) \cap \operatorname{Int}(\Omega_R) \neq \emptyset. $$
    We will show in this step that then $w$ is necessarily the identity. This, together with Steps 1 and 2, will finish the proof of the proposition.

Note that $\operatorname{Int}(\Omega_R)=\cal F_{{\mathscr D}_{m,m-1}}\cap\cal F_{{\mathscr D'}}$. Since $\langle g\rangle \cdot W_{\mathscr D'} = \Cr(X,D)$ by Setup \ref{setup:allcases}, there exist $n \in \Z$ and $w' \in W_{\mathscr D'}$ such that
$$ w = g^n w' = w' g^n, $$
where the last equality follows from Proposition \ref{pro:decomposable}(a). Lemma \ref{lem:wandw'invariant}(a) gives
$$ w (\cal F_{{\mathscr D}_{m,m-1}}) = g^n w' (\cal F_{{\mathscr D}_{m,m-1}}) = g^n (\cal F_{{\mathscr D}_{m,m-1}}), $$
and by Proposition \ref{pro:decomposable}(e) we have
$$ w (\cal F_{{\mathscr D'}}) = w'g^n (\cal F_{{\mathscr D'}}) = w' (\cal F_{{\mathscr D'}}). $$
Therefore,
\begin{align*}
\emptyset&\neq w \big(\operatorname{Int}(\Omega_R)\big) \cap \operatorname{Int}(\Omega_R) = \big( w (\cal F_{{\mathscr D}_{m,m-1}}) \cap w (\cal F_{{\mathscr D'}}) \big) \cap \big( \cal F_{{\mathscr D}_{m,m-1}}\cap\cal F_{{\mathscr D'}} \big) \\
&= \big(g^n (\cal F_{{\mathscr D}_{m,m-1}}) \cap\cal F_{{\mathscr D}_{m,m-1}}\big)\cap \big(w' (\cal F_{{\mathscr D'}}) \cap\cal F_{{\mathscr D'}}\big).
\end{align*}
In particular, we conclude that $g^n (\cal F_{{\mathscr D}_{m,m-1}}) \cap\cal F_{{\mathscr D}_{m,m-1}}\neq\emptyset$ and $w' (\cal F_{{\mathscr D'}}) \cap\cal F_{{\mathscr D'}}\neq\emptyset$. Thus, $g^n$ is the identity by Proposition \ref{prop:fund_domain_G}(c), and $w'$ is the identity by Proposition \ref{prop:Titsfundchamber}. This concludes the proof.
\end{proof}

The following result is the main technical result of this section.

\begin{lem}\label{lem:alternative_unique_gamma0}
    Let $X$ be a generic Sakai surface of type  $A_6^\smone$, $A_7^\smone$ or $D_7^\smone$. Assume Setup \ref{setup:allcases}. Then for each $k\in \{ 0,\dots, m-1\}$ there exists a finite subset $\mathcal R_k\subseteq\mathcal N_X\cap \widehat{\mathcal F}_{\mathscr D_{k,m-k-1}}$ such that
$$ \textstyle \cal N_X = \bigcup_{k=0}^{m-1} \widehat W_{\mathscr D_{k,m-k-1}} \cdot \mathcal R_k. $$
\end{lem}

\begin{proof}
By Setup \ref{setup:allcases}, the canonical central element of the affine root system of type $A_1^\smone$ associated to the roots $\{\alpha_0,\alpha_1\}$ is
$$\textstyle c=\alpha_0^\smvee+\alpha_1^\smvee=\frac1m(\alpha_0+\alpha_1)=\frac1m\delta.$$
Then for every $\gamma\in\mathcal N_X$ we have $\langle \gamma, mc \rangle={-}m\gamma \cdot c={-}\gamma \cdot \delta={-}1$, hence
\begin{equation}\label{eq:gamma}
\textstyle \langle \gamma, c \rangle={-}\frac{1}{m}\quad\text{for every }\gamma\in\mathcal N_X.
\end{equation}

We prove the lemma in several steps.

\medskip

\emph{Step 1.}
   In this step we construct the sets $\mathcal R_k$.
   
For $k\in\{0,\dots,m-1\}$, and with notation from Setup \ref{setup2}, denote
   $$ \mathcal R_k\coloneqq \begin{cases}
   \cal N_X \cap (F_{k,m-k-1}^0\cup F_{k,m-k-1}^1),& \text{if } R\in\{A_7^\smone,D_7^\smone\},\\
   \cal N_X \cap (F_{k,m-k-1}^0\cup F_{k,m-k-1}^1)\cap \overline{\mathcal F_{\mathscr D'}},& \text{if } R=A_6^\smone.
   \end{cases} $$
   Then by Corollary \ref{cor:finiteness}$\mathrm{(b)}_3$\mbox{}$\mathrm{(c)}_3$, the set $\mathcal R_k$ is finite, and clearly $\mathcal R_k\subseteq\mathcal N_X\cap \widehat{\mathcal F}_{\mathscr D_{k,m-k-1}}$.
   
   \medskip
      
      \emph{Step 2.}
   Fix $k\in \{ 0, \dots, m-1 \}$. In this step we show that
   $$ W_{\mathscr D_{k,m-k-1}} \cdot \mathcal R_k \subseteq \cal{N}_X. $$
By Lemma \ref{lem:weylpreservesnefgen1}(d), it suffices to show that $ W_{\mathscr D_{k,m-k-1}} \cdot \mathcal R_k \subseteq N^1(X) $.

To that end, fix $\gamma\in \mathcal R_k$. Since $\gamma \in \cal N_X$, by \eqref{eq:gamma} we have
$$ {-}1 = \langle \gamma, m c \rangle = \langle \gamma, kc + \alpha_0^\smvee \rangle + \big\langle \gamma, (m-k-1)c + \alpha_1^\smvee\big\rangle, $$
and since $\gamma$ lies on $F_{k,m-1}^0$ or on $F_{k,m-1}^1$ by construction, the equation above yields that either
\begin{equation}\label{eq:gamma_intersect_c1a}
\langle \gamma, kc + \alpha_0^\smvee \rangle=0\quad\text{and}\quad \big\langle\gamma, (m-k-1)c + \alpha_1^\smvee \big\rangle = {-}1,
\end{equation}
or
\begin{equation}\label{eq:gamma_intersect_c1b}
\langle \gamma, kc + \alpha_0^\smvee \rangle={-}1\quad\text{and}\quad \big\langle\gamma, (m-k-1)c + \alpha_1^\smvee \big\rangle = 0.
\end{equation}
If we have \eqref{eq:gamma_intersect_c1a}, then $ r_{k\delta + \alpha_0}(\gamma) = \gamma\in N^1(X)$ and $ r_{(m-k-1)\delta + \alpha_1 }(\gamma) = \gamma + (m-k-1)\delta + \alpha_1 \in N^1(X)$. Similarly, if we have \eqref{eq:gamma_intersect_c1b}, then $ r_{k\delta + \alpha_0}(\gamma) = \gamma+k\delta+\alpha_0\in N^1(X)$ and $ r_{(m-k-1)\delta + \alpha_1 }(\gamma) = \gamma \in N^1(X)$. Since $W_{\mathscr D_{k,m-k-1}} = \big\langle r_{k\delta + \alpha_0}, r_{(m-k-1)\delta + \alpha_1 }\big\rangle$ and
$$ \quad W_{\mathscr{D}_{k,m-k-1}}\cdot\{ k\delta + \alpha_0, (m-k-1)\delta + \alpha_1 \}\subseteq N^1(X), $$
we deduce that $ W_{\mathscr D_{k,m-k-1}} \cdot \{ \gamma \}\subseteq N^1(X)$, as desired.

\medskip

\emph{Step 3.}
Fix $k\in \{ 0, \dots, m-1 \}$. In this step we show that
   $$ \widehat W_{\mathscr D_{k,m-k-1}} \cdot \mathcal R_k \subseteq \cal{N}_X. $$

	Indeed, this follows immediately from Step 2 if $R\in\{A_7^\smone,D_7^\smone\}$. If $R=A_6^\smone$, note that $\widehat W_{\mathscr D_{k,m-k-1}}=W_{\mathscr{D}'}\cdot W_{\mathscr D_{k,m-k-1}}$ by Proposition \ref{prop:WDkl}(a) and Proposition \ref{pro:decomposable}(a). Therefore, since $W_{\mathscr{D}'}\subseteq \Cr(X,D)$ by Setup \ref{setup:allcases}, by Step 2 and by Lemma \ref{lem:weylpreservesnefgen1}(b) we have
	$$ \widehat W_{\mathscr D_{k,m-k-1}} \cdot \mathcal R_k = (W_{\mathscr{D}'}\cdot W_{\mathscr D_{k,m-k-1}})\cdot\mathcal R_k\subseteq W_{\mathscr{D}'}\cdot \mathcal N_X\subseteq \cal{N}_X. $$

\medskip

\emph{Step 4.}
In this step we assume that $R\in \{ A_7^\smone, D_7^\smone \}$. Let $v \in W_{\mathscr D}$ be a word of length $\ell$ such that $\ell\leq m$. In this step we show that
$$ \textstyle \cal N_X \cap v\big(\overline{\cal F(R^\smperp)}\big) \subseteq \bigcup_{k=0}^{m-1} W_{\mathscr D_{k,m-k-1}} \cdot \mathcal R_k. $$

To that end, fix $\gamma\in\cal N_X \cap v\big(\overline{\cal F(R^\smperp)}\big)$. Then by Corollary \ref{cor:finiteness}$\mathrm{(b)}_1$ we have $\gamma\in\cal N_X \cap  \big(v(F_0)\cup v(F_1)\big)$. If $\ell = 0$, then $v$ is the identity, hence Lemma \ref{lem:equalities_of_faces}(a)(b) gives
\begin{align*}
\gamma &\in  \cal N_X \cap (F_0 \cup F_1) = \cal N_X \cap (F^0_{0,m-1} \cup F^1_{m-1,0}) \subseteq \mathcal{R}_0 \cup \mathcal{R}_{m-1} \\
 &\subseteq \bigcup_{k=0}^{m-1} W_{\mathscr D_{k,m-k-1}} \cdot \mathcal R_k.
\end{align*}

Therefore, we may assume that $\ell >0$. By Lemma \ref{lem:equalities_of_faces}(c)(d) we have $ v(F_0)\cup v(F_1)\subseteq F_{\ell,0}^0\cup F_{\ell-1,0}^0\cup F_{0,\ell}^1\cup F_{0,\ell-1}^1 $, and thus,
\begin{equation}\label{eq:8765}
\gamma\in\cal N_X \cap (F_{\ell,0}^0\cup F_{\ell-1,0}^0\cup F_{0,\ell}^1\cup F_{0,\ell-1}^1).
\end{equation} 

Assume first that $\ell<m$. Then by \eqref{eq:8765} and by Lemma \ref{lem:equalities_of_faces}(a)(b) we have
\begin{align*}
\gamma&\in\cal N_X \cap (F_{\ell,m-\ell-1}^0\cup F_{\ell-1,m-\ell}^0\cup F_{m-\ell-1,\ell}^1\cup F_{m-\ell,\ell-1}^1)\\
&\textstyle \subseteq \mathcal R_\ell\cup\mathcal R_{\ell-1}\cup\mathcal R_{m-\ell-1}\cup\mathcal R_{m-\ell}\subseteq \bigcup_{k=0}^{m-1} W_{\mathscr D_{k,m-k-1}} \cdot \mathcal R_k.
\end{align*}
Now assume that $\ell=m$. Then \eqref{eq:8765} becomes
$$ \gamma\in\cal N_X \cap (F_{m,0}^0\cup F_{m-1,0}^0\cup F_{0,m}^1\cup F_{0,m-1}^1).$$
If $\gamma\in\cal N_X \cap (F_{m-1,0}^0\cup F_{0,m-1}^1)$, then $\gamma\in\mathcal R_{m-1}\cup\mathcal R_0\subseteq \bigcup_{k=0}^{m-1} W_{\mathscr D_{k,m-k-1}} \cdot \mathcal R_k$.
Thus, we may assume that $\gamma\in\cal N_X \cap (F_{m,0}^0\cup F_{0,m}^1)$. 

Assume first that $\gamma\in\cal N_X \cap F_{m,0}^0$. By Corollary \ref{cor:finiteness}$\mathrm{(b)}_3$ there exists $\eta\in \mathcal N_X\cap F_{0,m-1}^1$, and in particular, $\eta\in\mathcal R_0$. Since $r_0\in W_{\mathscr D_{0,m-1}}$, by Step 2 we have
\begin{equation}\label{eq:67890}
r_0(\eta)\in W_{\mathscr D_{0,m-1}}\cdot\mathcal R_0\subseteq\mathcal N_X.
\end{equation}
By Proposition \ref{prop:W_invariant_pairing} and since $\eta\in F^1_{0,m-1}$, we have 
$$ \big\langle r_0(\eta), mc + \alpha_0^\smvee\big\rangle = \big\langle \eta, r_0(mc + \alpha_0^\smvee)\big\rangle = \big\langle \eta, (m-1)c + \alpha_1^\smvee \big\rangle = 0, $$
and by Proposition \ref{prop:W_invariant_pairing}, by \eqref{eq:gamma} and since $\eta\in F^1_{0,m-1}$, we have
$$ \big\langle r_0(\eta), \alpha_1^\smvee\big\rangle = \big\langle \eta, r_0(\alpha_1^\smvee)\big\rangle =\textstyle\langle \eta, c + \alpha_0^\smvee \rangle = \langle \eta, c \rangle + \langle \eta, \alpha_0^\smvee \rangle  < 0. $$
Therefore, $r_0(\eta)\in F^0_{m,0}$, hence $r_0(\eta)\in\mathcal N_X\cap F^0_{m,0}$ by \eqref{eq:67890}. As $|\mathcal N_X\cap F^0_{m,0}|=1$ by Corollary \ref{cor:finiteness}$\mathrm{(b)}_3$, we conclude that $\gamma=r_0(\eta)$, hence \eqref{eq:67890} yields
$$\textstyle \gamma=r_0(\eta)\in W_{\mathscr D_{0,m-1}}\cdot\mathcal R_0\subseteq \bigcup_{k=0}^{m-1} W_{\mathscr D_{k,m-k-1}} \cdot \mathcal R_k.$$

Finally, assume that $\gamma\in\cal N_X \cap F_{0,m}^1$. By Corollary \ref{cor:finiteness}$\mathrm{(b)}_3$ there exists $\zeta\in \mathcal N_X\cap F_{m-1,0}^0$, and in particular, $\zeta\in\mathcal R_{m-1}$. Since $r_1\in W_{\mathscr D_{m-1,0}}$, by Step 2 we have
\begin{equation}\label{eq:67890a}
r_1(\zeta)\in W_{\mathscr D_{m-1,0}}\cdot\mathcal R_{m-1}\subseteq\mathcal N_X.
\end{equation}
By Proposition \ref{prop:W_invariant_pairing} and since $\zeta\in F^0_{m-1,0}$, we have 
$$ \big\langle r_1(\zeta), mc + \alpha_1^\smvee\big\rangle = \big\langle \zeta, r_1(mc + \alpha_1^\smvee)\big\rangle = \big\langle \zeta, (m-1)c + \alpha_0^\smvee \big\rangle = 0, $$
and by Proposition \ref{prop:W_invariant_pairing}, by \eqref{eq:gamma} and since $\zeta\in F^0_{m-1,0}$, we have
$$ \big\langle r_1(\zeta), \alpha_0^\smvee\big\rangle = \big\langle \zeta, r_1(\alpha_0^\smvee)\big\rangle =\textstyle\langle \zeta, c + \alpha_1^\smvee \rangle = \langle \zeta, c \rangle + \langle \zeta, \alpha_1^\smvee \rangle  < 0. $$
Therefore, $r_1(\zeta)\in F^1_{0,m}$, hence $r_1(\zeta)\in\mathcal N_X\cap F^1_{0,m}$ by \eqref{eq:67890a}. As $|\mathcal N_X\cap F^1_{0,m}|=1$ by Corollary \ref{cor:finiteness}$\mathrm{(b)}_3$, we conclude that $\gamma=r_1(\zeta)$, hence \eqref{eq:67890a} yields
$$\textstyle \gamma=r_1(\zeta)\in W_{\mathscr D_{m-1,0}}\cdot\mathcal R_{m-1}\subseteq \bigcup_{k=0}^{m-1} W_{\mathscr D_{k,m-k-1}} \cdot \mathcal R_k,$$
as desired.

\medskip

\emph{Step 5.}
In this step we assume that $R=A_6^\smone$. Let $v \in W_{\mathscr D}$ be a word of length $\ell$ such that $\ell\leq m$. In this step we show that
$$ \textstyle \cal N_X \cap v\big(\overline{\cal F(R^\smperp)}\big) \subseteq \bigcup_{k=0}^{m-1} \widehat W_{\mathscr D_{k,m-k-1}} \cdot \mathcal R_k. $$

To that end, fix $\gamma\in\cal N_X \cap v\big(\overline{\cal F(R^\smperp)}\big)$. Then by Corollary \ref{cor:finiteness}$\mathrm{(c)}_1$ we have $\gamma\in\cal N_X \cap  \big(v(F_0)\cup v(F_1)\big)\cap \overline{\cal{F}_{\mathscr{D}'}}$. If $\ell = 0$, then we conclude as in Step 4. Thus, we may assume that $\ell>0$, and as in Step 4 we have
\begin{equation}\label{eq:8765b}
\gamma\in\cal N_X \cap (F_{\ell,0}^0\cup F_{\ell-1,0}^0\cup F_{0,\ell}^1\cup F_{0,\ell-1}^1)\cap \overline{\cal{F}_{\mathscr{D}'}}.
\end{equation} 

If $\ell<m$, then similarly as in Step 4, \eqref{eq:8765b} yields
\begin{align*}
\gamma&\in\cal N_X \cap (F_{\ell,m-\ell-1}^0\cup F_{\ell-1,m-\ell}^0\cup F_{m-\ell-1,\ell}^1\cup F_{m-\ell,\ell-1}^1)\cap \overline{\cal{F}_{\mathscr{D}'}}\\
&\textstyle \subseteq \mathcal R_\ell\cup\mathcal R_{\ell-1}\cup\mathcal R_{m-\ell-1}\cup\mathcal R_{m-\ell}\subseteq \bigcup_{k=0}^{m-1} \widehat W_{\mathscr D_{k,m-k-1}} \cdot \mathcal R_k.
\end{align*}

Now assume that $\ell=m$. If $\gamma\in\cal N_X \cap (F_{m-1,0}^0\cup F_{0,m-1}^1)\cap \overline{\cal{F}_{\mathscr{D}'}}$, then $\gamma\in\mathcal R_{m-1}\cup\mathcal R_0\subseteq \bigcup_{k=0}^{m-1} \widehat W_{\mathscr D_{k,m-k-1}} \cdot \mathcal R_k$.
Thus, we may assume that $\gamma\in\cal N_X \cap (F_{m,0}^0\cup F_{0,m}^1)\cap \overline{\cal{F}_{\mathscr{D}'}}$.

Assume first that $\gamma\in\cal N_X \cap F_{m,0}^0\cap \overline{\cal{F}_{\mathscr{D}'}}$. Then the proof is similar to Step 4. By Corollary \ref{cor:finiteness}$\mathrm{(c)}_3$ there exists $\eta\in \mathcal N_X\cap F_{0,m-1}^1\cap \overline{\cal{F}_{\mathscr{D}'}}$, and in particular, $\eta\in\mathcal R_0$. Since $r_0\in \widehat W_{\mathscr D_{0,m-1}}$, by Step 3 we have
\begin{equation}\label{eq:67890bb}
r_0(\eta)\in \widehat W_{\mathscr D_{0,m-1}}\cdot\mathcal R_0\subseteq\mathcal N_X.
\end{equation}
As in Step 4 we show that $r_0(\eta)\in F^0_{m,0}$. Since $r_0\in W_{\mathscr D}$ and $\eta\in \overline{\cal{F}_{\mathscr{D}'}}$, we have $r_0(\eta)\in \overline{\cal{F}_{\mathscr{D}'}}$ by Proposition \ref{pro:decomposable}(e), hence $r_0(\eta)\in\mathcal N_X\cap F^0_{m,0}\cap \overline{\cal{F}_{\mathscr{D}'}}$ by \eqref{eq:67890bb}. As $|\mathcal N_X\cap F^0_{m,0}\cap \overline{\cal{F}_{\mathscr{D}'}}|=1$ by Corollary \ref{cor:finiteness}$\mathrm{(c)}_3$, we conclude that $\gamma=r_0(\eta)$, hence \eqref{eq:67890bb} yields
$$\textstyle \gamma=r_0(\eta)\in \widehat W_{\mathscr D_{0,m-1}}\cdot\mathcal R_0\subseteq \bigcup_{k=0}^{m-1} \widehat W_{\mathscr D_{k,m-k-1}} \cdot \mathcal R_k.$$

Finally, assume that $\gamma\in\cal N_X \cap F_{0,m}^1\cap \overline{\cal{F}_{\mathscr{D}'}}$. Then the proof is analogous to the last paragraph of Step 4.

\medskip

   \emph{Step 6.}
   Finally, in this step we show that
   $$ \textstyle \cal N_X \subseteq \bigcup_{k=0}^{m-1} \widehat W_{\mathscr D_{k,m-k-1}} \cdot \mathcal R_k. $$
This, together with Step 3, will finish the proof of the lemma.

Let $\beta \in \cal N_X$. Then $\beta\in \cal T (R^\smperp)$ by Lemma \ref{lem:effconeHalpensurface1}(e), hence by Proposition \ref{prop:fund_domain_Cr(X,D)_Tits_cone} there exists an element $w \in \Cr(X,D)$ such that $w(\beta) \in \Omega_R$. Moreover, as $\Cr(X,D)$ preserves $\cal N_X$ by Lemma \ref{lem:weylpreservesnefgen1}(b), we obtain
\begin{equation}\label{eq:move_to_fund_dom1}
    w(\beta) \in \Omega_R \cap \cal N_X.
\end{equation}

By Lemma \ref{lem:OmegaR} there are finitely many elements $u_1,\dots,u_{2m}\in W_{\mathscr D}$ of length at most $m$, such that
$$ \textstyle \Omega_R = \bigcup_{j=1}^{2m} u_j \big(\overline{\cal F(R^\smperp)}\big). $$
By this equality and by Steps 4 and 5 we have $ \Omega_R \cap \cal{N}_X \subseteq \textstyle \bigcup_{k=0}^{m-1} \widehat W_{\mathscr D_{k,m-k-1}} \cdot \mathcal R_k$, hence \eqref{eq:move_to_fund_dom1} gives
$$ \textstyle w(\beta) \in \bigcup_{k=0}^{m-1} \widehat W_{\mathscr D_{k,m-k-1}} \cdot \mathcal R_k. $$
Since $w^{-1}\in \Cr(X,D)\subseteq \widehat W_{\mathscr D_{k,m-k-1}}$ for each $k\in\{0,\dots,m-1\}$ by Setup \ref{setup:allcases} and by Proposition \ref{prop:fund_domain_G}(a), we conclude that $\beta \in \bigcup_{k=0}^{m-1} \widehat W_{\mathscr D_{k,m-k-1}} \cdot \mathcal R_k$, as desired.
\end{proof}

Finally, we can prove Theorem \ref{thm:mainthmsection8}.

\begin{proof}[Proof of Theorem \ref{thm:mainthmsection8}]
Part (c) follows immediately from Lemma \ref{lem:alternative_unique_gamma0}. In the remainder of the proof we show (a) and (b).

By Lemma \ref{lem:OmegaR} there are finitely many elements $u_1,\dots,u_{2m}\in W_{\mathscr D}$ of length at most $m$, such that $ \textstyle \Omega_R = \bigcup_{j=1}^{2m} u_j \big(\overline{\cal F(R^\smperp)}\big) $, hence
$$ \textstyle \mathcal N_X\cap\Omega_R = \bigcup_{j=1}^{2m} \big(\mathcal N_X\cap u_j \big(\overline{\cal F(R^\smperp)}\big)\big). $$
Therefore, the set $\mathcal N_X\cap\Omega_R$ is finite by Corollary \ref{cor:finiteness}(a), which shows (a).

Now, we have $\mathcal N_X\subseteq\mathcal T(R^\smperp)$ by Lemma \ref{lem:effconeHalpensurface1}(e), and the group $\Cr(X,D)$ preserves $\mathcal{N}_X$ by Lemma \ref{lem:weylpreservesnefgen1}(b). Therefore, by Proposition \ref{prop:fund_domain_Cr(X,D)_Tits_cone}, for every $x\in\mathcal N_X$ there exists $w\in \Cr(X,D)$ such that $w(x)\in \mathcal{N}_X\cap\Omega_R$. Since the set $\mathcal N_X\cap\Omega_R$ is finite by (a), part (b) immediately follows.
\end{proof}

\section{Proof of Theorem \ref{thm:main2}}\label{sec:proofA}

In this section we prove our first main result, Theorem \ref{thm:main2}.

We will construct a fundamental domain for the action of the group of Cremona isometries explicitly: in the notation from Setup \ref{setup:allcases}, the desired fundamental domain is $\Omega_R\cap\Nef(X)$. First we show in Proposition \ref{prop:ratpoly} that this cone is rational polyhedral. Then we show in Theorem \ref{thm:existencefunddom} that this cone is a desired fundamental domain.

We start with the following two technical propositions, which essentially prove Proposition \ref{prop:ratpoly}. The proof of the following proposition is partly inspired by the proof of \cite[Proposition 4.8]{Li25}.

\begin{prop}\label{prop:ratpol1}
   Let $X$ be a generic Sakai surface of type $R$, and assume that $ R\notin\{A_6^\smone, A_7^\smone, A_8^\smone, D_7^\smone, D_8^\smone, E_8^\smone \} $. Then the cone $ \overline{\mathcal F(R^\smperp)}\cap \Nef(X) $ is rational polyhedral.
\end{prop}

\begin{proof}
	Denote
$$ \Xi\coloneqq \overline{\mathcal F(R^\smperp)}\cap \Nef(X). $$
    By Corollary \ref{cor:finite_(-1)_curves_chamber1}(c) the set
    $$\mathcal S\coloneqq \big(\overline{\mathcal F(R^\smperp)}\cap \mathcal{N}_X\big)\cup \Comp(D)$$
    is finite, and denote
$$ \mathcal P\coloneqq \overline{\mathcal F(R^\smperp)}\cap \{x\in N^1(X)_\R \mid x\cdot \gamma\geq 0\text{ for any }\gamma\in\mathcal S\}.$$
Then $\mathcal{P}$ is a rational polyhedral cone as it is the intersection of two rational polyhedral cones.

Since $\mathcal N_X\subseteq\Eff(X)$, we have that $\mathcal S\subseteq\Eff(X)$, and therefore $\Nef(X)\subseteq \{x\in N^1(X)_\R \mid x\cdot \gamma\geq 0\text{ for any }\gamma\in\mathcal S\}$. This implies that $\Xi\subseteq \mathcal P$.

In the remainder of the proof we show that $ \mathcal P\subseteq\Xi $. This will then imply that $\Xi=\mathcal P$, and therefore prove the proposition.
    
Let $x\in \mathcal P$. In particular,
\begin{equation}\label{eq:909}
x\cdot \gamma\geq0 \quad\text{for all }\gamma\in\mathcal S.
\end{equation}
To show that $x \in \Xi$, it suffices to show that $x$ is nef. Therefore, by Lemma \ref{lem:effconeHalpensurface1}(d), we have to show that $x\cdot \beta \geq 0$ for any $\beta \in \mathcal{N}_X\cup\Comp(D)$. By \eqref{eq:909} it follows immediately that $x\cdot\beta\geq0$ for any $\beta\in\Comp(D)$.

Thus, it remains to show that $x\cdot \beta \geq 0$ for any $\beta\in \mathcal{N}_X$. To that end, fix $\beta\in \mathcal{N}_X$. 
By Lemma \ref{lem:weylpreservesnefgen1}(f) there exist $w\in W(R^\smperp)$ and $\zeta\in (\overline{\mathcal F(R^\smperp)}\cap \mathcal{N}_X\big) \subseteq \mathcal S$ such that $\beta = w(\zeta)$. Since $x\in \overline{\mathcal F(R^\smperp)}$, by Lemma \ref{lem:actionW} we can write $w^{-1}(x) = x + y$, where $y$ is a nonnegative linear combination of the elements from $\Pi(R^\smperp)$. Since $\zeta\in\overline{\mathcal F(R^\smperp)}$, we have $y\cdot \zeta \geq 0$ by the definition of $\overline{\mathcal F(R^\smperp)}$. Therefore, by Proposition \ref{prop:W_invariant_pairing} and by \eqref{eq:909} we have
    $$ x \cdot \beta = x \cdot w(\zeta) = w^{-1}(x)\cdot \zeta = x\cdot \zeta + y\cdot \zeta \geq 0 . $$
as desired.
\end{proof}

\begin{prop}\label{prop:ratpol_ugly_cases}
   Let $X$ be a generic Sakai surface whose type $R$ is $A_6^\smone$, $A_7^\smone$ or $D_7^\smone$. Assume Setup \ref{setup:allcases}. Then for any element $w\in W_{\mathscr{D}}$, the cone $ w\big(\overline{\cal F(R^\smperp)}\big) \cap  \Nef(X) $ is rational polyhedral.
\end{prop}

\begin{proof}
We fix $w\in  W_{\mathscr D}$, and denote
$$ \Xi \coloneqq w\big(\overline{\cal F(R^\smperp)}\big) \cap  \Nef(X). $$
Let $ m $ be as in Setup \ref{setup:allcases}. For every $k\in \{ 0,\dots, m-1\}$ we use the notation $\widehat W_{\mathscr D_{k,m-k-1}}$ and $\widehat{\mathcal F}_{\mathscr D_{k,m-k-1}}$ from Setup \ref{setup:allcases}.

\medskip

\emph{Step 1.}
By Proposition \ref{prop:WDkl}(d), for every $k\in \{0,\dots, m-1\}$ there exists $u_k\in W_{\mathscr{D}_{k,m-k-1}}$ such that 
$$ \textstyle w \big( \overline{\cal{F}_{\mathscr{D}}} \big) = \bigcap_{k=0}^{m-1} u_k \big(\overline{\cal{F}_{\mathscr{D}_{k,m-k-1}}}\big) , $$
hence by Proposition \ref{pro:decomposable}(e)(f) we have
\begin{align}\label{eq:intersect_two_chambers11}
    w \big( &\overline{\cal F(R^\smperp)} \big) = \textstyle w(\overline{\cal{F}_{\mathscr{D}}} )  \cap \overline{\cal{F}_{\mathscr{D}'}} = \bigcap_{k=0}^{m-1} \big(u_k (\overline{\cal{F}_{\mathscr{D}_{k,m-k-1}}})\cap \overline{\cal{F}_{\mathscr{D}'}}\big) \\
    &= \textstyle \bigcap_{k=0}^{m-1} \big(u_k (\overline{\cal{F}_{\mathscr{D}_{k,m-k-1}}})\cap u_k(\overline{\cal{F}_{\mathscr{D}'}})\big) = \bigcap_{k=0}^{m-1} u_k (\widehat{\mathcal F}_{\mathscr D_{k,m-k-1}}) .\notag
\end{align}
By Lemma \ref{lem:alternative_unique_gamma0}, for each $k\in\{ 0,\dots, m-1\}$ there exists a finite subset $\mathcal R_k\subseteq\mathcal N_X\cap \widehat{\mathcal F}_{\mathscr D_{k,m-k-1}}$ such that
\begin{equation}\label{eq:999o}
\textstyle \cal N_X = \bigcup_{k=0}^{m-1} \widehat W_{\mathscr D_{k,m-k-1}} \cdot \mathcal R_k.
\end{equation}
Set
$$ \mathcal S_k \coloneqq u_k(\mathcal R_k), $$
and denote
    $$\textstyle \mathcal S\coloneqq \bigcup_{k=0}^{m-1}\mathcal S_k\cup \Comp(D).$$
Then $\mathcal S$ is finite, and $\mathcal S\subseteq \cal{N}_X\cup \Comp(D)$ by \eqref{eq:999o}.

\medskip

\emph{Step 2.}
   Denote
$$ \mathcal P\coloneqq w\big(\overline{\mathcal{F}(R^\smperp)}\big) \cap \{x\in N^1(X)_\R \mid x\cdot \gamma\geq 0\text{ for any }\gamma\in\mathcal S\}.$$
Then $\mathcal{P}$ is a rational polyhedral cone as it is the intersection of two rational polyhedral cones.
    
Since $\mathcal N_X\subseteq\Eff(X)$, we have that $\mathcal S\subseteq\Eff(X)$, and therefore $\Nef(X)\subseteq \{x\in N^1(X)_\R \mid x\cdot \gamma\geq 0\text{ for any }\gamma\in\mathcal S\}$. This implies that $\Xi\subseteq \mathcal P$.

\medskip

\emph{Step 3.}
In the remainder of the proof we show that $\mathcal P\subseteq\Xi$. Together with Step 2, this will then imply that $\Xi=\mathcal P$, and therefore prove the proposition.

Let $x\in \mathcal P$. In particular,
\begin{equation}\label{eq:909aa}
x\cdot \gamma\geq0 \quad\text{for all }\gamma\in\mathcal S.
\end{equation}
To show that $x \in \Xi$, it suffices to show that $x$ is nef. Therefore, by Lemma \ref{lem:effconeHalpensurface1}(d), we have to show that $x\cdot \beta \geq 0$ for any $\beta\in \mathcal{N}_X\cup\Comp(D)$. By \eqref{eq:909aa} it follows immediately that $x\cdot\beta\geq0$ for any $\beta\in\Comp(D)$.

Thus, it remains to show that $x\cdot \beta \geq 0$ for any $\beta\in \mathcal{N}_X$. We will show this in the next step.

\medskip

\emph{Step 4.}
Fix $\beta\in \mathcal{N}_X$. By \eqref{eq:999o} there exist $\ell\in\{0,\dots,m-1\}$, $w_\ell\in \widehat W_{\mathscr{D}_{\ell,m-\ell-1}}$ and $\zeta\in\mathcal R_\ell$ such that 
\begin{equation}\label{eq:beta_translate}
\beta = w_\ell ( \zeta) .
\end{equation}
It follows from \eqref{eq:intersect_two_chambers11} that $x \in u_\ell( \widehat{\mathcal F}_{\mathscr D_{\ell,m-\ell-1}}) $, hence $(u_\ell)^{-1}(x) \in \widehat{\mathcal F}_{\mathscr D_{\ell,m-\ell-1}} $. Note that when $R=A_6^\smone$, the set $\widehat{\mathcal F}_{\mathscr D_{\ell,m-\ell-1}}$ is the negative fundamental chamber and $\widehat W_{\mathscr{D}_{\ell,m-\ell-1}}$ is the Weyl group of the set of root data from Lemma \ref{lem:wandw'invariant}(b). Therefore, by Lemma \ref{lem:actionW} applied to $w_\ell^{-1}u_\ell\in \widehat W_{\mathscr{D}_{\ell,m-\ell-1}}$ and $(u_\ell)^{-1}(x) \in \widehat{\mathcal F}_{\mathscr D_{\ell,m-\ell-1}} $, we obtain 
\begin{equation}\label{eq:sum_x_and_y}
(w_\ell)^{-1}(x) = (u_\ell)^{-1}(x) + y,
\end{equation}
where $y$ is a nonnegative linear combination of $\ell\delta + \alpha_0$ and $(m-\ell-1)\delta + \alpha_1 $ when $R\in\{A_7^\smone,D_7^\smone\}$, and $y$ is a nonnegative linear combination of $\ell\delta + \alpha_0$, $(m-\ell-1)\delta + \alpha_1 $, $\alpha_0'$ and $\alpha_1'$ when $R=A_6^\smone$. Since $\zeta\in\mathcal R_\ell\subseteq\widehat{\mathcal F}_{\mathscr D_{\ell,m-\ell-1}}$, we have
\begin{equation}\label{eq:9090a}
y\cdot \zeta \geq 0.
\end{equation}
Furthermore, since $u_\ell(\zeta)\in\mathcal S_\ell\subseteq\mathcal S$, by Proposition \ref{prop:W_invariant_pairing} and \eqref{eq:909aa} we have
\begin{equation}\label{eq:9090b}
(u_\ell)^{-1}(x) \cdot \zeta = x \cdot u_\ell (\zeta) \geq 0.
\end{equation}
Therefore, by \eqref{eq:beta_translate}, \eqref{eq:sum_x_and_y}, \eqref{eq:9090a} and \eqref{eq:9090b} and by Proposition \ref{prop:W_invariant_pairing} we obtain
    $$ x \cdot \beta = x \cdot w_\ell (\zeta) = (w_\ell)^{-1}(x)\cdot \zeta = (u_\ell)^{-1}(x)\cdot \zeta + y\cdot \zeta \geq 0, $$
as desired.
\end{proof}

As a consequence of the previous two propositions, we can prove the first main technical result of the section.

\begin{prop}\label{prop:ratpoly}
     Let $X$ be a generic Sakai surface of type $R$ which is different from $ A_8^\smone$, $D_8^\smone$ and $E_8^\smone$. Assume Setup \ref{setup:allcases}. Then the cone $ \Omega_R \cap \Nef(X) $ is rational polyhedral.
\end{prop}

\begin{proof}
If $ R\notin\{A_6^\smone, A_7^\smone, A_8^\smone, D_7^\smone, D_8^\smone, E_8^\smone \} $, then the result follows from Proposition \ref{prop:ratpol1}.

Now assume that $ R\in\{A_6^\smone, A_7^\smone, D_7^\smone \} $. By Lemma \ref{lem:OmegaR} there are finitely many elements $u_1,\dots,u_{2m}\in W_{\mathscr D}$ such that
$$ \textstyle \Omega_R = \bigcup_{j=1}^{2m} u_j\big(\overline{\cal F(R^\smperp)}\big) . $$
By Proposition \ref{prop:ratpol_ugly_cases}, for each $j$ the cone $ \mathcal C_j\coloneqq u_j\big(\overline{\cal F(R^\smperp)}\big) \cap \Nef(X) $ is rational polyhedral. As $\Omega_R \cap  \Nef(X)$ is a convex cone, and since it is the union of the rational polyhedral cones $\mathcal C_1,\dots,\mathcal C_{2m}$, it is generated by the finitely many generators of the cones $\mathcal C_1,\dots,\mathcal C_{2m}$. This finishes the proof.
\end{proof}

Now we can finally show that on a generic Sakai surface $X$ with $D\in|{-}K_X|$, there exists a rational polyhedral fundamental domain for the cone $\Nef^e(X)$ under the action of the group $\Cr(X,D)$. 

\begin{thm}\label{thm:existencefunddom}
     Let $X$ be a Sakai surface of type $R$, and let $D\in|{-}K_X|$. Then the following statements hold.
  \begin{enumerate}[\normalfont (a)]
  \item If $X$ is generic, then there exists a rational polyhedral fundamental domain for the cone $\Nef^e(X)$ under the action of the group $\Cr(X,D)$.
  \item If $R\in\{A_8^\smone, D_8^\smone, E_8^\smone\}$, then  the cone $\Nef^e(X)$ is rational polyhedral and there exists a rational polyhedral fundamental domain for $\Nef^e(X)$ under the action of the group $\Cr(X,D)$.
\end{enumerate}
 \end{thm}

\begin{proof}
We recall first that $\Nef^e(X)=\Nef(X)=\Nef^+(X)$ by Lemma \ref{lem:effconeHalpensurface1}(b). We will use this without explicit mention throughout the proof.

We first show (b). If $R\in\{A_8^\smone, D_8^\smone, E_8^\smone\}$, then by Appendix \ref{sec:Sakai_appendix_A} and by Lemma \ref{lem:lin_indep_-2curves}, the classes $\cal{D}_0,\ldots,\cal{D}_8$ are the only classes of $(-2)$-curves on $X$, and $[{-}K_X]$ is the sum of $\mathcal D_i$ with positive coefficients. Therefore, the cone $\Eff(X)$ is rational polyhedral by \cite[Proposition 4.3(d)]{LH05}, hence its dual cone $\Nef(X)$ is also rational polyhedral. Now (b) follows by \cite[Corollary 6.6]{Laz13}.

In the remainder of the proof we show (a). By (b) we may assume that $ R\notin\{A_8^\smone, D_8^\smone, E_8^\smone \} $. Assume Setup \ref{setup:allcases}, and set
$$ \Xi \coloneqq \Omega_R \cap \Nef(X). $$
Then $\Xi$ is a rational polyhedral cone by Proposition \ref{prop:ratpoly}. We will show that $\Xi$ is a fundamental domain for the cone $\Nef(X)$ under the action of the group $\Cr(X,D)$.

Let $x \in \Nef(X)$. By Lemma \ref{lem:effconeHalpensurface1}(f) we have $\Nef(X) \subseteq \mathcal T(R^\smperp)$, hence by Proposition \ref{prop:fund_domain_Cr(X,D)_Tits_cone} there exists $w \in \Cr(X,D)$ such that $w(x) \in \Omega_R$. Since $\Cr(X,D)$ preserves $\Nef(X)$, we conclude that $w(x) \in \Xi$. This implies that $\Nef(X) \subseteq \Cr(X,D) \cdot \Xi$, and the converse inclusion is trivial. Thus,
$$ \Nef(X) = \Cr(X,D) \cdot \Xi. $$

Finally, let $g\in\Cr(X,D)$ be such that $\operatorname{Int}\big(g(\Xi)\big)\cap\operatorname{Int}(\Xi)\neq\emptyset$. Since 
$$ \operatorname{Int}\big(g(\Xi)\big)\cap\operatorname{Int}(\Xi) = \operatorname{Int}\big(g(\Omega_R)\big)\cap\operatorname{Int}(\Omega_R)\cap\Amp(X), $$
we obtain that $\operatorname{Int}\big(g(\Omega_R)\big)\cap\operatorname{Int}\big(\Omega_R\big)\neq\emptyset$. By Proposition \ref{prop:fund_domain_Cr(X,D)_Tits_cone}, this implies that $g$ is the identity, which concludes the proof.
\end{proof}

Now we can show Theorem \ref{thm:main2}.

\begin{proof}[Proof of Theorem \ref{thm:main2}]
The first part of the statement follows from \cite[Proposition B.4]{LSX26}, whereas the second part follows from Theorem \ref{thm:existencefunddom}.
\end{proof}

\section{Proof of Theorem \ref{thm:mainB}}

In this section we prove our second main result, Theorem \ref{thm:mainB}.

\begin{proof}[Proof of Theorem \ref{thm:mainB}]
We prove this result in several steps. 

\medskip

\emph{Step 1.}
In this step we construct a non-generic Sakai surface $X$ of type $E^{\smone}_7$.

 In \cite[Appendix B]{Sak01}, Sakai gives an explicit realisation of $R$-surfaces of type $E^{\smone}_7$ as blowups of $\pr^2$ at nine points $p_1,\ldots,p_9$. Following Sakai's construction, these surfaces form a family depending on three complex parameters $ a_0$, $a_1$ and $s$, with $ a_0+a_1 \neq 0$. Choose the parameters 
    $$(a_0,a_1,s) := (1,0,1) .$$
    According to \cite[p.\ 187]{Sak01}, in this construction the condition $a_1 =0$ means that the point $p_5$ is infinitely near to $p_4$. Thus, this defines an $R$-surface $X$ of type $E^{\smone}_7$ with the blowup configuration as in the figure below, see \cite[pp.\ 186-187]{Sak01}.\footnote{In alignment with Appendix \ref{sec:Sakai_appendix_A}, we exchange $\mathcal D_0$ and $\mathcal D_6$.} Moreover, by \cite[Lemma B.3(b)]{LSX26}, the surface $X$ is a Sakai surface of type $E^{\smone}_7$.
    
\begin{center}
\scalebox{0.9}
{\begin{tikzpicture}[>=stealth]

\draw[ultra thick] (-1.6,-2) -- (0.16,0.2);
\draw[ultra thick] (-0.13,0.2) -- (1.3,-2);

\draw (0,0) circle (0.12);

\node at (-0.8,-1.8) {$x=0$};
\node at (1.8,-1.8) {$z=0$};

\fill (-0.95,-1.2) circle (2pt);
\node[left] at (-1.2,-1.2) {$p_4$};

\draw (-0.95,-1.2) circle (0.12);
\node[left] at (-2,-1.2) {$p_5$};

\draw[->] (-2.1,-1.2) -- (-1.75,-1.2);

\node at (0.7,0) {$p_1$};
\node at (1.5,0) {$p_2$};
\node at (2.3,0) {$p_3$};
\node at (3.1,0) {$p_6$};

\draw[<-] (0.9,0) -- (1.3,0);
\draw[<-] (1.7,0) -- (2.1,0);
\draw[<-] (2.5,0) -- (2.9,0);
\draw[<-] (3.3,0) -- (3.7,0);

\node at (1.5,-0.5) {$p_7$};
\node at (2.3,-0.5) {$p_8$};
\node at (3.1,-0.5) {$p_9$};

\draw[<-] (0.9,-0.5) -- (1.3,-0.5);
\draw[<-] (1.7,-0.5) -- (2.1,-0.5);
\draw[<-] (2.5,-0.5) -- (2.9,-0.5);

\draw[very thick,<-] (3.5,-1) -- (4.5,-1);

\begin{scope}[shift={(6.8,0)}]

\draw (-1.4,-2) -- (-1.4,1);
\draw (-0.4,-2) -- (-0.4,1);
\draw[ultra thick] (0.6,-2) -- (0.6,1);
\draw (1.6,-2) -- (1.6,1);
\draw[ultra thick] (2.6,-2) -- (2.6,1);
\draw (3.6,-2) -- (3.6,1);

\draw (-0.8,0) -- (2.0,0);

\draw (-1,0) -- (-2.3,0);

\draw (2.2,0) -- (3.8,0);

\draw (-1.8,-1) -- (0.2,-1);
\draw (1,-1) -- (3.0,-1);

\node at (-1.7,-1.8) {$\mathcal{D}_0$};
\node at (-0.7,-1.8) {$\mathcal{D}_4$};
\node at (0.3,-1.8) {$\mathcal{D}_7$};
\node at (1.3,-1.8) {$\mathcal{D}_2$};
\node at (2.3,-1.8) {$\mathcal{D}_6$};
\node at (-0.9,-0.8) {$\mathcal D_5$};
\node at (2.1,-0.8) {$\mathcal D_1$};
\node at (0.1,0.2) {$\mathcal D_3$};
\node at (-1.8,0.2) {$\mathcal E_9$};
\node at (3.2,0.2) {$\mathcal E_5$};
\node at (3.3,-1.8) {$\alpha_1$};

\end{scope}
\end{tikzpicture}}
\end{center}

Since $p_5$ is infinitely near to $p_4$, the class $\cal{E}_4 - \cal{E}_5$ is the class of a $({-2})$-curve on $X$. Let $D\in|{-}K_X|$. By Appendix \ref{sec:Sakai_appendix_A}, this $({-}2)$-curve is not a component of $D$, hence $\cal{E}_4 - \cal{E}_5 \in \Delta^{\nod}_{X,D}$ by \cite[Lemma 2.5]{LSX26}. Therefore, $X$ is a non-generic Sakai surface.

\medskip

\emph{Step 2.} In this step we show that $ \Delta^{\nod}_{X,D} =  \{\cal{E}_4 - \cal{E}_5\}$.

First note that by Appendix \ref{sec:Sakai_appendix_A} we have $\Comp(D) = \{ \mathcal{D}_0,\ldots, \mathcal{D}_7\}$, where each $\mathcal{D}_i \in N^1(X)$ is the class of a  $({-}2)$-curve. Hence, $X$ contains nine distinct classes  $ \mathcal{D}_0,\ldots, \mathcal{D}_7,\alpha_1$ of $({-}2)$-curves. By Lemma \ref{lem:lin_indep_-2curves}, there are no other classes of $({-}2)$-curves on $X$, which proves the claim.

\medskip

\emph{Step 3.} 
In this step we show that $\Cr(X,D) = \{ \id \}$. 

By \S\ref{subsec:Rperp1} and Appendix \ref{sec:Sakai_appendix_A}, the set of root data $R^{\smperp}$ is of affine type $A^{\smone}_1$, with $\Pi(R^{\smperp}) = \{ \alpha_0, \alpha_1\}$ and $W(R^\smperp ) = \langle r_0 , r_1 \rangle$. Since $\alpha_1=\cal{E}_4 - \cal{E}_5$, we have $\Delta^{\nod}_{X,D} = \{ \alpha_1 \}$ by Step 2. It follows from \cite[Theorem 26(a)]{Sak01} that $\Cr(X,D)$ is the stabiliser of the element $\alpha_1$ in $W(R^{\smperp})$.

Let $w\in W(R^{\smperp})$. By Lemma \ref{lem:Wklcalculation}(c), there exists $n\in \Z$ such that $w\in\big\{(r_0r_1)^n, r_1(r_0r_1)^n\big\}$. If $w=(r_0r_1)^n$, then by Lemma \ref{lem:r1r0_n}(a) we have
\begin{equation}\label{eq:elementsW}
    w(\alpha_1) = \alpha_1 - 2n \delta, 
\end{equation}
hence $w(\alpha_1)=\alpha_1$ if and only if $n=0$. If $w=r_1(r_0r_1)^n$, then by Lemma \ref{lem:r1r0_n}(c)(d) we have
\begin{equation}\label{eq:elementsW1}
     w(\alpha_1) = r_{n\delta+\alpha_1} (\alpha_1) = \alpha_0 - (2n+1)\delta = -2n \alpha_0 - (2n+1)\alpha_1,
\end{equation}
hence $w(\alpha_1)\neq\alpha_1$. This shows the claim.

\medskip

\emph{Step 4.}
In this step we show that $|\Cr(X)|\leq2$.

Let $\sigma \in \Cr(X)$. Let $\alpha\in \Eff(X)$ be the numerical class of a $({-}2)$-curve on $X$. Since $\sigma(\alpha)\in\Eff(X)$ by the definition of the group $\Cr(X)$, there exists an effective $\R$-divisor $G$ on $X$ such that $\sigma(\alpha)=[G]$. Since $\alpha^2<0$, the ray $\R_{\geq0}\alpha$ is an extremal ray of the cone $\Eff(X)$ by \cite[Lemma 2.4]{LSX26}, and therefore, the ray $\R_{\geq0}\sigma(\alpha)$ is an extremal ray of the cone $\Eff(X)$. This forces the support of $G$ to be irreducible, so
we can write $G = fF$, where $f > 0$ and $F$ is a prime divisor on $X$. Note that $f$ is a rational number since $\sigma(\alpha)\in N^1(X)$ by the definition of the group $\Cr(X)$. Since $\sigma$ preserves the intersection form on $X$, we have
$$f^2F^2 = G^2 = \sigma(\alpha)^2 = \alpha^2 = {-}2.$$
This is only possible if $f=1$ and $F^2={-}2$. Therefore, $\sigma(\alpha)$ is the numerical class of a $({-}2)$-curve.

Therefore, by Step 2 we have
$$\big\{\sigma(\mathcal D_0),\dots,\sigma(\mathcal D_7),\sigma(\mathcal E_4-\mathcal E_5)\big\}=\{\mathcal D_0,\dots,\mathcal D_7,\mathcal E_4-\mathcal E_5\}.$$
Moreover, since $\sigma$ preserves the intersection form on $X$, we have
$$\sigma(\mathcal D_i)\cdot\sigma(\mathcal E_4-\mathcal E_5)=\mathcal D_i\cdot(\mathcal E_4-\mathcal E_5)=0\quad\text{for each }i=0,\dots,7.$$
Therefore, $\sigma(\mathcal E_4-\mathcal E_5)$ is an element of the set $\{\mathcal D_0,\dots,\mathcal D_7,\mathcal E_4-\mathcal E_5\}$ which is orthogonal to every other element in that set. But this forces $\sigma(\cal{E}_4 - \cal{E}_5) = \cal{E}_4 - \cal{E}_5$. Consequently, $\sigma$ permutes the components $\mathcal{D}_0,\ldots, \mathcal{D}_7$. Since $\sigma$ preserves the intersection form on $X$, this permutation induces an automorphism of the associated Dynkin diagram of type $E^\smone_7$. 

Therefore, we obtain the group homomorphism $ \varphi\colon \Cr(X) \rightarrow \Aut(E^\smone_7)$, whose kernel is $\{ \sigma \in \Cr(X) \mid \sigma(\mathcal{D}_i) = \mathcal{D}_i \text{ for } i=0,\ldots,7\} = \Cr(X,D)$. Since $\Cr(X,D) = \{ \id \}$ by Step 3, the map $\varphi$ is injective. It is easy to see from the shape of the Dynkin diagram of type $E^\smone_7$ in \cite[p.\ 44]{Kac90} that $\Aut(E^\smone_7) \simeq \Z/2\Z$, which shows the claim.

\medskip

\emph{Step 5.}
In this step, we show that the cone $\Nef(X)$ is not rational polyhedral.

By duality, it suffices to show that the cone $\overline{\Eff}(X)$ is not rational polyhedral. Since the numerical classes of distinct $({-}1)$-curves span distinct extremal rays of $\overline{\Eff}(X)$ by \cite[Lemma 2.4]{LSX26}, it suffices to show that $X$ contains infinitely many numerical classes of $({-}1)$-curves. For every $n\geq 1$, set
$$\cal G_n \coloneqq \cal{E}_5 - (n-1)\alpha_1 + n(n-1)\delta \in N^1(X).$$
It suffices to show that $\{\mathcal{G}_n\}_{n\geq1}$ is a sequence of distinct numerical classes of $({-}1)$-curves on $X$.

By Appendix \ref{sec:Sakai_appendix_A} we have
\begin{equation}\label{eq:severalequalities}
\mathcal E_5\cdot\alpha_1=\mathcal{E}_5\cdot\delta=1,\quad\alpha_1^2={-}2,\quad \alpha_1\cdot\delta=\delta^2=0, 
\end{equation}
and
\begin{equation}\label{eq:severalequalities1}
\alpha_1\cdot\mathcal D_i=\delta\cdot\mathcal D_i=0\quad\text{and}\quad \cal{E}_5 \cdot \cal{D}_i \geq 0\quad\text{for each }\cal{D}_i\in \Comp (D).
\end{equation}
We first show that each $\mathcal G_n$ is the class of a $({-}1)$-curve. Fix $n\geq1$. By Lemma \ref{lem:irredclasses} we have to show that $\mathcal G_n\in \mathcal M_X$, and that $\mathcal G_n\cdot C \geq 0$ for all $C \in \Comp(D) \cup \Delta^{\nod}_{X,D}$. To that end, by \eqref{eq:severalequalities} we have $\cal{G}_n \cdot \delta = 1$ and
$$\cal{G}_n^2 = {-}1 - 2(n-1)-2(n-1)^2+ 2n(n-1) = {-}1,$$
hence $\cal G_n\in\mathcal M_X$ by Lemma \ref{lem:-1classesHalphen}(a). By \eqref{eq:severalequalities1} we have $\cal G_n \cdot \cal{D}_i = \cal{E}_5 \cdot \cal{D}_i \geq 0$ for each $\cal{D}_i\in \Comp (D)$, and by \eqref{eq:severalequalities} we have
\begin{equation}\label{eq:anequality}
\cal G_n \cdot \alpha_1 = 1 +2(n-1) = 2n -1 > 0.
\end{equation}
Since $\alpha_1=\cal{E}_4 - \cal{E}_5$, we have $\Delta^{\nod}_{X,D} = \{ \alpha_1 \}$ by Step 2, which finishes the proof that $\mathcal G_n$ is the class of a $({-}1)$-curve.

Finally, for distinct positive integers $n$ and $n'$ we have by \eqref{eq:anequality} that $\cal G_n \cdot \alpha_1 \neq \cal G_{n'}\cdot \alpha_1$, hence $\cal G_n \neq \cal G_{n'}$, as desired.

\medskip

\emph{Step 6.}
Let $G\subseteq\GL\big(N^1(X)_\R\big)$ be any group which preserves the intersection form on $N^1(X)_\R$ and the cone $\Eff(X)$. Then by \cite[Proposition A.1]{LSX26} we have $G \subseteq \Cr(X)$, thus $G$ is finite by Step 4. Since the cone $\Nef(X)$ is not rational polyhedral by Step 5, and since $\Nef(X)=\Nef^e(X)$ by Lemma \ref{lem:effconeHalpensurface1}(b), we conclude that there does not exist a rational polyhedral fundamental domain for the action of $G$ on $\Nef(X)$. This concludes the proof.
\end{proof}

\newpage

\appendix

\section{Root data of type $A_1^\smone$}\label{sec:appendix_cyclic_subgroup}

In this appendix we collect several results on sets of root data of type $A_1^\smone$ which we extensively use in Part \ref{part:2} of this paper.

\begin{setup}\label{setup1}
Throughout this appendix, we fix a root data of affine type
$$\mathscr D = \big(A,\Pi,\Pi^\smvee,V,V^*,\langle\cdot\, ,\cdot\rangle\big)$$
such that $ \Pi = \{\alpha_0, \alpha_1\} $, and the matrix
\begin{equation}\label{eq:CartanA_1}
A\coloneqq \begin{pmatrix}
2 & -2 \\
-2 & 2
\end{pmatrix}
\end{equation}
is of affine type $A^\smone_1$. Then there is a root system in the classical sense which is compatible with this set of root data. Thus, as in Theorem \ref{thm:imaginary_roots}(b), we have the imaginary root
$$\delta = \alpha_0 + \alpha_1,$$
and as in \S\ref{subsec:normalised} we have the canonical central element
$$c = \alpha_0^\smvee + \alpha_1^\smvee.$$
As in \S\ref{subsec:roots} we denote by $W_{\mathscr D}$ the Weyl group of $\mathscr D$ with generators 
$$ r_0\coloneqq r_{\alpha_0} \quad\text{and}\quad r_1\coloneqq r_{\alpha_1}. $$
As usual, we denote by $\cal F_{\mathscr D}$ the negative fundamental chamber of $\mathscr D$ and by $\cal T_{\mathscr D}$ the negative Tits cone of $\mathscr D$.
\end{setup}

With this setup, it is easy to check that
\begin{equation}\label{eq:orthogonalinA_1}
\langle\delta,\alpha_0^\smvee\rangle = \langle\delta,\alpha_1^\smvee\rangle = \langle\delta,c\rangle = \langle\alpha_0,c\rangle = \langle\alpha_1,c\rangle = 0.
\end{equation}
By Proposition \ref{prop:rootsafterremoving} we have
\begin{equation}\label{eq:positiverootsA_1}
(\Delta_{\mathscr D}^\re)_+=\{\alpha_0 + \ell\delta, \alpha_1 + \ell\delta\mid \ell\in\Z_{\geq0}\}.
\end{equation}

We start with the following lemma which is used throughout the appendix.

\begin{lem}\label{lem:veeinA_1}
Assume Setup \ref{setup1}. Then for any $\ell\in\Z$ we have
$$(\alpha_0+\ell\delta)^\smvee=\alpha_0^\smvee+\ell c\quad\text{and}\quad (\alpha_1+\ell\delta)^\smvee=\alpha_1^\smvee+\ell c.$$
\end{lem}

\begin{proof}
By definition, there is a realisation $(U,\Pi,\Pi^\smvee)$ of the generalised Cartan matrix $A$ in the classical sense which is compatible with this set of root data. We can define a standard bilinear form on $U$ associated to $\mathscr D$. Since all real roots of $\mathscr D$ have the same squared length by Lemma \ref{lem:lenthsofroots}, we have
$$\big((\ell+1)\alpha_0+\ell\alpha_1\big)^\smvee=(\ell+1)\alpha_0^\smvee+\ell\alpha_1^\smvee$$
by the proof of \cite[Proposition 5.10(b)]{Kac90}. Therefore,
$$(\alpha_0+\ell\delta)^\smvee=\big((\ell+1)\alpha_0+\ell\alpha_1\big)^\smvee=(\ell+1)\alpha_0^\smvee+\ell\alpha_1^\smvee=\alpha_0^\smvee+\ell c.$$
This shows the first equality; the second one is proved analogously.
\end{proof}

\subsection{Fundamental reflections}

We need a few lemmas calculating actions of fundamental reflections $r_0$ and $r_1$ as in Setup \ref{setup1}, as well as actions of reflections of the form $r_{k\delta+\alpha_0}$ and $r_{\ell\delta+\alpha_1}$ for non-negative integers $k$ and $\ell$. We will often use the following formulas, which follow by easy calculations: for any $x\in V$ we have
$$ r_1 r_0 (x) = x- \langle x ,\alpha_0^\smvee\rangle \delta - \langle x ,c\rangle \alpha_1,\quad r_0r_1 (x) = x- \langle x ,\alpha_1^\smvee\rangle \delta - \langle x ,c\rangle \alpha_0 , $$
and
$$ r_{\ell\delta + \alpha_0} (x) = x - \langle x , \ell c + \alpha_0^\smvee \rangle ( \ell\delta + \alpha_0 ) ,\quad  r_{\ell \delta + \alpha_1} (x) = x - \langle x , \ell c + \alpha_1^\smvee \rangle ( \ell \delta + \alpha_1 ). $$

\begin{lem}\label{lem:r1r0_n}
	Assume Setup \ref{setup1}. Let $k$ and $\ell$ be integers. Let $m$ be a non-negative integer and let $v\in W_{\mathscr D}$ be the unique word of length $m$ which starts with $r_0$. Then:
	\begin{enumerate}[\normalfont (a)]
	\item $ (r_0r_1)^\ell(\alpha_0) = \alpha_0 +2\ell \delta $, $(r_0r_1)^\ell(\alpha_1) = \alpha_1 - 2\ell \delta $,
	\item $ r_{\ell\delta + \alpha_0}(\alpha_0) = \alpha_1-(2\ell+1)\delta$, $r_{\ell\delta + \alpha_0}(\alpha_1) = \alpha_0+(2\ell+1)\delta $,
	\item $ r_{\ell\delta + \alpha_1}(\alpha_0) = \alpha_1+(2\ell+1)\delta$, $r_{\ell\delta + \alpha_1}(\alpha_1) = \alpha_0-(2\ell+1)\delta $,
    \item $ r_{\ell\delta + \alpha_0} = (r_0 r_1)^\ell r_0 $ , $ r_{\ell\delta+\alpha_1} = r_1 (r_0 r_1)^\ell $,
	\item $r_{k\delta+\alpha_0} r_{\ell\delta + \alpha_1} = (r_0 r_1)^{k+\ell+1}$, $r_{\ell\delta + \alpha_1} r_{k\delta+\alpha_0} = (r_1 r_0)^{k+\ell+1}$,
	\item $ \big\{v(\alpha_0),v(\alpha_1)\big\} = \{\alpha_0 + m\delta, \alpha_1 - m\delta\} $,
	\item $ v(\overline{\cal F_{\mathscr D}}) = \big\{x\in V\mid \langle x , mc+\alpha_0^\smvee \rangle \leq 0,\ \langle x, (m-1) c+\alpha_0^\smvee \rangle \geq 0 \big\} $.
	\end{enumerate}		
\end{lem}

\begin{proof}
First we prove (a).  We only show the first equality as the second one is analogous. The equality is clear for $\ell=0$. Assume that it holds for some non-negative integer $\ell$. Then by the formulas before this lemma and by \eqref{eq:orthogonalinA_1} we have
\begin{align*}
	(r_0r_1)^{\ell+1}(\alpha_0) &= (r_0r_1)(\alpha_0+2\ell \delta)=\alpha_0+2\ell \delta\\
	&- \langle \alpha_0+2\ell \delta ,\alpha_1^\smvee\rangle \delta - \langle \alpha_0+2\ell \delta ,c\rangle \alpha_0=\alpha_0+2(\ell+1) \delta,
\end{align*}
as desired. This deals with non-negative integers. For non-positive integers the proof is similar, by noticing that $(r_0r_1)^{-\ell}=(r_1r_0)^\ell$.

Parts (b) and (c) follow from the formulas before this lemma and by \eqref{eq:orthogonalinA_1}.

	Next we show (d). For $\ell=0$ the statements are clear. Now assume that the first formula in (d) holds for some non-negative integer $\ell$. Let $x\in V$. Then by induction, by the formulas before this lemma, by Lemma \ref{lem:veeinA_1}, and by \eqref{eq:orthogonalinA_1} we have
    \begin{align*}
    (r_0 &r_1)^{\ell+1} r_0 (x) = r_{\ell\delta + \alpha_0} \big( r_1 r_0 (x)\big) = r_{\ell\delta+ \alpha_0} \big( x- \langle x,\alpha_0^\smvee\rangle \delta - \langle x , c\rangle \alpha_1 \big) \\
	&= x- \langle x,\alpha_0^\smvee\rangle \delta - \langle x , c\rangle \alpha_1 - \big\langle x- \langle x,\alpha_0^\smvee\rangle \delta - \langle x , c\rangle \alpha_1 ,\ell c+ \alpha_0^\smvee\big\rangle (\ell\delta+ \alpha_0)\\
	&= x - \big\langle x , (\ell+1) c + \alpha_0^\smvee \big\rangle \big( (\ell+1)\delta + \alpha_0 \big) = r_{(\ell+1)\delta + \alpha_0}(x).
	\end{align*}
    This shows the first formula for non-negative integers, and the proof of the second formula for non-negative integers is analogous. Further, for a non-negative integer $\ell$, by the second formula for non-negative integers we have
    \begin{align*}
	(r_0r_1)^{-\ell}r_0&=(r_1r_0)^\ell r_0=(r_1r_0)^{\ell-1} r_1=r_1(r_0r_1)^{\ell-1}\\
	&=r_{(\ell-1)\delta+\alpha_1}=r_{\ell\delta-\alpha_0}=r_{\alpha_0-\ell\delta}.
    \end{align*}
	This shows the first formula for all integers $\ell$, and the proof of the second formula is analogous.

	Part (e) follows immediately from (d).
	
	Now we show (f). If $m$ is even, then $v=(r_0r_1)^{m/2}$, hence by (a) we have $ v(\alpha_0) = \alpha_0 + m\delta $ and $ v(\alpha_1) = \alpha_1 - m\delta $. If $m$ is odd, then $v=(r_0r_1)^{(m-1)/2}r_0=r_{\frac{m-1}{2}\delta+\alpha_0}$ by (d), hence by (b) we have $ v(\alpha_0) = \alpha_1 -m\delta $ and $ v(\alpha_1) = \alpha_0 + m\delta $, as desired.
	
	Finally, part (f), together with Proposition \ref{prop:W_invariant_pairing} and Lemma \ref{lem:veeinA_1}, gives
\begin{align*}
v(\overline{\cal F_{\mathscr D}}) &= \big\{v(y)\mid y\in V,\ \langle y , \alpha_0^\smvee \rangle \leq 0,\  \langle y, \alpha_1^\smvee  \rangle \leq 0 \big\}\\
&= \big\{x\in V\mid \langle x , v(\alpha_0)^\smvee \rangle \leq 0,\ \langle x, v(\alpha_1)^\smvee \rangle \leq 0 \big\}\\
&= \big\{x\in V\mid \langle x , mc+\alpha_0^\smvee \rangle \leq 0,\ \langle x, -mc+\alpha_1^\smvee \rangle \leq 0 \big\}\\
&= \big\{x\in V\mid \langle x , mc+\alpha_0^\smvee \rangle \leq 0,\ \langle x, (m-1) c+\alpha_0^\smvee \rangle \geq 0 \big\},
\end{align*}
where for the last equality we used that $-mc+\alpha_1^\smvee = {-}\big((m-1) c+\alpha_0^\smvee\big) $. This shows (g) and finishes the proof.
\end{proof}

\subsection{Subroot systems}

From now on, we work additionally in the following setup.

\begin{setup}\label{setup2}
Assume Setup \ref{setup1}. For any non-negative integers $k$ and $\ell$ set
$$ \Pi_{k,\ell} \coloneqq \big\{ k\delta+\alpha_0, \ell\delta + \alpha_1 \big\}\quad\text{and}\quad \Pi_{k,\ell}^\smvee \coloneqq \big\{ kc+\alpha_0^\smvee, \ell c + \alpha_1^\smvee \big\}. $$
By \eqref{eq:positiverootsA_1} we have
\begin{equation}\label{eq:Pikell}
\Pi_{k,\ell} \subseteq (\Delta_{\mathscr D}^\re)_+.
\end{equation}
Then by Lemma \ref{lem:veeinA_1} and with the matrix $A$ as in \eqref{eq:CartanA_1}, it is easy to check that 
$$\mathscr D_{k,\ell}\coloneqq \big(A,\Pi_{k,\ell},\Pi_{k,\ell}^\smvee,V,V^*,\langle\cdot\, ,\cdot\rangle\big)$$
is a subroot system of $\mathscr D$, and it is also an affine root system of type $A_1^\smone$. As usual, we denote by $W_{\mathscr D_{k,\ell}}$, $\cal{F}_{\mathscr D_{k,\ell}}$ and $\cal{T}_{\mathscr D_{k,\ell}}$ the Weyl group, the negative fundamental chamber and the negative Tits cone of $\mathscr D_{k,\ell}$. Define the \emph{faces} of $\overline{\cal{F}_{\mathscr D_{k,\ell}}}$ by
    \begin{align*}
	F_{k,\ell}^0 &\coloneqq \big\{ x\in V \mid \langle x, kc +\alpha_0^\smvee\rangle = 0,\ \langle x, \ell c+\alpha_1^\smvee \rangle <0 \big\},\\
	F_{k,\ell}^1 &\coloneqq \big\{ x\in V \mid \langle x, kc+\alpha_0^\smvee\rangle < 0,\ \langle x, \ell c+\alpha_1^\smvee \rangle  = 0 \big\}.    
    \end{align*}
    When $k=\ell=0$, we denote by $F_0\coloneqq F_{0,0}^0$ and $F_1\coloneqq F_{0,0}^1$ the faces of $\overline{\cal{F}_{\mathscr D}}$.
\end{setup}

\begin{lem}\label{lem:Wklcalculation}
Assume Setup \ref{setup2}. Let $k$ and $\ell$ be non-negative integers. Then the following statements hold.
\begin{enumerate}[\normalfont (a)]
\item If $k\geq1$, then $ r_0(\overline{\cal F_{\mathscr D_{k,\ell}}}) = \overline{\cal{F}_{\mathscr D_{\ell+1,k-1}}} $, and if $\ell\geq1$, then $ r_1(\overline{\cal F_{\mathscr D_{k,\ell}}}) =\overline{\cal{F}_{\mathscr D_{\ell-1,k+1}}} $.
\item We have
$$ \overline{\cal F_{{\mathscr D}_{\ell+1, \ell}}} = \overline{\cal F_{\mathscr D_{0,\ell}}}\cup r_{0}(\overline{\cal F_{\mathscr D_{0,\ell}}}) \quad\text{and}\quad (r_1r_0)^{\ell+1}(\overline{\cal F_{{\mathscr D}_{\ell+1, \ell}}} )=r_{\ell\delta + \alpha_1} (\overline{\cal F_{{\mathscr D}_{\ell+1, \ell}}} ) . $$
\item If $w\in W_{\mathscr D_{k,\ell}}$, then there exists $n\in\Z$ such that
$$ w\in\big\{(r_0r_1)^{n(k+\ell+1)},r_{\ell\delta + \alpha_1} (r_0r_1)^{n(k+\ell+1)}\big\}. $$
    \item Let $w\in W_{\mathscr D_{k,\ell}}$. If $k=0$, then there exists $u \in W_{\mathscr D_{0,\ell}}$ such that 
    $$ r_0 w( \overline{\cal{F}_{\mathscr D_{0,\ell}}})=u( \overline{\cal{F}_{\mathscr D_{0,\ell}}}). $$
    If $k\geq1$, then there exists $u\in W_{\mathscr D_{\ell+1,k-1}}$ such that
$$ r_0w( \overline{\cal{F}_{\mathscr D_{k,\ell}}})=u ( \overline{\cal{F}_{\mathscr D_{\ell+1,k-1}}}). $$
\item Let $w\in W_{\mathscr D_{k,\ell}}$. If $\ell = 0$, then there exists $u\in W_{\mathscr D_{k,0}}$ such that
$$ r_1w ( \overline{\cal{F}_{\mathscr D_{k,0}}})=u( \overline{\cal{F}_{\mathscr D_{k,0}}}). $$
If $\ell\geq1$, then there exists $u\in W_{\mathscr D_{\ell-1,k+1}}$ such that
$$ r_1w( \overline{\cal{F}_{\mathscr D_{k,\ell}}})=u( \overline{\cal{F}_{\mathscr D_{\ell-1,k+1}}}). $$
\end{enumerate}
\end{lem}

\begin{proof}
We start by noting that $W_{\mathscr D_{k,\ell}}=\langle r_{k\delta+\alpha_0},r_{\ell\delta + \alpha_1}\rangle$.

\medskip

\emph{Step 1.}
In this step we show (a).

Since $ r_0(kc + \alpha_0^\smvee) = (k-1)c + \alpha_1^\smvee $ and $ r_0 ( \ell c + \alpha_1^\smvee ) = (\ell + 1)c + \alpha_0^\smvee $, by Lemma \ref{lem:veeinA_1} we have
\begin{align*}
&r_0(\overline{\cal F_{\mathscr D_{k,\ell}}}) = \big\{r_0(y)\mid y\in V,\ \langle y , kc + \alpha_0^\smvee \rangle \leq 0,\  \langle y, \ell c + \alpha_1^\smvee \rangle \leq 0 \big\}\\
&= \big\{x\in V\mid \langle x , r_0(kc + \alpha_0^\smvee) \rangle \leq 0,\ \langle x, r_0 ( \ell c + \alpha_1^\smvee ) \rangle \leq 0 \big\}=\overline{\cal{F}_{\mathscr D_{\ell+1,k-1}}},
\end{align*}
which gives the first equality in (a). The second equality follows similarly.

\medskip

\emph{Step 2.}
In this step we show (b).

As in Step 1, we have
\begin{align*}
r_0(\overline{\cal F_{\mathscr D_{0,\ell}}}) &=\big\{ x\in V \mid \langle x,{-}c+\alpha_1^\smvee\rangle \leq 0,\ \langle x, (\ell+1)c+\alpha_0^\smvee \rangle \leq 0 \big\} \\
&=\big\{ x\in V \mid \langle x,\alpha_0^\smvee\rangle \geq 0,\ \langle x, (\ell+1)c+\alpha_0^\smvee \rangle \leq 0 \big\}.
\end{align*}
Now the first equality in (b) follows from Lemma \ref{lem:simplelemma}(b) since $\textstyle (\ell+1)c+\alpha_0^\smvee= \alpha_0^\smvee +(\ell c+\alpha_1^\smvee)$.
    
    For the second equality, denote $g\coloneqq (r_1r_0)^{\ell+1}$. By Lemma \ref{lem:r1r0_n}(d) we have $r_{\ell\delta + \alpha_1}=gr_0$ and $g=r_{\ell\delta + \alpha_1} r_0$. Thus, by applying the first equality in (b) twice, we obtain 
\begin{align*}
	g(\overline{\cal F_{{\mathscr D}_{\ell+1, \ell}}} )&=g\big(\overline{\cal F_{\mathscr D_{0,\ell}}}\cup r_{0}(\overline{\cal F_{\mathscr D_{0,\ell}}})\big)=g(\overline{\cal F_{\mathscr D_{0,\ell}}})\cup gr_0(\overline{\cal F_{\mathscr D_{0,\ell}}})\\
	&=r_{\ell\delta + \alpha_1} r_0(\overline{\cal F_{\mathscr D_{0,\ell}}})\cup r_{\ell\delta + \alpha_1}(\overline{\cal F_{\mathscr D_{0,m-1}}})\\
	&=r_{\ell\delta + \alpha_1}\big(r_{0}(\overline{\cal F_{\mathscr D_{0,\ell}}})\cup \overline{\cal F_{\mathscr D_{0,\ell}}}\big)=r_{\ell\delta + \alpha_1}(\overline{\cal F_{{\mathscr D}_{\ell+1, \ell}}} ),
\end{align*}
as desired.

\medskip

\emph{Step 3.}
In this step we show (c).

As $w$ is a word in $r_{k\delta+\alpha_0}$ and $r_{\ell\delta + \alpha_1}$, there exists $m\in\Z_{\geq0}$ such that
\begin{align*}
w\in\big\{&(r_{k\delta+\alpha_0} r_{\ell\delta + \alpha_1})^m,(r_{k\delta+\alpha_0} r_{\ell\delta + \alpha_1})^m r_{k\delta+\alpha_0},\\
&(r_{\ell\delta + \alpha_1} r_{k\delta+\alpha_0})^m,r_{\ell\delta + \alpha_1} (r_{k\delta+\alpha_0} r_{\ell\delta + \alpha_1})^m\big\} .
\end{align*}
Since $(r_{\ell\delta + \alpha_1} r_{k\delta+\alpha_0})^m=(r_{k\delta+\alpha_0} r_{\ell\delta + \alpha_1})^{-m}$ and $(r_{k\delta+\alpha_0} r_{\ell\delta + \alpha_1})^m r_{k\delta+\alpha_0}=r_{\ell\delta + \alpha_1} (r_{k\delta+\alpha_0} r_{\ell\delta + \alpha_1})^{-(m+1)}$, it follows that there exists $n\in\Z$ such that $ w\in\big\{(r_{k\delta+\alpha_0} r_{\ell\delta + \alpha_1})^n,r_{\ell\delta + \alpha_1} (r_{k\delta+\alpha_0} r_{\ell\delta + \alpha_1})^n\big\}$. We conclude by Lemma \ref{lem:r1r0_n}(e).
 
\medskip

\emph{Step 4.}
    In this step we show (d).
    
    First note that $r_0\in W_{\mathscr D_{0,\ell}}$, hence the first equation in (d) follows by setting $u \coloneqq r_0w$. Now we prove the second equation. By (c) there exists $n\in\Z$ such that $ w\in \big\{ (r_0r_1)^{n(k+\ell+1)},  r_{\ell\delta + \alpha_1} (r_0r_1)^{n(k+\ell+1)} \big\} $.
    
    Assume first that $ w = (r_0r_1)^{n(k+\ell+1)}$. It is easy to check that $r_0 w = w^{-1} r_0$, hence by (a) we have
    $$ r_0 w ( \overline{\cal{F}_{\mathscr D_{k,\ell}}}) = w^{-1} r_0 ( \overline{\cal{F}_{\mathscr D_{k,\ell}}}) = w^{-1} (\overline{\cal{F}_{\mathscr D_{\ell+1,k-1}}}). $$
    We obtain (d) by setting $u\coloneqq w^{-1}$.

    Now assume that $w=r_{\ell\delta + \alpha_1} (r_0r_1)^{n(k+\ell+1)}$. Then by Lemma \ref{lem:r1r0_n}(d) it is easy to check that $r_0 w = (r_0 r_1)^{n(k+\ell+1)} r_{(\ell+1)\delta+\alpha_0} r_0 $. Therefore, (a) gives
    \begin{align*}
        r_0 w ( \overline{\cal{F}_{\mathscr{D}_{k,\ell}}} ) &= (r_0 r_1)^{n(k+\ell+1)} r_{(\ell+1)\delta+\alpha_0} r_0 ( \overline{\cal{F}_{\mathscr{D}_{k,\ell}}} )\\
        & = (r_0 r_1)^{n(k+\ell+1)} r_{ (\ell+1)\delta + \alpha_0} ( \overline{\cal{F}_{\mathscr{D}_{\ell+1,k-1}}}).
    \end{align*}
    Clearly $r_{ (\ell+1)\delta + \alpha_0} \in  W_{\mathscr{D}_{\ell+1,k-1}}$, and Lemma \ref{lem:r1r0_n}(e) yields $(r_0 r_1)^{n(k+\ell+1)} \in W_{\mathscr{D}_{\ell+1,k-1}}$. Part (d) now follows by setting $u\coloneqq (r_0 r_1)^{n(k+\ell+1)} r_{ (\ell+1)\delta + \alpha_0}$.

\medskip

\emph{Step 5.}
    Finally, in this step we show (e).
    
    First note that $r_1\in W_{\mathscr D_{k,0}}$, hence the first equation in (e) follows by setting $u \coloneqq r_1w$. Now we prove the second equation. By (c) there exists $n\in\Z$ such that $ w\in \big\{ (r_0r_1)^{n(k+\ell+1)},  r_{\ell\delta + \alpha_1} (r_0r_1)^{n(k+\ell+1)} \big\} $.
    
    Assume first that $ w = (r_0r_1)^{n(k+\ell+1)}$. It is easy to check that $r_1 w = w^{-1} r_1$, hence by (a) we have
    $$ r_1 w ( \overline{\cal{F}_{\mathscr D_{k,\ell}}}) = w^{-1} r_1 ( \overline{\cal{F}_{\mathscr D_{k,\ell}}}) = w^{-1} (\overline{\cal{F}_{\mathscr D_{\ell-1,k+1}}}). $$
    We obtain (e) by setting $u\coloneqq w^{-1}$.

    Now assume that $w=r_{\ell\delta + \alpha_1} (r_0r_1)^{n(k+\ell+1)}$. Then by Lemma \ref{lem:r1r0_n}(d) it is easy to check that $r_1 w = (r_0 r_1)^{n(k+\ell+1)} r_{(\ell-1)\delta+\alpha_0} r_1 $. Therefore, (a) gives
    \begin{align*}
        r_1 w ( \overline{\cal{F}_{\mathscr{D}_{k,\ell}}} ) &= (r_0 r_1)^{n(k+\ell+1)} r_{(\ell-1)\delta+\alpha_0} r_1 ( \overline{\cal{F}_{\mathscr{D}_{k,\ell}}} ) \\
        &= (r_0 r_1)^{n(k+\ell+1)} r_{(\ell-1)\delta+\alpha_0} ( \overline{\cal{F}_{\mathscr{D}_{\ell-1,k+1}}} ).
    \end{align*}
    Clearly $r_{(\ell-1)\delta+\alpha_0} \in W_{\mathscr{D}_{\ell-1,k+1}}$, and Lemma \ref{lem:r1r0_n}(e) yields $(r_0 r_1)^{n(k+\ell+1)} \in W_{\mathscr{D}_{\ell-1,k+1}}$. Part (e) now follows by setting $u\coloneqq (r_0 r_1)^{n(k+\ell+1)} r_{ (\ell-1)\delta + \alpha_0}$.
\end{proof}

\subsection{Negative fundamental chambers}
We need throughout the paper the following precise description of the negative fundamental domains $\overline{\cal{F}_{\mathscr D_{k,\ell}}}$ and negative Tits cones $\cal{T}_{\mathscr D_{k,\ell}}$.

\begin{prop}\label{prop:WDkl}
Assume Setup \ref{setup2}. Let $k$ and $\ell$ be non-negative integers. Then the following statements hold.
\begin{enumerate}[\normalfont (a)]
\item We have $W_{\mathscr D_{k,\ell}} \subseteq W_{\mathscr{D}} $.
\item For each $(i,j)\in\{0,1\}\times\Z_{\geq0}$ denote by $w_{i,j}$ the unique reduced word in $W_{\mathscr D}$ which has length $j$ and which starts with $r_i$. Then
$$ \textstyle \overline{\cal F_{\mathscr D_{k,\ell}}} = \bigcup_{j=0}^k w_{0,j}(\overline{\cal F_{\mathscr D}}) \cup  \bigcup_{j=1}^{\ell} w_{1,j}(\overline{\cal F_{\mathscr D}}) . $$
\item We have $\mathcal{T}_{\mathscr D_{k,\ell}} =  \mathcal{T}_{\mathscr D}$.
\item For each non-negative integer $m$ we have
$$ \textstyle \overline{\cal F_{\mathscr D}} = \bigcap_{p = 0}^m \overline{\cal F_{\mathscr D_{p,m-p}}}. $$
\item Let $w\in W_{\mathscr D}$ and let $m$ be a non-negative integer. Then for each $p\in\{0,\dots,m\}$ there exists $u_p\in W_{\mathscr D_{p,m-p}}$ such that 
    $$ \textstyle w ( \overline{\mathcal F_{\mathscr{D}}} ) = \bigcap_{p = 0}^m u_p (\overline{\cal{F}_{\mathscr{D}_{p,m-p}}}). $$
\end{enumerate}
\end{prop}

\begin{proof}
Part (a) follows from Lemma \ref{lem:r1r0_n}(d). For (d), note first that $ \overline{\cal F_{\mathscr D}} \subseteq \bigcap_{p = 0}^m \overline{\cal F_{\mathscr D_{p,m-p}}} $ by Lemma \ref{lem:equivalence_chamber_inclusion}. Conversely, let $x \in  \bigcap_{p = 0}^m \overline{\cal F_{\mathscr D_{p,m-p}}}$. Since $x \in \overline{\cal F_{\mathscr D_{0,m}}}$, we have $\langle x, \alpha_0^\smvee \rangle \leq 0$, and since $x \in \overline{\cal F_{\mathscr{D}_{m,0}}}$, we have $\langle x, \alpha_1^\smvee \rangle \leq 0$. Therefore, $x \in \overline{\cal F_{\mathscr D}}$ by definition.

Next we show that (b) implies (c). We have $\mathcal{T}_{\mathscr D}\subseteq \mathcal{T}_{\mathscr D_{k,\ell}}$ by \eqref{eq:Pikell} and by Corollary \ref{cor:Titssubcone}, hence it suffices to prove that $\mathcal{T}_{\mathscr D_{k,\ell}} \subseteq \mathcal{T}_{\mathscr D}$. To that end, for each $w \in W_{\mathscr{D}_{k,\ell}}$ we have $w\in W_{\mathscr D}$ by (a), hence (b) gives
$$ \textstyle w(\overline{\cal F_{\mathscr D_{k,\ell}}})  = \bigcup_{j=0}^k ww_{0,j}(\overline{\cal F_{\mathscr D}}) \cup  \bigcup_{j=1}^{\ell} ww_{1,j}(\overline{\cal F_{\mathscr D}}) \subseteq \cal{T}_{\mathscr D}. $$
Therefore, $ \cal{T}_{\mathscr D_{k,\ell}} = \bigcup_{w \in W_{\mathscr{D}_{k,\ell}}} w(\overline{\cal F_{\mathscr D_{k,\ell}}}) \subseteq \cal{T}_{\mathscr D}$, as desired.

It remains to prove (b) and (e). We do this in the following four steps.

\medskip

\emph{Step 1.}
In this step we show that 
$$ \textstyle\bigcup_{j=0}^k w_{0,j}(\overline{\cal F_{\mathscr D}}) = \big\{ x\in V \mid \langle x , \alpha_1^\smvee \rangle \leq 0, \ \langle x, k c+\alpha_0^\smvee  \rangle \leq 0 \big\}. $$

The proof is by induction on $k$. For $k = 0$ the claim is clear. Now assume that the claim holds for some integer $k\geq 0$. Then by the induction hypothesis and by Lemma \ref{lem:r1r0_n}(f) we have
\begin{align*}
\textstyle\bigcup_{j=0}^{k+1} w_{0,j}(\overline{\cal F_{\mathscr D}}) & = \textstyle \bigcup_{j=0}^{k} w_{0,j}(\overline{\cal F_{\mathscr D}}) \cup w_{0,k+1}(\overline{\cal F_{\mathscr D}}) \\
& = \big\{ x\in V \mid \langle x , \alpha_1^\smvee \rangle \leq 0, \ \langle x, k c+\alpha_0^\smvee  \rangle \leq 0 \big\} \\
&\quad\ \cup \big\{x\in V\mid \langle x , (k+1)c+\alpha_0^\smvee \rangle \leq 0,\ \langle x, k c+\alpha_0^\smvee \rangle \geq 0 \big\}\\
&=\big\{ x\in V \mid \langle x , \alpha_1^\smvee \rangle \leq 0, \  \big\langle x, (k+1)c +\alpha_0^\smvee  \big\rangle \leq 0 \big\},
\end{align*}
where the last equality follows from Lemma \ref{lem:simplelemma}(b) since $ (k +1)c+\alpha_0^\smvee= \frac{1}{k+1} \alpha_1^\smvee+\frac{k+2}{k+1} (k c+\alpha_0^\smvee) $. This proves the claim.

\medskip

\emph{Step 2.}
In this step we show that for $\ell\geq1$ we have
$$ \textstyle\bigcup_{j=1}^{\ell} w_{1,j}(\overline{\cal F_{\mathscr D}}) = \big\{ x\in V \mid \langle x , \alpha_1^\smvee \rangle \geq 0, \ \langle x, \ell c+\alpha_1^\smvee \rangle \leq 0 \big\}. $$

To that end, as clearly $w_{1,j}= r_1 w_{0,j-1}$ for $j \geq 1$, we deduce from Step 1 and from Proposition \ref{prop:W_invariant_pairing} that 
\begin{align*}
\textstyle \bigcup_{j=1}^{\ell} w_{1,j}(\overline{\cal F_{\mathscr D}}) & = \textstyle\bigcup_{j=0}^{\ell -1} r_1w_{0,j}(\overline{\cal F_{\mathscr D}}) = r_1\big(\bigcup_{j=0}^{\ell -1 
} w_{0,j}(\overline{\cal F_{\mathscr D}})\big) \\
& = r_1\big( \big\{ x\in V \mid \langle x , \alpha_1^\smvee \rangle \leq 0, \ \langle x, (\ell -1) c+\alpha_0^\smvee \rangle \leq 0 \big\}\big) \\
& = \big\{ x\in V \mid \langle x , r_1(\alpha_1^\smvee) \rangle \leq 0, \ \langle x, r_1((\ell -1) c+\alpha_0^\smvee) \rangle \leq 0 \big\}\\
& = \big\{ x\in V \mid \langle x , \alpha_1^\smvee\rangle \geq 0, \ \langle x, \ell c+\alpha_1^\smvee \rangle \leq 0 \big\}, 
\end{align*}
where in the last equality we used that $r_1(\alpha_1^\smvee)={-}\alpha_1^\smvee$ and $r_1\big((\ell -1) c+\alpha_0^\smvee\big)=\ell c+\alpha_1^\smvee$. This proves the claim.

\medskip

\emph{Step 3.}
In this step we finish the proof of (b).

If $\ell=0$, then by Step 1 we have
$$\textstyle \bigcup_{j=0}^{k} w_{0,j}(\overline{\cal F_{\mathscr D}}) = \big\{ x\in V \mid \langle x , \alpha_1^\smvee \rangle \leq 0, \ \langle x, k c+\alpha_0^\smvee  \rangle \leq 0 \big\} = \overline{\cal F_{\mathscr D_{k,\ell}}},$$
where the last equality follows by the definition of the negative fundamental chamber. Now assume that $\ell\geq1$. By Steps 1 and 2, the set $\bigcup_{j=0}^{k} w_{0,j}(\overline{\cal F_{\mathscr D}}) \cup \bigcup_{j=1}^{\ell} w_{1,j}(\overline{\cal F_{\mathscr D}}) $ is equal to the set
\begin{align*}
\big\{ x\in V \mid \langle &x , \alpha_1^\smvee \rangle \leq 0, \ \langle x, k c+\alpha_0^\smvee  \rangle \leq 0 \big\}\\
&\,\,\cup \big\{ x\in V \mid \langle x , \alpha_1^\smvee \rangle \geq 0, \ \langle x, \ell c+\alpha_1^\smvee \rangle \leq 0 \big\}.
\end{align*}
Since $\ell c+\alpha_1^\smvee= \frac{\ell + k+1}{k+1} \alpha_1^\smvee+\frac{\ell}{k+1} (k c+\alpha_0^\smvee) $, by Lemma \ref{lem:simplelemma}(b) the set above is equal to the set $ \big\{ x\in V \mid \langle x, k c+\alpha_0^\smvee  \rangle \leq 0 , \ \big\langle x, \ell c+\alpha_1^\smvee \big\rangle \leq 0 \big\} $, which is $\overline{\cal F_{\mathscr D_{k,\ell}}}$ by the definition of the negative fundamental chamber.

\medskip

\emph{Step 4.}
In this step we prove (e).

    Let $j$ be the length of $w$. The proof is by induction on $j$. If $j = 0$, then we conclude by (d). 
    
     Now assume that the result holds for all elements of $W_{\mathscr D}$ of length $j-1$. There exist an index $i\in \{0,1\}$ and an element $w'\in W_{\mathscr D}$ of length $j-1$ such that $ w = r_i w' $. By induction, for each $p\in\{0,\dots,m\}$ there exists $u_p'\in W_{\mathscr D_{p,m-p}}$ such that $ w' ( \overline{\mathcal F_{\mathscr D}} ) = \bigcap_{p = 0}^m u'_p (\overline{\cal{F}_{\mathscr D_{p,m-p}}}) $, hence
    \begin{equation}\label{eq:009h}
    \textstyle w( \overline{\mathcal F_{\mathscr D}} ) =  \bigcap_{p = 0}^m r_i u'_p (\overline{\cal{F}_{\mathscr D_{p,m-p}}}).    
    \end{equation}
    
    Assume first that $i=0$. By Lemma \ref{lem:Wklcalculation}(d) there exists $u_0\in W_{\mathscr D_{0,m}}$ such that $ r_0 u'_0\big( \overline{\cal{F}_{\mathscr D_{0,m}}}\big)=u_0\big( \overline{\cal{F}_{\mathscr D_{0,m}}}\big) $, and for each $p\in\{1,\dots,m\}$ there exists $u_{m-p+1}\in W_{\mathscr D_{m-p+1,p-1}}$ such that $ r_0 u'_p\big( \overline{\cal{F}_{\mathscr D_{p,m-p}}}\big)=u_{m-p+1}\big( \overline{\cal{F}_{\mathscr D_{m-p+1,p-1}}}\big)$. Therefore, \eqref{eq:009h} gives $ w ( \overline{\mathcal F_{\mathscr{D}}} ) = \bigcap_{p = 0}^m u_p (\overline{\cal{F}_{\mathscr{D}_{p,m-p}}}) $, as desired.
	
    Assume now that $i=1$. By Lemma \ref{lem:Wklcalculation}(e) there exists $u_m\in W_{\mathscr D_{m,0}}$ such that $ r_1 u'_m\big( \overline{\cal{F}_{\mathscr D_{m,0}}}\big)=u_m\big( \overline{\cal{F}_{\mathscr D_{m,0}}}\big) $, and for each $p\in\{0,\dots,m-1\}$ there exists $u_{m-p-1}\in W_{\mathscr D_{m-p-1,p+1}}$ such that $ r_1u'_p\big( \overline{\cal{F}_{\mathscr D_{p,m-p}}}\big)=u_{m-p-1}\big( \overline{\cal{F}_{\mathscr D_{m-p-1,p+1}}}\big)$. We conclude as in the previous paragraph.
\end{proof}

\subsection{A fundamental domain}

The following result is fundamental for the proof of Theorem \ref{thm:main2}.

\begin{prop}\label{prop:fund_domain_G}
	Assume Setup \ref{setup2}. Let $m$ be a positive integer. Then:
\begin{enumerate}[\normalfont (a)]
\item for any $k\in \{0,\dots,m-1\}$ we have $ (r_0r_1)^m\subseteq W_{\mathscr D_{k,m-k-1}}$,
\item $(r_0r_1)^{2m}\in W_{\mathscr D_{m,m-1}}$,
\item $ \overline{\cal F_{{\mathscr D}_{m, m-1}}} $ is a fundamental domain for the action of $\big\langle (r_0r_1)^m \big\rangle $ on $\cal{T}_{\mathscr D}$. 
\end{enumerate}
\end{prop}

\begin{proof}
By Lemma \ref{lem:r1r0_n}(e) we have
$$ (r_0r_1)^m = r_{k\delta + \alpha_0}r_{(m-k-1)\delta + \alpha_1} \in W_{\mathscr D_{k,m-k-1}} $$
and
$$ (r_0r_1)^{2m} = r_{m\delta + \alpha_0}r_{(m-1)\delta + \alpha_1} \in W_{\mathscr D_{m,m-1}}, $$
which gives (a) and (b). In the remainder of the proof we show (c). Set $ g\coloneqq (r_0r_1)^m $.

	\medskip

    \emph{Step 1.}
    In this step we show that for any $x\in \cal T_{\mathscr D}$ there exists $s \in \Z$ such that 
    $$ x \in g^s(\overline{\cal F_{{\mathscr D}_{m, m-1}}}). $$

    To that end, note first that by Lemma \ref{lem:Wklcalculation}(b) we have
    \begin{equation}\label{eq:union}
    \overline{\cal F_{{\mathscr D}_{m, m-1}}} = \overline{\cal F_{\mathscr D_{0,m-1}}}\cup r_{0}(\overline{\cal F_{\mathscr D_{0,m-1}}})
    \end{equation}
    Let $x \in \mathcal{T}_{\mathscr D}$. Since $\cal T_{\mathscr D} = \cal{T}_{\mathscr D_{0,m-1}} $ by Proposition \ref{prop:WDkl}(c), by Proposition \ref{prop:Titsfundchamber} there exist $w \in W_{\mathscr D_{0,m-1}}$ and $y \in \overline{\mathcal F_{\mathscr D_{0,m-1}}}$ such that $ x=w (y) $. By Lemma \ref{lem:Wklcalculation}(c) there exists $n\in \Z$ such that $ w\in\big\{g^n, r_{(m-1)\delta+\alpha_1} g^n \big\} $. 
    
    If $w = g^n$, then by \eqref{eq:union} we have $ x=g^n (y)\in g^n(\overline{\mathcal F_{\mathscr D_{0,m-1}}})\subseteq g^n(\overline{\cal F_{{\mathscr D}_{m, m-1}}})$, as desired. Otherwise, we have $w = r_{(m-1)\delta+\alpha_1} g^n $. By Lemma \ref{lem:r1r0_n}(d) we easily calculate that $ g^{n+1}=r_0w^{-1}$, hence by \eqref{eq:union} we have
    $$ g^{n+1}(x) = r_{0}w^{-1} (x) = r_{0} (y) \in r_{0}\big(\overline{\cal{F}_{\mathscr D_{0,m-1}}}\big) \subseteq \overline{\cal F_{{\mathscr D}_{m, m-1}}}, $$
    which finishes the proof of the claim.
    
    \medskip

    \emph{Step 2.}
	Assume that there exists an integer $\ell$ such that
    $$g^\ell \big( \operatorname{Int}(\overline{\cal F_{{\mathscr D}_{m, m-1}}})\big) \cap \operatorname{Int}(\overline{\cal F_{{\mathscr D}_{m, m-1}}})\neq\emptyset.$$
    We will show in this step that then necessarily $g^\ell$ is the identity. This, together with Step 1, will finish the proof of the lemma.

	Assume first that $\ell$ is even. Then by (b) we have $ g^\ell \in W_{\mathscr D_{m,m-1}}$. As $\overline{\cal F_{{\mathscr D}_{m, m-1}}}$ is a fundamental domain for the action of $W_{\mathscr D_{m,m-1}}$ on $\cal T_{\mathscr D_{m,m-1}}$ by Proposition \ref{prop:Titsfundchamber}, this implies that $g^\ell$ is the identity.
    
    Now assume that $\ell$ is odd. Then there exists an integer $\ell'$ such that $\ell=2\ell'-1$. Since $g^{-1}=(r_1r_0)^m$, by Lemma \ref{lem:Wklcalculation}(b) we have 
    $$g^{-1} (\overline{\cal F_{{\mathscr D}_{m, m-1}}})=(r_1r_0)^m (\overline{\cal F_{{\mathscr D}_{m, m-1}}})=r_{(m-1)\delta+\alpha_1}(\overline{\cal F_{{\mathscr D}_{m, m-1}}}),$$
    and therefore,
    \begin{align*}
    g^\ell \big( \operatorname{Int}(\overline{\cal F_{{\mathscr D}_{m, m-1}}})\big) &= g^{2\ell'}g^{-1} \big( \operatorname{Int}(\overline{\cal F_{{\mathscr D}_{m, m-1}}})\big)\\
    &=g^{2\ell'}r_{(m-1)\delta+\alpha_1}\big( \operatorname{Int}(\overline{\cal F_{{\mathscr D}_{m, m-1}}})\big).    
    \end{align*}
    Therefore, the assumption of this step becomes
    $$g^{2\ell'}r_{(m-1)\delta+\alpha_1} \big( \operatorname{Int}(\overline{\cal F_{{\mathscr D}_{m, m-1}}})\big) \cap \operatorname{Int}(\overline{\cal F_{{\mathscr D}_{m, m-1}}})\neq\emptyset.$$
     Since $g^{2\ell'}\in W_{\mathscr D_{m,m-1}}$ by (b), and since clearly $r_{(m-1)\delta+\alpha_1}\in W_{\mathscr D_{m,m-1}}$, we conclude that $g^{2\ell'}r_{(m-1)\delta+\alpha_1}\in W_{\mathscr D_{m,m-1}}$. Then Proposition \ref{prop:Titsfundchamber} implies that $g^{2\ell'}r_{(m-1)\delta+\alpha_1}$ is the identity, hence $g^{2\ell'}=r_{(m-1)\delta+\alpha_1}$ as $(r_{(m-1)\delta+\alpha_1})^{-1}=r_{(m-1)\delta+\alpha_1}$. This is a contradiction since $g$ has infinite order by Lemma \ref{lem:r1r0_n}(a), whereas $r_{(m-1)\delta+\alpha_1}$ has order $2$.
\end{proof}

\subsection{Auxiliary results}

We use the following result very often.

\begin{lem}\label{lem:equalities_of_faces}
    Assume Setup \ref{setup2}. Let $k$ and $\ell$ be non-negative integers. Let $w_{0,k}\in W_{\mathscr D}$ be the unique word which has length $k$ and which starts with $r_0$, and let $w_{1,\ell}\in W_{\mathscr D}$ be the unique word which has length $\ell$ and which starts with $r_1$. Then:
    \begin{enumerate}[\normalfont(a)]
    \item $F_{k,\ell}^0 = \{ x\in V \mid \langle x, kc + \alpha_0^\smvee \rangle = 0, \langle x, c  \rangle < 0 \}$, and in particular, the set $F_{k,\ell}^0$ does not depend on $\ell$,
    \item $F_{k,\ell}^1 = \{ x\in V \mid \langle x, \ell c + \alpha_1^\smvee \rangle = 0, \langle x, c  \rangle < 0 \}$, and in particular, the set $F_{k,\ell}^1$ does not depend on $k$,
     \item if $k>0$, then $\big\{ w_{0,k} (F_0), w_{0,k} (F_1)\big\} = \{ F_{k-1,\ell}^0, F_{k,\ell}^0 \}$,
     \item if $\ell>0$, then $\big\{ w_{1,\ell} (F_0), w_{1,\ell} (F_1)\big\} = \{ F_{k,\ell}^1, F_{k,\ell-1}^1 \}$.
    \end{enumerate}
\end{lem}

\begin{proof}
Consider $x\in V$. If $\langle x, kc + \alpha_0^\smvee \rangle = 0$ and $\langle x, \ell c + \alpha_1^\smvee \rangle < 0$, then by adding these two together, we obtain $\big\langle x, (k+\ell+1)c \big\rangle < 0$, or equivalently, $\langle x,c  \rangle < 0$. If $\langle x, kc + \alpha_0^\smvee \rangle = 0$ and $\langle x, c \rangle < 0$, then by running the previous argument backwards we obtain $\langle x, \ell c + \alpha_1^\smvee \rangle < 0$. This shows that (a), and we obtain (b) analogously.
    
    Now assume that $k>0$. We will show that $ w_{0,k} (F_0) \in \{ F_{k-1,\ell}^0, F_{k,\ell}^0 \}$, which will prove a half of (c). The rest of (c) and (d) is proved analogously.
    
    To that end, Lemma \ref{lem:r1r0_n}(f) gives $ \big\{w_{0,k}(\alpha_0),w_{0,k}(\alpha_1)\big\} = \{\alpha_0 + k\delta, \alpha_1 - k\delta\}$, and in particular, $w_{0,k}(c)=c$. Assume, first, that $w_{0,k}(\alpha_0)=\alpha_0 + k\delta$. This, together with (a), with Proposition \ref{prop:W_invariant_pairing} and with Lemma \ref{lem:veeinA_1}, gives
\begin{align*}
w_{0,k}(F_0) &= \big\{w_{0,k}(y)\mid y\in V,\ \langle y , \alpha_0^\smvee \rangle = 0,\  \langle y, c \rangle < 0 \big\}\\
&= \big\{x\in V\mid \langle x , w_{0,k}(\alpha_0)^\smvee \rangle = 0,\ \langle x, w_{0,k}(c) \rangle < 0 \big\}\\
&= \big\{x\in V\mid \langle x , \alpha_0^\smvee + kc \rangle = 0,\ \langle x, c \rangle < 0 \big\} = F_{k,\ell}^0.
\end{align*}
Now assume that $w_{0,k}(\alpha_0)=\alpha_1 - k\delta$. As above, together with (a), with Proposition \ref{prop:W_invariant_pairing} and with Lemma \ref{lem:veeinA_1}, this gives
\begin{align*}
w_{0,k}(F_0) &= \big\{x\in V\mid \langle x , w_{0,k}(\alpha_0)^\smvee \rangle = 0,\ \langle x, w_{0,k}(c) \rangle < 0 \big\}\\
&= \big\{x\in V\mid \langle x , \alpha_1^\smvee-kc \rangle = 0,\ \langle x, c \rangle < 0 \big\} \\
&= \big\{x\in V\mid \langle x , \alpha_0^\smvee+(k-1) c \rangle = 0,\ \langle x, c \rangle < 0 \big\} = F_{k-1,\ell}^0,
\end{align*}
where for the last equality we used that $\alpha_1^\smvee-kc = {-}\big((k-1) c+\alpha_0^\smvee\big) $. That concludes the proof.
\end{proof}

The next lemma is needed in Sections \ref{sec:orbits} and \ref{sec:proofA}.

\begin{lem}\label{lem:wandw'invariant}
Let $X$ be a Sakai surface of type $A_6^\smone$. Assume Setup \ref{setup:allcases}, and recall that $m=7$. Then:
\begin{enumerate}[\normalfont (a)]
\item for any $w' \in W_{\mathscr{D}'}$ we have $ w'( \cal F_{{\mathscr D}_{m,m-1}}) = \cal F_{{\mathscr D}_{m,m-1}} $,
\item for $k\in\{0,\dots,m-1\}$ and
$$\widehat\Pi_{k,m-k-1}\coloneqq \big\{ k\delta+\alpha_0, (m-k-1)\delta + \alpha_1,\alpha_0',\alpha_1' \big\},$$
the $6$-tuple 
$$\widehat{\mathscr D}_{k,m-k-1}\coloneqq \big(B,\widehat\Pi_{k,m-k-1},\widehat\Pi_{k,m-k-1}^\smvee,N^1(X)_\R,N_1(X)_\R,\langle\cdot\, ,\cdot\rangle\big)$$
is a set of root data, where $B$ is the associated matrix given by the bilinear pairing between $\widehat\Pi_{k,m-k-1}$ and $\widehat\Pi_{k,m-k-1}^\smvee$.
\end{enumerate}
\end{lem}

\begin{proof}
By Setup \ref{setup:allcases}, the canonical central element of the affine root system of type $A_1^\smone$ associated to the roots $\{\alpha_0,\alpha_1\}$ is
$$\textstyle c=\alpha_0^\smvee+\alpha_1^\smvee=\frac1m(\alpha_0+\alpha_1)=\frac1m\delta,$$
and $\cal F_{{\mathscr D}_{m,m-1}}=\big\{ x\in V \mid \langle x, mc +\alpha_0^\smvee\rangle < 0,\ \langle x, (m-1) c+\alpha_1^\smvee \rangle <0 \big\}$. Then the proof of (a) is analogous to that of Proposition \ref{pro:decomposable}(e).

Now we show (b). The condition (i) in Definition \ref{dfn:root_data} follows from the construction. Denote $\widetilde\Pi_{k,m-k-1}\coloneqq \widehat\Pi_{k,m-k-1}\setminus\{k\delta+\alpha_0\}$. Then it is easy to check that the set $\widetilde\Pi_{k,m-k-1}$ is linearly independent in $N^1(X)_\R$ and that
$$ k\delta+\alpha_0 = m\alpha_0'+m\alpha_1' - \big((m-k-1)\delta+\alpha_1). $$ If we denote $\mathcal Q\coloneqq \sum_{\alpha\in \widehat\Pi_{k,m-k-1}}\Z \alpha$, then the previous equation implies that $\mathcal Q=\sum_{\alpha\in \widetilde\Pi_{k,m-k-1}}\Z \alpha$, hence $\mathcal Q$ is a lattice in $\mathcal Q\otimes_\Z \R$, and the condition (ii) in Definition \ref{dfn:root_data} holds. The proof that the condition (iii) in Definition \ref{dfn:root_data} holds is the same as the second paragraph of Step 1 of the proof of Proposition \ref{prop:subrootdata000}, by replacing $\Pi(R^\smperp)$ by $\widehat\Pi_{k,m-k-1}$. This finishes the proof.
\end{proof}

\section{A computational lemma}\label{sec:appendix_computations}

The goal of this section is to prove the following auxiliary lemma, which completes the proof of Corollary \ref{cor:finiteness}.

\begin{lem}\label{lem:aux_finiteness}
Let $X$ be a Sakai surface of type $R\in \{A_6^\smone, A_7^\smone, D_7^\smone \}$. Assume Setup \ref{setup:allcases}. Let $w\in W_{\mathscr{D}}$. With the notation from Step 2 of the proof of Corollary \ref{cor:finiteness}, we have
    $$\{ \beta\in\mathcal S_R(w) \mid \beta\cdot M_R^{-1}\cdot A_R\in\Z^9\} = \{ \beta_1,\beta_2\},$$
where the first coordinate of $\beta_1$ is different from the first coordinate of $\beta_2$.
\end{lem}

\begin{proof}
Let $k$ be the length of $w$, and assume that $w$ starts with $r_0$; the lemma is proved analogously when $w$ starts with $r_1$. We distinguish three cases.

\medskip

\emph{Step 1.}
In this step we prove the lemma when $R=A_6^\smone$.

Recall from Step 4 of the proof of Proposition \ref{prop:finmanytranslat} that
\begin{align}\label{eq:Sw_A6}
\mathcal S_R(w)\coloneqq &\{-k,-k-1\}\times\{0,-1\}\\
\times&\Big\{(c_1,\ldots,c_6)\in\Z_{\leq0}^6\ \big| \ \textstyle \sum\limits_{i=1}^6 c_i\in\{0,-1\}\Big\}.\notag
\end{align}
By Appendix \ref{sec:Sakai_appendix_A} we have
$$ M_R=\begin{psmallmatrix}
-14 & 0 & 0 & 0 & 0 & 0 & 0 & 0\\
0 & -2 & 0 & 0 & 0 & 0 & 0 & 0\\
0 & 0 & -2 & 1 & 0 & 0 & 0 & 0\\
0 & 0 & 1 & -2 & 1 & 0 & 0 & 0\\
0 & 0 & 0 & 1 & -2 & 1 & 0 & 0\\
0 & 0 & 0 & 0 & 1 & -2 & 1 & 0\\
0 & 0 & 0 & 0 & 0 & 1 & -2 & 1\\
0 & 0 & 0 & 0 & 0 & 0 & 1 & -2
\end{psmallmatrix}, \quad
A_R=\begin{psmallmatrix}
1&2&-1&-2&1&-2&-1&0&0\\
1&0&-1&0&-1&0&-1&0&0\\
1&0&-1&0&0&0&0&-1&-1\\
0&0&1&0&0&0&-1&0&0\\
1&-1&-1&-1&0&0&0&0&0\\
0&0&0&1&0&-1&0&0&0\\
1&0&0&-1&-1&0&0&-1&0\\
0&0&0&0&0&0&0&1&-1
\end{psmallmatrix}, $$
and hence,
$$ M_R^{-1}\cdot A_R=\textstyle\frac{1}{14}
\begin{psmallmatrix}
-1 & -2 & 1 & 2 & -1 & 2 & 1 & 0 & 0\\
-7 & 0 & 7 & 0 & 7 & 0 & 7 & 0 & 0\\
-24 & 8 & 10 & 6 & 4 & 6 & 10 & 14 & 14\\
-34 & 16 & 6 & 12 & 8 & 12 & 20 & 14 & 14\\
-44 & 24 & 16 & 18 & 12 & 18 & 16 & 14 & 14\\
-40 & 18 & 12 & 10 & 16 & 24 & 12 & 14 & 14\\
-36 & 12 & 8 & 16 & 20 & 16 & 8 & 14 & 14\\
-18 & 6 & 4 & 8 & 10 & 8 & 4 & 0 & 14
\end{psmallmatrix}. $$

Let $(a,b,c_1,\dots,c_6)\in\mathcal S_R(w)$, and denote
$$(x_0,\dots,x_8) \coloneqq (a,b,c_1,\dots,c_6)\cdot M_R^{-1}\cdot A_R.$$
We need to determine all $(a,b,c_1,\dots,c_6)$ such that $(x_0,\dots,x_8) \in\Z^9$. Note that $a\in\{-k,-k-1\}$ by \eqref{eq:Sw_A6}; for each of these two choices for $a$, we will show that there exists a unique $(a,b,c_1,\dots,c_6)$ such that $(x_0,\dots,x_8) \in\Z^9$, which will prove the lemma when $R=A_6^\smone$.

We see that $x_7$ and $x_8$ are always integral, and that
\begin{align*}
x_6-x_2&=c_2,\quad x_5-x_3=c_4,\\
x_1+2x_2&=b+2c_1+2c_2+4c_3+3c_4+2c_5+c_6,\\ 
x_0+x_6&=-c_1-c_2-2c_3-2c_4-2c_5-c_6,\\
x_4+x_6&=b+c_1+2c_2+2c_3+2c_4+2c_5+c_6,\\
x_1+x_3&=c_1+2c_2+3c_3+2c_4+2c_5+c_6.
\end{align*}
It follows from these equations that it suffices to determine all $(a,b,c_1,\dots,c_6)$ such that $x_2\in\Z$, or equivalently,
\begin{equation}\label{eq:A6}
\textstyle \frac{1}{14}(a+7b+10c_1+6c_2+16c_3+12c_4+8c_5+4c_6)\in\Z.
\end{equation}
This forces $a+7b$ to be an even number. Assume first that $a=2a'$ for some $a'\in\Z$. Since $b\in\{0,-1\}$ by \eqref{eq:Sw_A6}, we must have $b=0$. There exist unique integers $p$ and $0\leq q\leq6$ such that $a'=7p+q$. Then \eqref{eq:A6} is equivalent to
\begin{equation}\label{eq:A6_2}
\textstyle \frac17q+\frac57c_1+\frac37c_2+\frac17c_3+\frac67c_4+\frac47c_5+\frac27c_6\in\Z.
\end{equation}
Moreover, by \eqref{eq:Sw_A6} we have that either $c_i=0$ for all $i$, or there exists an index $i$ such that $c_i={-}1$ and $c_j=0$ for $j\neq i$. Since $q$ is determined uniquely by $a$, it is clear that there exists a unique $(c_1,\dots,c_6)$ such that \eqref{eq:A6_2} holds.

Now assume that $a=2a'+1$ for some $a'\in\Z$. Since $a+7b$ is even and since $b\in\{0,-1\}$ by \eqref{eq:Sw_A6}, we must have $b={-}1$. There exist unique integers $p$ and $0\leq q\leq6$ such that $a'=7p+q$. Then \eqref{eq:A6} is equivalent to
\begin{equation}\label{eq:A6_3}
\textstyle \frac17q+\frac47+\frac57c_1+\frac37c_2+\frac17c_3+\frac67c_4+\frac47c_5+\frac27c_6\in\Z.
\end{equation}
It follows as in the previous paragraph that there exists a unique $(c_1,\dots,c_6)$ such that \eqref{eq:A6_3} holds.

\medskip

\emph{Step 2.}
In this step we prove the lemma when $R=A_7^\smone$.

Recall from Step 2 of the proof of Proposition \ref{prop:finmanytranslat} that
\begin{equation}\label{eq:Sw_A7}
\mathcal S_R(w)=\{-k,-k-1\}\times\Big\{(c_1,\ldots,c_7)\in\Z_{\leq0}^7\ \big| \ \textstyle \sum\limits_{i=1}^7 c_i\in\{0,-1\}\Big\}.
\end{equation}
By Appendix \ref{sec:Sakai_appendix_A} we have
$$ M_R = \begin{psmallmatrix}
-8 & 0 & 0 & 0 & 0 & 0 & 0 & 0\\
0 & -2 & 1 & 0 & 0 & 0 & 0 & 0\\
0 & 1 & -2 & 1 & 0 & 0 & 0 & 0\\
0 & 0 & 1 & -2 & 1 & 0 & 0 & 0\\
0 & 0 & 0 & 1 & -2 & 1 & 0 & 0\\
0 & 0 & 0 & 0 & 1 & -2 & 1 & 0\\
0 & 0 & 0 & 0 & 0 & 1 & -2 & 1\\
0 & 0 & 0 & 0 & 0 & 0 & 1 & -2
\end{psmallmatrix}, \quad 
A_R=\begin{psmallmatrix}
2&1&-2&-1&-1&-1&-2&0&0\\
1&0&-1&0&0&0&0&-1&-1\\
0&0&1&0&0&0&-1&0&0\\
1&-1&-1&-1&0&0&0&0&0\\
0&0&0&1&-1&0&0&0&0\\
0&0&0&0&1&-1&0&0&0\\
1&0&0&-1&-1&0&0&-1&0\\
0&0&0&0&0&0&0&1&-1
\end{psmallmatrix}, $$
and hence,
$$M_R^{-1}\cdot A_R=\textstyle\frac18\begin{psmallmatrix}
-2 & -1 & 2 & 1 & 1 & 1 & 2 & 0 & 0 \\
-14 & 5 & 6 & 3 &  3 &  3 &  6 &  8 &  8\\
-20 &  10 &  4 &  6 &  6 &  6 &  12 &  8  & 8 \\
-26  & 15  & 10  & 9  & 9  & 9  & 10  & 8 &  8\\
-24  & 12  & 8  & 4  & 12  & 12  & 8  & 8 &  8\\
-22  & 9  & 6  & 7  & 7  & 15  & 6  & 8 &  8\\
-20  & 6  & 4  & 10  & 10  & 10  & 4  & 8 &  8\\
-10  & 3  & 2  & 5  & 5  & 5  & 2  & 0  & 8
\end{psmallmatrix}.$$

Let $(a,c_1,\dots,c_7)\in\mathcal S_R(w)$, and denote
$$(x_0,\dots,x_8) \coloneqq (a,c_1,\dots,c_7)\cdot M_R^{-1}\cdot A_R.$$
We need to determine all $(a,c_1,\dots,c_7)$ such that $(x_0,\dots,x_8) \in\Z^9$. Note that $a\in\{-k,-k-1\}$ by \eqref{eq:Sw_A7}; for each of these two choices for $a$, we will show that there exists a unique $(a,c_1,\dots,c_7)$ such that $(x_0,\dots,x_8) \in\Z^9$, which will prove the lemma when $R=A_7^\smone$.

We see that $x_7$ and $x_8$ are always integral, and that
\begin{align*}
x_4-x_3&=c_4,\quad x_5-x_4=c_5,\quad x_6-x_2=c_2,\\
-x_0-x_2&=c_1+2c_2+2c_3+2c_4+2c_5+2c_6+c_7,\\
x_1+x_3&=c_1+2c_2+3c_3+2c_4+2c_5+2c_6+c_7,\\
-x_0+2x_1&=3c_1+5c_2+7c_3+6c_4+5c_5+4c_6+2c_7.
\end{align*}
It follows from these equations that it suffices to determine all $(a,c_1,\dots,c_7)$ such that $x_1\in\Z$, or equivalently,
$$\textstyle -\frac18a+\frac58c_1+\frac54c_2+\frac{15}{8}c_3+\frac32c_4+\frac98c_5+\frac34c_6+\frac38c_7\in\Z.$$
There exist unique integers $p$ and $0\leq q\leq7$ such that $a=8p+q$, hence the condition above is equivalent to
\begin{equation}\label{eq:A7}
\textstyle -\frac{1}{8}q+\frac58c_1+\frac28c_2+\frac78c_3+\frac48c_4+\frac18c_5+\frac68c_6+\frac38c_7\in\Z.
\end{equation}
Moreover, by \eqref{eq:Sw_A7} we have that either $c_i=0$ for all $i$, or there exists an index $i$ such that $c_i={-}1$ and $c_j=0$ for $j\neq i$. Since $q$ is determined uniquely by $a$, it is clear that there exists a unique $(c_1,\dots,c_7)$ such that \eqref{eq:A7} holds.

\medskip

\emph{Step 3.}
In this step we prove the lemma when $R=D_7^\smone$.

Recall from Step 2 of the proof of Proposition \ref{prop:finmanytranslat} that we have
\begin{align}\label{eq:Sw_D7}
	\mathcal S_R(w):=\{&-k,-k-1\}\times\Big\{(c_1,\ldots,c_7)\in\Z_{\leq0}^7\ \big| \\
	&c_1+2c_2+2c_3+2c_4+2c_5+c_6+c_7\in\{0,-1\}\Big\}.\notag
	\end{align}
By Appendix \ref{sec:Sakai_appendix_A} we have
$$ M_R=\begin{psmallmatrix}
-4 & 0 & 0 & 0 & 0 & 0 & 0 & 0\\
0 & -2 & 1 & 0 & 0 & 0 & 0 & 0\\
0 & 1 & -2 & 1 & 0 & 0 & 0 & 0\\
0 & 0 & 1 & -2 & 1 & 0 & 0 & 0\\
0 & 0 & 0 & 1 & -2 & 1 & 0 & 0\\
0 & 0 & 0 & 0 & 1 & -2 & 1 & 1\\
0 & 0 & 0 & 0 & 0 & 1 & -2 & 0\\
0 & 0 & 0 & 0 & 0 & 1 & 0 & -2
\end{psmallmatrix},\quad
A_R=\begin{psmallmatrix}
0&0&0&0&1&-1&1&-1&0\\
1&-1&-1&-1&0&0&0&0&0\\
0&0&0&1&0&0&0&0&-1\\
0&0&1&-1&0&0&0&0&0\\
0&1&-1&0&0&0&0&0&0\\
1&-1&0&0&-1&-1&0&0&0\\
0&0&0&0&1&0&-1&0&0\\
0&0&0&0&0&1&0&-1&0
\end{psmallmatrix}, $$
and hence,
$$M_R^{-1}\cdot A_R=\textstyle\frac14\begin{psmallmatrix}
0 & 0 & 0 & 0 & -1 & 1 & -1 & 1 & 0\\
-8 & 4 & 4 & 4 & 2 & 2 & 2 & 2 & 4\\
-12 & 4 & 4 & 4 & 4 & 4 & 4 & 4 & 8\\
-16 & 4 & 4 & 8 & 6 & 6 & 6 & 6 & 8\\
-20 & 4 & 8 & 8 & 8 & 8 & 8 & 8 & 8\\
-24 & 8 & 8 & 8 & 10 & 10 & 10 & 10 & 8\\
-12 & 4 & 4 & 4 & 3 & 5 & 7 & 5 & 4\\
-12 & 4 & 4 & 4 & 5 & 3 & 5 & 7 & 4
\end{psmallmatrix}.$$

Let $(a,c_1,\dots,c_7)\in\mathcal S_R(w)$, and denote
$$(x_0,\dots,x_8) \coloneqq (a,c_1,\dots,c_7)\cdot M_R^{-1}\cdot A_R.$$
We need to determine all $(a,c_1,\dots,c_7)$ such that $(x_0,\dots,x_8) \in\Z^9$. Note that $a\in\{-k,-k-1\}$ by \eqref{eq:Sw_D7}; for each of these two choices for $a$, we will show that there exists a unique $(a,c_1,\dots,c_7)$ such that $(x_0,\dots,x_8) \in\Z^9$, which will prove the lemma when $R=D_7^\smone$.

We see that $x_0$, $x_1$, $x_2$, $x_3$ and $x_8$ are always integral, and that
\begin{align*}
x_4+x_5&=c_1+2c_2+3c_3+4c_4+5c_5+2c_6+2c_7,\\
x_6+x_7&= c_1+2c_2+3c_3+4c_4+5c_5+3c_6+3c_7,\\
x_5+x_6&= c_1+2c_2+3c_3+4c_4+5c_5+3c_6+2c_7.
\end{align*}
It follows from these equations that it suffices to determine all $(a,c_1,\dots,c_7)$ such that $x_5\in\Z$, or equivalently,
$$\textstyle \frac14a+\frac12c_1+c_2+\frac32c_3+2c_4+\frac52c_5+\frac54c_6+\frac34c_7\in\Z.$$
There exist unique integers $p$ and $0\leq q\leq3$ such that $a=4p+q$, and by \eqref{eq:Sw_D7} we necessarily have $c_2=c_3=c_4=c_5=0$. Thus, the condition above is equivalent to
\begin{equation}\label{eq:D7}
\textstyle \frac14q+\frac24c_1+\frac14c_6+\frac34c_7\in\Z.
\end{equation}
Moreover, by \eqref{eq:Sw_D7} we have that either $c_i=0$ for all $i$, or there exists an index $i$ such that $c_i={-}1$ and $c_j=0$ for $j\neq i$. Since $q$ is determined uniquely by $a$, it is clear that there exists a unique $(c_1,\dots,c_7)$ such that \eqref{eq:D7} holds. This concludes the proof.
\end{proof}

\section{Root bases of Sakai surfaces}\label{sec:Sakai_appendix_A}

Let $X$ be a Sakai surface of type $R$. In this section, we recall \cite[Appendix A]{Sak01} for the convenience of the reader, since we use it very often in the paper.

Our notation below differs slightly from that in \cite[Appendix A]{Sak01} when $R\in \{A_6^\smone, A_7^\smone, E_7^\smone\}$, since these changes are adapted to our computations. Concretely: (i) when $R=A^{(1)}_6$, we interchange $\alpha_0$ with $\alpha'_0$ and $\alpha_1$ with $\alpha'_1$; (ii) when $R=A^{(1)}_7$, we interchange $\alpha_0$ with $\alpha_1$; (iii) when $R=E^{(1)}_7$, we interchange $\mathcal{D}_0$ with $\mathcal{D}_6$.

With notation from \S\ref{subsec:root_system_R}, \S\ref{subsec:Rperp1}, and \S\ref{subsec:Rperp2}, recall that $(\mathcal{E}_0,\ldots,\mathcal{E}_9)$ denotes the tautological basis of $N^1(X)_\R $ defining the blowdown structure $X\to\mathbb{P}^2$, $D$ is the unique element in $|{-}K_X|$, and $\mathcal{D}_i$ are the numerical classes of irreducible components of $D$. In the list below, for each type $R$ we list the sets $\Pi(R)$ and $\Pi(R^{\smperp})$, as well as the imaginary root $\delta=[{-}K_X]$.

Note that if $R\in\{A_0^\smone,A_0^{\smone*},A_0^{\smone**}\}$, then $R$ has no real roots, and if $R\in\{A_8^{\smone},D_8^{\smone},E_8^{\smone}\}$, then $R^\smperp$ has no real roots.

As in \cite[Appendix A]{Sak01}, we use the notation $(A_2+A_1)^\smone$ (respectively $2A_1^\smone$) to denote sets of root data which decompose into two root systems of affine type $A_2^\smone$ and $A_1^\smone$ (respectively of affine type $A_1^\smone$ and $A_1^\smone$); see also our conventions in \S\ref{subsec:Rperp1} and \S\ref{subsec:Rperp2}.

{\small
\newpage

\begin{center}
$\boxed{R\in\{A_0^\smone,A_0^{\smone*},A_0^{\smone**}\},\qquad R^\smperp = E_8^\smone}$
\end{center}
\begin{flalign*}
    \underline{\Pi(R^\smperp)}: \  \alpha_0 &= \mathcal{E}_8-\mathcal{E}_9,\quad \alpha_1 = \mathcal{E}_1-\mathcal{E}_2,\quad \alpha_2 = \mathcal{E}_2-\mathcal{E}_3, \quad \alpha_3 = \mathcal{E}_3-\mathcal{E}_4, & \\
    \alpha_4 &= \mathcal{E}_4-\mathcal{E}_5,\quad \alpha_5 = \mathcal{E}_5-\mathcal{E}_6, \quad \alpha_6 = \mathcal{E}_6-\mathcal{E}_7,\quad \alpha_7 = \mathcal{E}_7-\mathcal{E}_8, \\
    \alpha_8 &= \mathcal{E}_0-\mathcal{E}_1-\mathcal{E}_2-\mathcal{E}_3
\end{flalign*}
$\delta = \alpha_0+2\alpha_1+4\alpha_2+6\alpha_3+5\alpha_4+4\alpha_5+3\alpha_6+2\alpha_7+3\alpha_8$

\noindent\hrulefill

\begin{center}
$\boxed{R\in\{A_1^{\smone},A_1^{\smone*}\},\qquad R^\smperp = E_7^\smone}$
\end{center}
\begin{flalign*}
    \underline{\Pi(R)}: \ \hspace{4,5pt}  \mathcal{D}_0&=2\mathcal{E}_0-\mathcal{E}_4-\mathcal{E}_5-\mathcal{E}_6-\mathcal{E}_7-\mathcal{E}_8-\mathcal{E}_9, \quad \mathcal{D}_1 =\mathcal{E}_0-\mathcal{E}_1-\mathcal{E}_2-\mathcal{E}_3 &
\end{flalign*}
\vspace{-15pt}
\begin{flalign*}
    \underline{\Pi(R^\smperp)}: \  \alpha_0 &= \mathcal{E}_8-\mathcal{E}_9,\quad \alpha_1 = \mathcal{E}_2-\mathcal{E}_3, \quad \alpha_2 = \mathcal{E}_1-\mathcal{E}_2, \quad \alpha_3 = \mathcal{E}_0-\mathcal{E}_1-\mathcal{E}_4-\mathcal{E}_5, &  \\
    \alpha_4 &= \mathcal{E}_5-\mathcal{E}_6, \quad \alpha_5 = \mathcal{E}_6-\mathcal{E}_7, \quad \alpha_6 = \mathcal{E}_7-\mathcal{E}_8, \quad \alpha_7 = \mathcal{E}_4-\mathcal{E}_5  
\end{flalign*}
$\delta = \mathcal{D}_0+\mathcal{D}_1 = \alpha_0 + \alpha_1 + 2\alpha_2 + 3\alpha_3 + 4\alpha_4 + 3\alpha_5 + 2\alpha_6 + 2\alpha_7$

\noindent\hrulefill

\begin{center}
$\boxed{R\in\{A_2^{\smone},A_2^{\smone*}\},\qquad R^\smperp = E_6^\smone}$
\end{center}
\begin{flalign*}
    \underline{\Pi(R)}: \ \hspace{4,5pt}  \mathcal{D}_0&=\mathcal{E}_0-\mathcal{E}_4-\mathcal{E}_5-\mathcal{E}_9,\quad \mathcal{D}_1=\mathcal{E}_0-\mathcal{E}_6-\mathcal{E}_7-\mathcal{E}_8, & \\
    \mathcal{D}_2 &= \mathcal{E}_0-\mathcal{E}_1-\mathcal{E}_2-\mathcal{E}_3
\end{flalign*}
\vspace{-15pt}
\begin{flalign*}
    \underline{\Pi(R^\smperp)}: \  \alpha_0 &= \mathcal{E}_5-\mathcal{E}_9,\quad \alpha_1 = \mathcal{E}_2-\mathcal{E}_3, \quad \alpha_2 = \mathcal{E}_1-\mathcal{E}_2, \quad \alpha_3 = \mathcal{E}_0-\mathcal{E}_1-\mathcal{E}_4-\mathcal{E}_6, &  \\
    \alpha_4 &= \mathcal{E}_6-\mathcal{E}_7, \quad \alpha_5 = \mathcal{E}_7-\mathcal{E}_8, \quad \alpha_6 = \mathcal{E}_4-\mathcal{E}_5
\end{flalign*}
$\delta = \mathcal{D}_0 + \mathcal{D}_1 + \mathcal{D}_2 = \alpha_0 + \alpha_1 + 2\alpha_2 + 3\alpha_3 + 2\alpha_4 + \alpha_5 + 2\alpha_6$

\noindent\hrulefill

\begin{center}
$\boxed{R=A_3^{\smone},\qquad R^\smperp = D_5^\smone}$
\end{center}
\begin{flalign*}
    \underline{\Pi(R)}: \ \hspace{4,5pt} \mathcal{D}_0 &= \mathcal{E}_8-\mathcal{E}_9,\quad \mathcal{D}_1=\mathcal{E}_0-\mathcal{E}_6-\mathcal{E}_7-\mathcal{E}_8, \quad \mathcal{D}_2 = \mathcal{E}_0-\mathcal{E}_1-\mathcal{E}_2-\mathcal{E}_3, & \\
    \mathcal{D}_3 &= \mathcal{E}_0-\mathcal{E}_4-\mathcal{E}_5-\mathcal{E}_8
\end{flalign*}
\vspace{-15pt}
\begin{flalign*}
    \underline{\Pi(R^\smperp)}: \  \alpha_0 &= \mathcal{E}_0-\mathcal{E}_1-\mathcal{E}_8-\mathcal{E}_9,\quad \alpha_1 = \mathcal{E}_2-\mathcal{E}_3, \quad \alpha_2 = \mathcal{E}_1-\mathcal{E}_2, & \\
    \alpha_3 &= \mathcal{E}_0-\mathcal{E}_1-\mathcal{E}_4-\mathcal{E}_6, \quad \alpha_4 = \mathcal{E}_6-\mathcal{E}_7, \quad \alpha_5 = \mathcal{E}_4-\mathcal{E}_5
\end{flalign*}
$\delta = \mathcal{D}_0 + \mathcal{D}_1 + \mathcal{D}_2 + \mathcal{D}_3 = \alpha_0 + \alpha_1 + 2\alpha_2 + 2\alpha_3 + \alpha_4 + \alpha_5$

\noindent\hrulefill

\begin{center}
$\boxed{R=A_4^{\smone},\qquad R^\smperp = A_4^\smone}$
\end{center}
\begin{flalign*}
    \underline{\Pi(R)}: \ \hspace{4,5pt} \mathcal{D}_0 &= \mathcal{E}_8-\mathcal{E}_9,\quad \mathcal{D}_1=\mathcal{E}_0-\mathcal{E}_6-\mathcal{E}_7-\mathcal{E}_8, \quad \mathcal{D}_2 = \mathcal{E}_0-\mathcal{E}_1-\mathcal{E}_2-\mathcal{E}_3, & \\
    \mathcal{D}_3 &= \mathcal{E}_0-\mathcal{E}_4-\mathcal{E}_5-\mathcal{E}_7, \quad \mathcal{D}_4 = \mathcal{E}_7-\mathcal{E}_8
\end{flalign*}
\vspace{-15pt}
\begin{flalign*}
    \underline{\Pi(R^\smperp)}: \  \alpha_0 &= 2\mathcal{E}_0-\mathcal{E}_1-\mathcal{E}_2-\mathcal{E}_4-\mathcal{E}_7-\mathcal{E}_8-\mathcal{E}_9, \quad \alpha_1 = \mathcal{E}_2-\mathcal{E}_3, \quad \alpha_2 = \mathcal{E}_1-\mathcal{E}_2, & \\
    \alpha_3 &= \mathcal{E}_0-\mathcal{E}_1-\mathcal{E}_4-\mathcal{E}_6, \quad \alpha_4 = \mathcal{E}_4-\mathcal{E}_5
\end{flalign*}
$\delta = \mathcal{D}_0 + \mathcal{D}_1 + \mathcal{D}_2 + \mathcal{D}_3 + \mathcal{D}_4 = \alpha_0 + \alpha_1 + \alpha_2 + \alpha_3 + \alpha_4$

\noindent\hrulefill
\newpage

\begin{center}
$\boxed{R=A_5^{\smone},\qquad R^\smperp = (A_2+A_1)^\smone}$
\end{center}
\begin{flalign*}
    \underline{\Pi(R)}: \ \hspace{4,5pt} \mathcal{D}_0 &= \mathcal{E}_8-\mathcal{E}_9, \quad \mathcal{D}_1=\mathcal{E}_0-\mathcal{E}_6-\mathcal{E}_7-\mathcal{E}_8, \quad \mathcal{D}_2 = \mathcal{E}_0-\mathcal{E}_1-\mathcal{E}_2-\mathcal{E}_3, & \\
    \mathcal{D}_3 &= \mathcal{E}_3-\mathcal{E}_5, \quad \mathcal{D}_4 =\mathcal{E}_0-\mathcal{E}_3- \mathcal{E}_4-\mathcal{E}_7,\quad \mathcal{D}_5 = \mathcal{E}_7-\mathcal{E}_8
\end{flalign*}
\vspace{-15pt}
\begin{flalign*}
    \underline{\Pi(R^\smperp)}: \  \alpha_0 &= 2\mathcal{E}_0-\mathcal{E}_1-\mathcal{E}_3-\mathcal{E}_5-\mathcal{E}_7-\mathcal{E}_8-\mathcal{E}_9, \quad \alpha_1 = \mathcal{E}_1-\mathcal{E}_2, & \\
    \alpha_2 &= \mathcal{E}_0-\mathcal{E}_1-\mathcal{E}_4-\mathcal{E}_6, \quad \alpha'_0 = 2\mathcal{E}_0-\mathcal{E}_1-\mathcal{E}_2-\mathcal{E}_4-\mathcal{E}_7-\mathcal{E}_8-\mathcal{E}_9, \\
    \alpha'_1 &= \mathcal{E}_0-\mathcal{E}_3-\mathcal{E}_5-\mathcal{E}_6
\end{flalign*}
$\delta = \mathcal{D}_0 + \mathcal{D}_1 + \mathcal{D}_2 + \mathcal{D}_3 + \mathcal{D}_4 + \mathcal{D}_5 = \alpha_0 + \alpha_1 + \alpha_2 = \alpha'_0 + \alpha'_1$

\noindent\hrulefill

\begin{center}
$\boxed{R=A_6^{\smone},\qquad R^\smperp = 2A_1^\smone}$
\end{center}
\begin{flalign*}
    \underline{\Pi(R)}: \ \hspace{4,5pt} \mathcal{D}_0 &= \mathcal{E}_8-\mathcal{E}_9,\quad \mathcal{D}_1=\mathcal{E}_0-\mathcal{E}_2-\mathcal{E}_7-\mathcal{E}_8, \quad \mathcal{D}_2 = \mathcal{E}_2-\mathcal{E}_6, & \\
    \mathcal{D}_3 &= \mathcal{E}_0-\mathcal{E}_1-\mathcal{E}_2-\mathcal{E}_3, \quad\mathcal{D}_4=\mathcal{E}_3-\mathcal{E}_5, \\
    \mathcal{D}_5 &=\mathcal{E}_0-\mathcal{E}_3- \mathcal{E}_4-\mathcal{E}_7, \quad \mathcal{D}_6 = \mathcal{E}_7-\mathcal{E}_8
\end{flalign*}
\vspace{-15pt}
\begin{flalign*}
    \underline{\Pi(R^\smperp)}: \  \alpha_0 &= 2\mathcal{E}_0-3\mathcal{E}_1+\mathcal{E}_3-2\mathcal{E}_4+\mathcal{E}_5-\mathcal{E}_7-\mathcal{E}_8-\mathcal{E}_9, & \\
    \alpha_1 &= \mathcal{E}_0+2\mathcal{E}_1-\mathcal{E}_2-2\mathcal{E}_3+\mathcal{E}_4-2\mathcal{E}_5-\mathcal{E}_6, \\
    \alpha'_0 &= 2\mathcal{E}_0-\mathcal{E}_1-\mathcal{E}_3-\mathcal{E}_5-\mathcal{E}_7-\mathcal{E}_8-\mathcal{E}_9, \quad \alpha'_1 = \mathcal{E}_0-\mathcal{E}_2-\mathcal{E}_4-\mathcal{E}_6
\end{flalign*}
$\delta = \mathcal{D}_0 + \mathcal{D}_1 + \mathcal{D}_2 + \mathcal{D}_3 + \mathcal{D}_4 + \mathcal{D}_5 + \mathcal{D}_6 = \alpha_0 + \alpha_1 = \alpha'_0 + \alpha'_1$

\noindent\hrulefill

\begin{center}
$\boxed{R=A_7^{\smone},\qquad R^\smperp = A_1^\smone}$
\end{center}
\begin{flalign*}
    \underline{\Pi(R)}: \ \hspace{4,5pt} \mathcal{D}_0 &= \mathcal{E}_8-\mathcal{E}_9, \quad \mathcal{D}_1=\mathcal{E}_0-\mathcal{E}_2-\mathcal{E}_7-\mathcal{E}_8, \quad \mathcal{D}_2 = \mathcal{E}_2-\mathcal{E}_6, & \\
    \mathcal{D}_3 &= \mathcal{E}_0-\mathcal{E}_1-\mathcal{E}_2-\mathcal{E}_3, \quad\mathcal{D}_4=\mathcal{E}_3-\mathcal{E}_4, \quad \mathcal{D}_5 = \mathcal{E}_4-\mathcal{E}_5, \\
    \mathcal{D}_6 &=\mathcal{E}_0-\mathcal{E}_3- \mathcal{E}_4-\mathcal{E}_7, \quad \mathcal{D}_7 = \mathcal{E}_7-\mathcal{E}_8
\end{flalign*}
\vspace{-15pt}
\begin{flalign*}
    \underline{\Pi(R^\smperp)}: \  \alpha_0 &= \mathcal{E}_0-2\mathcal{E}_1+\mathcal{E}_2+\mathcal{E}_6-\mathcal{E}_7-\mathcal{E}_8-\mathcal{E}_9, &\\
    \alpha_1 &= 2\mathcal{E}_0+\mathcal{E}_1-2\mathcal{E}_2-\mathcal{E}_3-\mathcal{E}_4-\mathcal{E}_5-2\mathcal{E}_6
\end{flalign*}
$\delta = \mathcal{D}_0 + \mathcal{D}_1 + \mathcal{D}_2 + \mathcal{D}_3 + \mathcal{D}_4 + \mathcal{D}_5 + \mathcal{D}_6 + \mathcal{D}_7 = \alpha_0 + \alpha_1$

\noindent\hrulefill

\begin{center}
$\boxed{R=A_7^{\smone'},\qquad R^\smperp = A_1^\smone}$
\end{center}
\begin{flalign*}
    \underline{\Pi(R)}: \ \hspace{4,5pt} \mathcal{D}_0 &= \mathcal{E}_8-\mathcal{E}_9,\quad \mathcal{D}_1=\mathcal{E}_0-\mathcal{E}_2-\mathcal{E}_7-\mathcal{E}_8, \quad \mathcal{D}_2 = \mathcal{E}_2-\mathcal{E}_6, & \\
    \mathcal{D}_3 &= \mathcal{E}_0-\mathcal{E}_1-\mathcal{E}_2-\mathcal{E}_3,\quad \mathcal{D}_4=\mathcal{E}_3-\mathcal{E}_5,\quad \mathcal{D}_5 = \mathcal{E}_1-\mathcal{E}_3, \\
    \mathcal{D}_6 &=\mathcal{E}_0-\mathcal{E}_1- \mathcal{E}_4-\mathcal{E}_7, \quad \mathcal{D}_7 = \mathcal{E}_7-\mathcal{E}_8
\end{flalign*}
\vspace{-15pt}
\begin{flalign*}
    \underline{\Pi(R^\smperp)}: \  \alpha_0 &= 2\mathcal{E}_0-\mathcal{E}_1-\mathcal{E}_3-\mathcal{E}_5-\mathcal{E}_7-\mathcal{E}_8-\mathcal{E}_9, \quad \alpha_1 = \mathcal{E}_0-\mathcal{E}_2-\mathcal{E}_4-\mathcal{E}_6 &
\end{flalign*}
$\delta = \mathcal{D}_0 + \mathcal{D}_1 + \mathcal{D}_2 + \mathcal{D}_3 + \mathcal{D}_4 + \mathcal{D}_5 + \mathcal{D}_6 + \mathcal{D}_7 = \alpha_0 + \alpha_1$

\noindent\hrulefill
\newpage

\begin{center}
$\boxed{R=A_8^{\smone},\qquad R^\smperp = A_0^\smone}$
\end{center}
\begin{flalign*}
    \underline{\Pi(R)}: \ \hspace{4,5pt} \mathcal{D}_0 &= \mathcal{E}_8-\mathcal{E}_9,\quad \mathcal{D}_1=\mathcal{E}_0-\mathcal{E}_1-\mathcal{E}_7-\mathcal{E}_8, \quad \mathcal{D}_2 = \mathcal{E}_1-\mathcal{E}_2, \quad \mathcal{D}_3 = \mathcal{E}_2-\mathcal{E}_6, & \\
    \mathcal{D}_4 &= \mathcal{E}_0-\mathcal{E}_1-\mathcal{E}_2-\mathcal{E}_3, \quad \mathcal{D}_5=\mathcal{E}_3-\mathcal{E}_4, \quad \mathcal{D}_6 = \mathcal{E}_4-\mathcal{E}_5, \\
    \mathcal{D}_7 &=\mathcal{E}_0-\mathcal{E}_3- \mathcal{E}_4-\mathcal{E}_7, \quad \mathcal{D}_8 = \mathcal{E}_7-\mathcal{E}_8
\end{flalign*}
$\delta = \mathcal{D}_0 + \mathcal{D}_1 + \mathcal{D}_2 + \mathcal{D}_3 + \mathcal{D}_4 + \mathcal{D}_5 + \mathcal{D}_6 + \mathcal{D}_7 + \mathcal{D}_8$

\noindent\hrulefill

\begin{center}
$\boxed{R=D_4^{\smone},\qquad R^\smperp = D_4^\smone}$
\end{center}
\begin{flalign*}
    \underline{\Pi(R)}: \ \hspace{4,5pt} \mathcal{D}_0 &= \mathcal{E}_8-\mathcal{E}_9, \quad \mathcal{D}_1=\mathcal{E}_0-\mathcal{E}_1-\mathcal{E}_2-\mathcal{E}_3, \quad \mathcal{D}_2 = \mathcal{E}_1 -\mathcal{E}_8, &\\
    \mathcal{D}_3 &= \mathcal{E}_0-\mathcal{E}_1-\mathcal{E}_4-\mathcal{E}_5, \quad \mathcal{D}_4=\mathcal{E}_0-\mathcal{E}_1-\mathcal{E}_6-\mathcal{E}_7
\end{flalign*}
\vspace{-15pt}
\begin{flalign*}
    \underline{\Pi(R^\smperp)}: \  \alpha_0 &= \mathcal{E}_0-\mathcal{E}_1-\mathcal{E}_8-\mathcal{E}_9,\quad \alpha_1 = \mathcal{E}_2-\mathcal{E}_3, \quad \alpha_2 = \mathcal{E}_0-\mathcal{E}_2-\mathcal{E}_4-\mathcal{E}_6, &\\
    \alpha_3 &= \mathcal{E}_4-\mathcal{E}_5, \quad \alpha_4 = \mathcal{E}_6-\mathcal{E}_7
\end{flalign*}
$\delta = \mathcal{D}_0 + \mathcal{D}_1 + 2\mathcal{D}_2 + \mathcal{D}_3 + \mathcal{D}_4 = \alpha_0 + \alpha_1 + 2\alpha_2 + \alpha_3 + \alpha_4$

\noindent\hrulefill

\begin{center}
$\boxed{R=D_5^{\smone},\qquad R^\smperp = A_3^\smone}$
\end{center}
\begin{flalign*}
    \underline{\Pi(R)}: \ \hspace{4,5pt} \mathcal{D}_0 &= \mathcal{E}_8-\mathcal{E}_9, \quad \mathcal{D}_1=\mathcal{E}_0-\mathcal{E}_1-\mathcal{E}_2-\mathcal{E}_3, \quad \mathcal{D}_2 = \mathcal{E}_2 -\mathcal{E}_8, \quad \mathcal{D}_3 = \mathcal{E}_1-\mathcal{E}_2, &\\
    \mathcal{D}_4 &= \mathcal{E}_0-\mathcal{E}_1-\mathcal{E}_4-\mathcal{E}_5, \quad \mathcal{D}_5 = \mathcal{E}_0-\mathcal{E}_1-\mathcal{E}_6-\mathcal{E}_7
\end{flalign*}
\vspace{-15pt}
\begin{flalign*}
    \underline{\Pi(R^\smperp)}: \  \alpha_0 &= 2\mathcal{E}_0-\mathcal{E}_1-\mathcal{E}_2-\mathcal{E}_4-\mathcal{E}_6-\mathcal{E}_8-\mathcal{E}_9, \quad \alpha_1 = \mathcal{E}_4-\mathcal{E}_5, &\\
    \alpha_2 &= \mathcal{E}_0-\mathcal{E}_3-\mathcal{E}_4-\mathcal{E}_6, \quad \alpha_3 = \mathcal{E}_6-\mathcal{E}_7
\end{flalign*}
$\delta = \mathcal{D}_0 + \mathcal{D}_1 + 2\mathcal{D}_2 + 2\mathcal{D}_3 + \mathcal{D}_4 + \mathcal{D}_5 = \alpha_0 + \alpha_1 + \alpha_2 + \alpha_3$

\noindent\hrulefill

\begin{center}
$\boxed{R=D_6^{\smone},\qquad R^\smperp = 2A_1^\smone}$
\end{center}
\begin{flalign*}
    \underline{\Pi(R)}: \ \hspace{4,5pt} \mathcal{D}_0 &= \mathcal{E}_8-\mathcal{E}_9, \quad \mathcal{D}_1=\mathcal{E}_0-\mathcal{E}_1-\mathcal{E}_2-\mathcal{E}_3, \quad \mathcal{D}_2 = \mathcal{E}_3 -\mathcal{E}_8, \quad \mathcal{D}_3 = \mathcal{E}_2 - \mathcal{E}_3, & \\
    \mathcal{D}_4 &= \mathcal{E}_1-\mathcal{E}_2,\quad \mathcal{D}_5 = \mathcal{E}_0-\mathcal{E}_1-\mathcal{E}_4-\mathcal{E}_5, \quad \mathcal{D}_6 = \mathcal{E}_0-\mathcal{E}_1-\mathcal{E}_6-\mathcal{E}_7 
\end{flalign*}
\vspace{-15pt}
\begin{flalign*}
    \underline{\Pi(R^\smperp)}: \ \alpha_0 &= 3\mathcal{E}_0-\mathcal{E}_1-\mathcal{E}_2-\mathcal{E}_3-2\mathcal{E}_4-\mathcal{E}_6-\mathcal{E}_7-\mathcal{E}_8-\mathcal{E}_9, \quad \alpha_1 = \mathcal{E}_4-\mathcal{E}_5, &\\
    \alpha'_0 &= 3\mathcal{E}_0-\mathcal{E}_1-\mathcal{E}_2-\mathcal{E}_3-\mathcal{E}_4-\mathcal{E}_5-2\mathcal{E}_6-\mathcal{E}_8-\mathcal{E}_9, \quad \alpha'_1 = \mathcal{E}_6-\mathcal{E}_7
\end{flalign*}
$\delta = \mathcal{D}_0 + \mathcal{D}_1 + 2\mathcal{D}_2 + 2\mathcal{D}_3 + 2\mathcal{D}_4 + \mathcal{D}_5 + \mathcal{D}_6 = \alpha_0 + \alpha_1 = \alpha'_0 + \alpha'_1$

\noindent\hrulefill

\begin{center}
$\boxed{R=D_7^{\smone},\qquad R^\smperp = A_1^\smone}$
\end{center}
\begin{flalign*}
    \underline{\Pi(R)}: \ \hspace{4,5pt} \mathcal{D}_0 &= \mathcal{E}_8-\mathcal{E}_9, \quad \mathcal{D}_1=\mathcal{E}_0-\mathcal{E}_1-\mathcal{E}_2-\mathcal{E}_3, \quad \mathcal{D}_2 = \mathcal{E}_3 -\mathcal{E}_8, \quad \mathcal{D}_3 = \mathcal{E}_2 - \mathcal{E}_3, & \\
    \mathcal{D}_4 &= \mathcal{E}_1-\mathcal{E}_2, \quad \mathcal{D}_5 = \mathcal{E}_0-\mathcal{E}_1-\mathcal{E}_4-\mathcal{E}_5, \quad \mathcal{D}_6 = \mathcal{E}_4-\mathcal{E}_6, \quad \mathcal{D}_7 = \mathcal{E}_5-\mathcal{E}_7
\end{flalign*}
\vspace{-15pt}
\begin{flalign*}
    \underline{\Pi(R^\smperp)}: \  \alpha_0 &= 3\mathcal{E}_0-\mathcal{E}_1-\mathcal{E}_2-\mathcal{E}_3-2\mathcal{E}_4-2\mathcal{E}_6-\mathcal{E}_8-\mathcal{E}_9, \quad \alpha_1 = \mathcal{E}_4-\mathcal{E}_5+\mathcal{E}_6-\mathcal{E}_7 &
\end{flalign*}
$\delta = \mathcal{D}_0 + \mathcal{D}_1 + 2\mathcal{D}_2 + 2\mathcal{D}_3 + 2\mathcal{D}_4 + 2\mathcal{D}_5 + \mathcal{D}_6 + \mathcal{D}_7 = \alpha_0 + \alpha_1$

\noindent\hrulefill
\newpage

\begin{center}
$\boxed{R=D_8^{\smone},\qquad R^\smperp = A_0^\smone}$
\end{center}
\begin{flalign*}
    \underline{\Pi(R)}: \ \hspace{4,5pt} \mathcal{D}_0 &= \mathcal{E}_8-\mathcal{E}_9, \quad \mathcal{D}_1=\mathcal{E}_0-\mathcal{E}_1-\mathcal{E}_2-\mathcal{E}_3, \quad \mathcal{D}_2 = \mathcal{E}_3 -\mathcal{E}_8, \quad \mathcal{D}_3 = \mathcal{E}_2 - \mathcal{E}_3, & \\
    \mathcal{D}_4 &= \mathcal{E}_1-\mathcal{E}_2, \quad \mathcal{D}_5 = \mathcal{E}_0-\mathcal{E}_1-\mathcal{E}_4-\mathcal{E}_5, \quad \mathcal{D}_6 = \mathcal{E}_5-\mathcal{E}_6, \quad \mathcal{D}_7 = \mathcal{E}_4-\mathcal{E}_5, \\
    \mathcal{D}_8 &= \mathcal{E}_6-\mathcal{E}_7
\end{flalign*}
$\delta = \mathcal{D}_0 + \mathcal{D}_1 + 2\mathcal{D}_2 + 2\mathcal{D}_3 + 2\mathcal{D}_4 + 2\mathcal{D}_5 + 2\mathcal{D}_6 + \mathcal{D}_7 + \mathcal{D}_8$

\noindent\hrulefill

\begin{center}
$\boxed{R=E_6^{\smone},\qquad R^\smperp = A_2^\smone}$
\end{center}
\begin{flalign*}
    \underline{\Pi(R)}: \ \hspace{4,5pt} \mathcal{D}_0 &= \mathcal{E}_8-\mathcal{E}_9, \quad \mathcal{D}_1=\mathcal{E}_0-\mathcal{E}_1-\mathcal{E}_4-\mathcal{E}_5, \quad \mathcal{D}_2 = \mathcal{E}_1 -\mathcal{E}_2, \quad \mathcal{D}_3 = \mathcal{E}_2 - \mathcal{E}_7, & \\
    \mathcal{D}_4 &= \mathcal{E}_0-\mathcal{E}_1-\mathcal{E}_2-\mathcal{E}_3, \quad \mathcal{D}_5=\mathcal{E}_3-\mathcal{E}_6, \quad \mathcal{D}_6 = \mathcal{E}_7-\mathcal{E}_8
\end{flalign*}
\vspace{-15pt}
\begin{flalign*}
    \underline{\Pi(R^\smperp)}: \  \alpha_0 &= 2\mathcal{E}_0-\mathcal{E}_1-\mathcal{E}_2-\mathcal{E}_4-\mathcal{E}_7-\mathcal{E}_8-\mathcal{E}_9, \quad \alpha_1 = \mathcal{E}_4-\mathcal{E}_5, & \\
    \alpha_2 &= \mathcal{E}_0-\mathcal{E}_3-\mathcal{E}_4-\mathcal{E}_6 
\end{flalign*}
$\delta = \mathcal{D}_0 + \mathcal{D}_1 + 2\mathcal{D}_2 + 3\mathcal{D}_3 + 2\mathcal{D}_4 + \mathcal{D}_5 + 2\mathcal{D}_6 = \alpha_0 + \alpha_1 + \alpha_2$

\noindent\hrulefill

\begin{center}
$\boxed{R=E_7^{\smone},\qquad R^\smperp = A_1^\smone}$
\end{center}
\begin{flalign*}
    \underline{\Pi(R)}: \ \hspace{4,5pt} \mathcal{D}_0 &= \mathcal{E}_8-\mathcal{E}_9, \quad \mathcal{D}_1 = \mathcal{E}_1 -\mathcal{E}_2, \quad \mathcal{D}_2 = \mathcal{E}_2 - \mathcal{E}_3, \quad \mathcal{D}_3 = \mathcal{E}_3-\mathcal{E}_6, & \\
    \mathcal{D}_4 &= \mathcal{E}_6-\mathcal{E}_7, \quad \mathcal{D}_5 = \mathcal{E}_7-\mathcal{E}_8, \quad \mathcal{D}_6 = \mathcal{E}_0-\mathcal{E}_1-\mathcal{E}_4-\mathcal{E}_5, \\
    \mathcal{D}_7 &= \mathcal{E}_0-\mathcal{E}_1-\mathcal{E}_2-\mathcal{E}_3 
\end{flalign*}
\vspace{-15pt}
\begin{flalign*}
    \underline{\Pi(R^\smperp)}: \  \alpha_0 &= 3\mathcal{E}_0-\mathcal{E}_1-\mathcal{E}_2-\mathcal{E}_3-2\mathcal{E}_4-\mathcal{E}_6-\mathcal{E}_7-\mathcal{E}_8-\mathcal{E}_9, \quad \alpha_1 = \mathcal{E}_4-\mathcal{E}_5 &
\end{flalign*}
$\delta = \mathcal{D}_0 + 2\mathcal{D}_1 + 3\mathcal{D}_2 + 4\mathcal{D}_3 + 3\mathcal{D}_4 + 2\mathcal{D}_5 + \mathcal{D}_6 + 2\mathcal{D}_7 = \alpha_0 + \alpha_1$

\noindent\hrulefill

\begin{center}
$\boxed{R=E_8^{\smone},\qquad R^\smperp = A_0^\smone}$
\end{center}
\begin{flalign*}
    \underline{\Pi(R)}: \ \hspace{4,5pt} \mathcal{D}_0 &= \mathcal{E}_8-\mathcal{E}_9, \quad \mathcal{D}_1=\mathcal{E}_1-\mathcal{E}_2, \quad \mathcal{D}_2=\mathcal{E}_2-\mathcal{E}_3, \quad \mathcal{D}_3 = \mathcal{E}_3-\mathcal{E}_4, & \\
    \mathcal{D}_4 &= \mathcal{E}_4-\mathcal{E}_5, \quad \mathcal{D}_5=\mathcal{E}_5-\mathcal{E}_6, \quad \mathcal{D}_6 = \mathcal{E}_6-\mathcal{E}_7, \quad \mathcal{D}_7=\mathcal{E}_7-\mathcal{E}_8, \\
    \mathcal{D}_8 &= \mathcal{E}_0-\mathcal{E}_1-\mathcal{E}_2-\mathcal{E}_3 
\end{flalign*}
$\delta = \mathcal{D}_0 + 2\mathcal{D}_1 + 4\mathcal{D}_2 + 6\mathcal{D}_3 + 5\mathcal{D}_4 + 4\mathcal{D}_5 + 3\mathcal{D}_6 + 2\mathcal{D}_7 + 3\mathcal{D}_8$}

	%Bibliography
	\bibliographystyle{amsalpha}
	\bibliography{biblio}
	
\end{document}